\documentclass[UTF-8,reqno]{amsart}
\usepackage{enumerate}
\usepackage{mhequ}
\usepackage[margin=1.2in]{geometry}
\usepackage{amssymb,url,color, booktabs,nccmath}

\usepackage{mathrsfs}
\usepackage{enumitem}
\usepackage{graphicx}
\usepackage{dirtytalk} 
\usepackage{tikz}
\usepackage{bbm}
\usetikzlibrary{shapes,snakes}
\usetikzlibrary{calc}
\usetikzlibrary{decorations.shapes}

\usepackage{color}
\usepackage[colorlinks=true]{hyperref}
\hypersetup{
    linkcolor=blue,         
    citecolor=red,       
    filecolor=blue,      
    urlcolor=cyan
}

\definecolor{darkergreen}{rgb}{0.0, 0.5, 0.0}

\numberwithin{equation}{section}

\newcommand{\be}{\begin{eqnarray}}
\newcommand{\ee}{\end{eqnarray}}
\newcommand{\ce}{\begin{eqnarray*}}
\newcommand{\de}{\end{eqnarray*}}
\newtheorem{theorem}{Theorem}[section]
\newtheorem{lemma}[theorem]{Lemma}
\newtheorem{remark}[theorem]{Remark}
\newtheorem{definition}[theorem]{Definition}
\newtheorem{proposition}[theorem]{Proposition}
\newtheorem{Examples}[theorem]{Example}
\newtheorem{corollary}[theorem]{Corollary}
\newtheorem{assumption}[theorem]{Assumption}

\newenvironment{nouppercase}{
  
  \renewcommand{\uppercasenonmath}[1]{}}{}

\def\[{{\Big[}}
\def\]{{\Big]}}
\def\<{{\langle}}
\def\>{{\rangle}}
\def\({{\Big(}}
\def\){{\Big)}}

\def\bx{{\mathbf{x}}}

\def\sgn{\mbox{\rm sgn}}

\def\={&\!\!=\!\!&}

\def\1{{\mathbf{1}}}

\def\geq{\geqslant}
\def\leq{\leqslant}
\def\ge{\geqslant}
\def\le{\leqslant}

\def\[{{\Big[}}
\def\]{{\Big]}}
\def\<{{\langle}}
\def\>{{\rangle}}
\def\({{\Big(}}
\def\){{\Big)}}

\def\bx{{\mathbf{x}}}

\def\sgn{\mbox{\rm sgn}}

\def\={&\!\!=\!\!&}
\def\bt{\begin{theorem}}
\def\et{\end{theorem}}
\def\bl{\begin{lemma}}
\def\el{\end{lemma}}
\def\br{\begin{remark}}
\def\er{\end{remark}}
\def\bx{\begin{Examples}}
\def\ex{\end{Examples}}
\def\bd{\begin{definition}}
\def\ed{\end{definition}}
\def\bp{\begin{proposition}}
\def\ep{\end{proposition}}
\def\bc{\begin{corollary}}
\def\ec{\end{corollary}}

\def\geq{\geqslant}
\def\leq{\leqslant}
\def\ge{\geqslant}
\def\le{\leqslant}

\def\<{\langle} \def\>{\rangle}

\allowdisplaybreaks

\tikzset{
        dot/.style={circle,fill=black,inner sep=0pt, outer sep=0.7pt, minimum size=1mm},
        Phi/.style={white!40!red,thick,snake=coil,segment amplitude=0.6pt, segment length=2pt},
         Z/.style={black!40!green,thick,snake=coil,segment amplitude=0.6pt, segment length=2pt},
        C/.style={thick,black!20!blue},
          Cr/.style={thick,black!20!red},
            Cg/.style={thick,black!20!green},
       }

\begin{document}

\title[Ergodicity for Dean--Kawasaki Equation]{\LARGE Ergodicity for the Dean--Kawasaki Equation with Dirichlet Boundary Conditions: Taming the Square-Root}

\author[Shyam Popat]{\large Shyam Popat}
\address[S. Popat]{Centre de Math\'ematiques Appliqu\'ees (CMAP), \'Ecole Polytechnique, 91120 Palaiseau, France}
\email{shyam.popat@polytechnique.edu}

\author[Zhengyan Wu]{\large Zhengyan Wu}
\address[Z. Wu]{Department of Mathematics, Technische Universit\"at M\"unchen, Boltzmannstr. 3, 85748 Garching, Germany}
\email{wuzh@cit.tum.de}

\begin{abstract}
In this paper, we establish the ergodicity of generalized Dean--Kawasaki equations with correlated noise and Dirichlet boundary conditions. In contrast to the ergodicity results of Fehrman, Gess, and Gvalani \cite{fehrman2022ergodicity}, our analysis accommodates irregular, square-root type noise coefficients. For such irregular coefficients, we prove that the law of the classical Dean--Kawasaki equation converges exponentially fast to equilibrium, while for the porous medium type Dean--Kawasaki equation, the convergence occurs at a polynomial rate. Furthermore, we obtain a regularization by noise effect, showing that the polynomial convergence rate improves to an exponential one whenever the noise coefficient is sufficiently regular, including the case of conservative multiplicative linear noise. Our approach relies on establishing a supercontraction property in a suitably weighted Lebesgue space, achieved through a refined doubling of variables argument. The construction of the weight function crucially exploits the specific structure of the Dean--Kawasaki-type correlated noise. 
\end{abstract}


\date{\today}

\begin{nouppercase}
\maketitle
\end{nouppercase}

\setcounter{tocdepth}{1}
\tableofcontents

\section{Introduction}

In this paper, we investigate the long-time behaviour of the generalised Dean--Kawasaki equation with correlated noise
\begin{align}\label{SPDE-0}
\begin{cases}
    \partial_t\rho=\Delta\Phi(\rho)-\nabla\cdot(\sigma(\rho)\circ\dot{\xi}^F+\nu(\rho)),& (x,t)\in U\times(0,T],\\
	\rho(x,0)=\rho_0(x),& (x,t)\in U\times\{t=0\},\\
	\Phi(\rho(x,t))=\rho_b(x),& (x,t)\in\partial U\times(0,T]. 
\end{cases}
\end{align}
Here $U$ denotes a $C^2$-regular bounded domain in $\mathbb{R}^d$ and $T>0$ denotes a terminal time.
The Stratonovich noise $\circ\dot{\xi}^F$ is white in time and correlated in space.
The assumptions on the non-linear functions $\Phi, \sigma$ and $\nu$ allow us to consider a wide range of non-linear SPDE, including the classical Dean--Kawasaki equation with correlated noise 
\begin{equation}\label{SPDE-1}
\partial_t\rho=\Delta\rho-\nabla\cdot(\sqrt{\rho}\circ\dot{\xi}^F+\nu(\rho)). 	
\end{equation}
The irregularity of solutions to \eqref{SPDE-0} means that we can only express the boundary condition $\rho_b$ indirectly via $\Phi(\rho)$.
This is done using the trace theorem, see Chapter 5.5 of Evans \cite{evans2022partial}.
Our results apply to the class of boundary data $\rho_b$ for which well-posedness of equation \eqref{SPDE-0} holds, see the paper of the first author \cite{Shyam25}. 
This includes all non-negative constant functions including zero and all smooth functions bounded away from zero.

The Dean--Kawasaki equation was first introduced by Dean \cite{D96} and Kawasaki \cite{K98} as a continuum stochastic PDE formally satisfied by the empirical density of an interacting diffusion system. The characteristic conservative square-root type noise is designed to statistically capture the fluctuations of the underlying particle system. From the perspective of fluctuating hydrodynamics, the Dean--Kawasaki equation can also be interpreted as possessing a gradient-flow type structure, with the stochastic term governed by the fluctuation-dissipation relation. In the case of a bounded domain, it is natural to impose Dirichlet boundary conditions on the density. Physically, this corresponds to coupling the particle system to reservoirs, sources, or sinks at the boundary, thereby enforcing a prescribed density (or chemical potential) at the walls. At the microscopic level, such boundary conditions arise, for instance, when particles are injected or removed at specified rates, or when the boundary acts as a particle bath maintained at fixed concentration. 

To the best of authors' knowledge, the ergodicity of Dean--Kawasaki type equations has previously been investigated by Fehrman, Gess and Gvalani \cite{fehrman2022ergodicity}, where ergodicity was established for the Dean--Kawasaki equation on the torus,  under It\^o-type noise and with a regularized diffusion coefficient. The authors of \cite{fehrman2022ergodicity} employed techniques from the theory of random dynamical systems, combined with gradient estimates, to demonstrate that the solution to the regularized Dean--Kawasaki equation generates a random dynamical system and admits a unique invariant measure. However, their approach encounters limitations when dealing with the irregular square-root noise coefficient. In contrast, the present work adopts a fundamentally different approach to prove ergodicity, specifically designed to handle irregular diffusion coefficients, including the particularly challenging square-root case. It is important to note, however, that the validity of this method is currently restricted to equations with Dirichlet boundary conditions. This distinction highlights both the novelty and the limitations of the approach in comparison to existing results.

\subsection{Main results} \label{subsec: main results}
The main results of this paper consist of two parts: a qualitative ergodicity result for \eqref{SPDE-0}, and quantitative estimates on the convergence rate of the laws of \eqref{SPDE-0} towards equilibrium. For the reader's convenience, the precise assumptions on the coefficients $\Phi$, $\nu$, and $\sigma$ will be specified later in Section \ref{subsec: assumptions}. However, we emphasize that our framework for $\Phi$ and $\sigma$ includes the typical case $\Phi(\xi)=\xi^m$ and $\sigma(\xi)=\xi^{\frac{m}{2}}$ for $\xi\ge0$ and any $m\ge1$. This particular choice allows \eqref{SPDE-0} to be interpreted as the fluctuating hydrodynamics of the zero-range process with polynomial jump rates. In the special case $m=1$, the equation \eqref{SPDE-0} reduces to the classical Dean--Kawasaki equation introduced by Dean \cite{D96}, which describes the fluctuating mean-field limit of Brownian particles. 

We will analyse the ergodicity of the system in a weighted Lebesgue space. Specifically, let $w: U \to \mathbb{R}$ be a non-negative integrable function. For $p\in\mathbb{N}$ we define $L^p_{w;x}$ to be the space of $p$'th power integrable functions on $U$ with respect to the measure $w(x)\,dx$. The precise construction of the weight function will be presented later. Under this weighted framework, the ergodicity result can be stated as follows.

\begin{theorem}
Suppose that $(\Phi,\sigma)$ satisfy either Assumption \ref{asm: classic-DK} or \ref{asm: assumptions on non-linear coeffs for contraction estime}, and that $(\Phi,\nu)$ further satisfy Assumptions \ref{asm: Assumptions on nu} and \ref{asm: assumption on curly A}. Then there exists a unique invariant measure $\mu \in \mathcal{P}(L^1_x)$, where $\mathcal{P}(L^1_{x})$ denotes the space of probability measures on $L^1_{x}$.

\end{theorem}

\begin{theorem}
Suppose that $(\Phi,\sigma,\nu)$ satisfy Assumptions~\ref{asm: classic-DK}, \ref{asm: Assumptions on nu}, and~\ref{asm: assumption on curly A}. These assumptions, in particular, cover the case $(\Phi(\xi),\sigma(\xi))=(\xi,\sqrt{\xi})$. Let $(P_t)_{t\ge0}$ denote the semigroup generated by \eqref{SPDE-0}, and let $\mu$ be its invariant measure. Then there exist a constants $c_1,c_2\in(0,\infty)$ independent of $t$ such that, for every $t\ge0$,  
\begin{equation*}
    \sup_{\rho_0\in L^2_{x}}\sup_{\|F\|_{\mathrm{Lip}(L^1_{w;x})}\le1}
    \left|P_t F(\rho_0)-\int_{L^1_{w;x}}F(\tilde\rho)\,\mu(d\tilde\rho)\right|
    \le c_1 e^{-c_2 t},
\end{equation*}
where $\mathrm{Lip}(L^1_{w;x})$ denotes the space of Lipschitz continuous functions from $L^1_{w;x}$ to $\mathbb{R}$.

On the other hand, suppose that $(\Phi,\sigma,\nu)$ satisfy Assumptions~\ref{asm: assumptions on non-linear coeffs for contraction estime}, \ref{asm: Assumptions on nu}, and~\ref{asm: assumption on curly A}. These assumptions, in particular, cover the case $(\Phi(\xi),\sigma(\xi))=(\xi^m,\xi^{\frac{m}{2}})$ for any $m>1$. Then there exists a constant $c\in(0,\infty)$ independent of $t$ such that, for every $t\ge0$,  
\begin{equation*}
    \sup_{\rho_0\in L^2_{x}}\sup_{\|F\|_{\mathrm{Lip}(L^1_{w;x})}\le1}
    \left|P_t F(\rho_0)-\int_{L^1_{w;x}}F(\tilde\rho)\,\mu(d\tilde\rho)\right|
    \le c\,t^{-\frac{1}{q_0}},
\end{equation*}
where $q_0$ is the exponent associated with Assumption~\ref{asm: assumption on curly A}. 
\end{theorem}
The results on the convergence rate exhibit a clear dichotomy. For the classical Dean--Kawasaki equation, corresponding to $(\Phi(\xi),\sigma(\xi))=(\xi,\sqrt{\xi})$, the laws of \eqref{SPDE-0} converge exponentially fast to equilibrium. In contrast, for the more general porous medium case $(\Phi(\xi),\sigma(\xi))=(\xi^m,\xi^{\frac{m}{2}})$ with $m>1$, the convergence occurs at a polynomial rate.

The distinction between exponential and polynomial convergence to equilibrium stems from the qualitative difference in the dissipative structures of the operators $\Delta \rho$ and $\Delta \rho^m = m \nabla \cdot (\rho^{m-1} \nabla \rho)$, $m>1$. Intuitively, the Laplacian induces a uniform smoothing effect across the domain, which gives rise to a spectral gap in the associated generator (see, for instance, Chen et al. \cite{CKNP21}) and, consequently, may lead to exponential ergodicity. In contrast, in the porous medium case, the effective diffusion coefficient $D(\rho) = m \rho^{m-1}$ degenerates in low-density regions, resulting in slower particle diffusivity and a weaker dissipation mechanism. From an analytical perspective, this degeneracy destroys the spectral gap, and the long-time dynamics are instead captured by alternative techniques, such as those developed in Dareiotis, Gess and Tsatsoulis \cite{DGT20}. In the present work, rather than relying on spectral gap arguments, we analyze the trajectories directly. A key technical tool for distinguishing the corresponding convergence rates is provided later in Lemma~\ref{lemma: lemma B.2 from DGT}.

With the aid of the weight function, the Stratonovich-to-It\^o correction terms arising from the noise structure tend to contribute additional dissipation. However, as will be seen in the proofs of the above results, this contribution will be neglected due to the potential irregularity of the noise coefficient, which renders integrals involving $\log \rho$-type functions ill-defined. Consequently, utilizing only the dissipation from the porous medium term $\Delta \Phi(\rho)$ yields convergence merely at a polynomial rate, consistent with the results in Dareiotis, Gess and Tsatsoulis \cite{DGT20}. To analyse the contribution of the conservative noise, we establish the following corollary. We show that when the noise coefficient is sufficiently regular,  including the linear case $\sigma(\xi)=\xi$, the polynomial convergence rate of the porous medium type Dean--Kawasaki equation can be improved to an exponential rate. We interpret this enhancement as a manifestation of a regularization-by-noise phenomenon. 
\begin{corollary}[Regularization by noise]\label{coro: L1 omega super contraction-intro}
Suppose that $(\Phi,\sigma)$ satisfy Assumptions~\ref{asm: assumptions on non-linear coeffs for contraction estime}, \ref{asm: Assumptions on nu}, and~\ref{asm: assumption on curly A}. Let 
\begin{equation*}
    \Psi_{\sigma}(\xi)=\int^{\xi}_0[\sigma'(\xi')]^2d\xi',\quad \xi\geq 0.
\end{equation*}
Suppose additionally that $\Psi_{\sigma}$ is well-defined and satisfies the lower Lipschitz bound 
\begin{align}\label{eq: assumption on linear growth of Psi sigma}
c|\xi_1-\xi_2|\leq|\Psi_{\sigma}(\xi_1)-\Psi_{\sigma}(\xi_2)|, 
\end{align}
for some $c>0$, uniformly for every $\xi_1,\xi_2\geq0$. Let $(P_t)_{t\ge0}$ denote the semigroup generated by \eqref{SPDE-0}, and let $\mu$ be its invariant measure. Then there exist constants $c_1,c_2\in(0,\infty)$ independent of $t$ such that, for every $t\ge0$,  
\begin{equation*}
    \sup_{\rho_0\in L^2_{x}}\sup_{\|F\|_{\mathrm{Lip}(L^1_{w;x})}\le1}
    \left|P_t F(\rho_0)-\int_{L^1_{w;x}}F(\tilde\rho)\,\mu(d\tilde\rho)\right|
    \le c_1 e^{-c_2 t}. 
\end{equation*}
\end{corollary}
The additional assumption on the coefficient $\sigma$ ensures that the dissipative contribution arising from the Stratonovich-to-It\^o correction term is well-justified. A representative example is $\sigma(\xi) = \xi$, which corresponds to the fluctuating zero-range process characterized by $\Phi(\xi) = \xi^m$ and $\sigma(\xi) = \xi^{m/2}$ with $m = 2$. Moreover, the coefficient $\sigma$ can also be chosen in a more general or artificial manner, for instance, $\sigma(\xi) = \xi + \frac{1}{2}\sin \xi$ or $\sigma(\xi) = \xi+e^{-\xi}$.

\subsection{Key ideas and technical comments}\label{subsec: key ideas and technical comments}
The ergodicity argument is inspired by the approaches of Dareiotis, Gess and Tsatsoulis \cite{DGT20} and Dong, Zhang and Zhang \cite{DZZ23} and further developed by combining it with the doubling variable method in the kinetic formulation framework, as employed by Fehrman and Gess \cite{FG24} and the first author \cite{Shyam25}. More precisely, we construct a weight function $w: U \to \mathbb{R}_+$, tailored to the specific structure of the Dean--Kawasaki equation, and analyse ergodicity within the weighted Lebesgue space $L^1_{w;x}$.

The fundamental step in establishing ergodicity is to prove a super-contraction property in $L^1_{w;x}$.
Specifically, let $\rho_0 \in L^2_\omega L^2_{x}$, and denote by $\rho(\cdot;\rho_0):[0,T]\to U$ the renormalized kinetic solution of \eqref{SPDE-0} with initial condition $\rho_0$ in the sense of Definition \ref{def: stochastic kinetic solution of DK} below. Under the assumptions in Section~\ref{subsec: assumptions}, we show that there exists a constant $c \in (0,\infty)$, depending only on $d, U, q_0$, and the weight $w$ (but notably independent of $t$), such that for every $t \in [0,T]$,
\begin{equation}\label{super-contraction-intro}
    \sup_{\rho_{1,0},\rho_{2,0} \in L^2_\omega L^2_{x}} 
    \mathbb{E}\big\|\rho(t;\rho_{1,0})-\rho(t;\rho_{2,0})\big\|_{L^1_{w;x}}
    \leq c\, \mathrm{Decay}(t),  
\end{equation}
where $\mathrm{Decay}(t)$ denotes the temporal decay rate, which will be specified later. To achieve this, the weight function $w$ is constructed to satisfy, for every $x \in U$,
\begin{align}\label{weight-conds}
 \begin{cases}
        \Delta w <0 \\
        F_1\Delta w+c|F_2\cdot\nabla w|<0,\\
        \partial_{x_i}w(x)<0\quad \text{for every } i=1,\dots,d,        
    \end{cases}
\end{align}
where $F_1$ and $F_2$ depend on the specific structure of the noise $\xi^F$ which will be specified in Section~\ref{subsec: setup and kinetic formulation} and for constant $c\in(0,\infty)$ that will be specified.

{\bf The super-contraction property.}
The super-contraction property is established through the kinetic formulation, which characterizes the evolution of the renormalized kinetic function $\chi = \mathbbm{1}_{\{0 < \xi < \rho(t; \rho_0)\}}$. The proof employs the doubling of variables method utilized in Fehrman and Gess \cite{FG24}, combined with a weight function $w$. As a representative example, we take
$$
\Phi(\rho) = \rho, \quad \text{and} \quad \sigma(\rho) = \sqrt{\rho}.
$$
Let $\rho_1 = \rho(\cdot; \rho_{1,0})$ and $\rho_2 = \rho(\cdot; \rho_{2,0})$ denote two stochastic kinetic solutions of \eqref{SPDE-0} with initial conditions $\rho_{1,0}$ and $\rho_{2,0}$ respectively, and $\chi_1 = \mathbbm{1}_{\{0 < \xi < \rho_1\}}$, $\chi_2 = \mathbbm{1}_{\{0 < \xi < \rho_2\}}$ denote the corresponding kinetic functions. Using the elementary property of indicator functions, one observes that
\begin{align*}
\int_U |\rho_1 - \rho_2| \omega(x)
&= \int_U \left( \int_{\mathbb{R}} |\chi_1 - \chi_2|^2 \, d\xi \right) \omega(x) \, dx \\
&= \int_U \left( \int_{\mathbb{R}} \chi_1 + \chi_2 - 2\chi_1 \chi_2 \, d\xi \right) \omega(x) \, dx.
\end{align*}
Taking the kinetic formulation into account and applying It\^o's product rule, one can expand the expression $\chi_1 + \chi_2 - 2\chi_1 \chi_2$ analogously to Fehrman and Gess \cite{FG24}. However, in contrast to Fehrman and Gess \cite{FG24} and the first author \cite{Shyam25}, the presence of the weight function $w$ leads to terms whose sign can be quantified. 
In the case of the classical Dean--Kawasaki equation $(\Phi(\xi),\sigma(\xi))=(\xi,\sqrt{\xi})$ the contributions read
\begin{align*}
I_t^{wgt}
&= \int_0^t \int_U \left( \nabla \rho_1 - \nabla \rho_2 + \frac{1}{8} F_1 \rho_1^{-1} \nabla \rho_1 - \frac{1}{8} F_1 \rho_2^{-1} \nabla \rho_2 \right) \cdot \nabla w(x) \, \sgn(\rho_2 - \rho_1) \, dx \, ds,
\end{align*}
and
\begin{align}\label{negative-contrib-3}
I_t^{flux}
&= \int_0^t \int_U w(x) \left( \nabla \cdot \nu(\rho_1) - \nabla \cdot \nu(\rho_2) \right) \sgn(\rho_2 - \rho_1) \, dx \, ds \nonumber\\
&= \sum_{i=1}^d \int_0^t \int_U \partial_{x_i} w(x) \, |\nu_i(\rho_1) - \nu_i(\rho_2)| \, dx \, ds.
\end{align}
Here $I_t^{wgt}$ arises from the diffusion and the Stratonovich-to-It\^o correction, while $I_t^{flux}$ originates from the flux term $-\nabla \cdot (\nu(\rho))$.

Concerning the term $I_t^{wgt}$, the monotonicity of the identity function implies that the first two terms can be formally rewritten as
\begin{align}\label{negative-contrib-1}
\int_0^t \int_U (\nabla \rho_1 - \nabla \rho_2) \cdot \nabla w(x) \, \sgn(\rho_2 - \rho_1) \, dx \, ds
&= -\int_0^t \int_U \nabla |\rho_1 - \rho_2| \cdot \nabla w(x) \, dx \, ds \notag\\
&= \int_0^t \int_U |\rho_1 - \rho_2| \, \Delta w(x) \, dx \, ds.
\end{align}
To ensure that this contribution is negative, we require
\begin{equation}\label{weight-condition-1}
(\Delta w)(x) < 0, \quad \forall x \in U.
\end{equation}
A similar argument applies to the logarithmic term, since $\xi \mapsto \frac{1}{8} F_1 \log \xi$ is increasing. Thus,
\begin{align}\label{negative-contrib-2}
\frac{1}{8}\int_0^t \int_U F_1 (\nabla \log \rho_1 - \nabla \log \rho_2) \cdot \nabla w(x) \, &\sgn(\rho_2 - \rho_1) \, dx \, ds\notag\\
=& -\frac{1}{8}\int_0^t \int_U F_1 \nabla |\log \rho_1 - \log \rho_2| \cdot \nabla w(x) \, dx \, ds \notag\\
=& \frac{1}{8}\int_0^t \int_U |\log \rho_1 - \log \rho_2| \, \nabla \cdot (F_1 \nabla w(x)) \, dx \, ds.
\end{align}
This yields the additional condition
\begin{equation}\label{weight-condition-2}
\nabla \cdot (F_1 \nabla w(x)) < 0.
\end{equation}
We remark that, due to the potential lack of integrability of $\log \rho_1$ and $\log \rho_2$, the negative contribution in \eqref{negative-contrib-2} will be dropped in the rigorous proof when the truncation near the zero and infinity values of the solution are still present, ensuring that the integral is well-defined.
Furthermore, to obtain a negative contribution from the flux term $I_t^{flux}$, we impose
\begin{equation}\label{weight-condition-3}
\partial_{x_i} w(x) < 0, \quad \text{for every } i = 1, \dots, d.
\end{equation}
Combining \eqref{weight-condition-1}, \eqref{weight-condition-2}, and \eqref{weight-condition-3}, we obtain the set of conditions \eqref{weight-conds} for the admissible weight function $w$.

In summary, compared with Fehrman and Gess \cite{FG24}, the presence of the weight function yields additional negative contributions \eqref{negative-contrib-1} and \eqref{negative-contrib-3} in the weighted $L^1_x$-contraction property. We will utilize these results to derive bounds on the difference between two solutions corresponding to distinct initial data, with the bounds exhibiting decay in time.

{\bf Construction of the weight function.}
Let $e=\frac{1}{\sqrt{d}}(1,\ldots,1)$ be a $d$-dimensional unit vector. We construct the weight function as  
\begin{align}\label{eq: weight function intro}
w(x) = -\exp(\alpha\, x\cdot e) + C,
\end{align}
where $\alpha>0$ is chosen sufficiently large depending on the noise structure $F_1$, and $C>0$ is a constant ensuring the positivity of $w$. We refer the reader to Proposition~\ref{prop: choice of weight function in informal proof} for further details and for the verification of conditions~\eqref{weight-conds}.

Since $U$ is bounded, $\exp(\alpha\, x\cdot e)$ admits both a positive lower bound and a finite upper bound. Hence, by choosing $C>0$ appropriately, we can guarantee that 
$$
0 < c_{\mathrm{low}} \leq w(x) \leq c_{\mathrm{up}} < \infty, \quad \text{for every } x \in U.
$$
Consequently,
\begin{align}\label{eq: equivalence between l1 norm and weighted l1 norm}
c_{\mathrm{low}} \|f\|_{L^1_x} \leq \|f\|_{L^1_{w;x}} \leq c_{\mathrm{up}} \|f\|_{L^1_x}, \quad \text{for all } f \in L^1_x,
\end{align}
which shows the equivalence between the standard $L^1_x$-norm and the weighted $L^1_{w;x}$-norm.

Furthermore, we emphasize that there does not exist an admissible weight function on either the torus $\mathbb{T}^d$ or 
the whole space $\mathbb{R}^d$.
We take the case of $\mathbb{T}^d$ for 
illustration.
Since a function on $\mathbb{T}^d$ cannot be strictly decreasing, 
condition \eqref{weight-condition-3} cannot be satisfied. Moreover, integrating 
over space yields
$$
\int_{\mathbb{T}^d} \Delta w(x)\,dx
= \int_{\mathbb{T}^d} \nabla\cdot(F_1\nabla w(x))\,dx
= 0.
$$
This contradicts \eqref{weight-condition-1} and \eqref{weight-condition-2}. 
A similar argument applies to the case of $\mathbb{R}^d$. 
This observation indicates that the present approach is only valid 
for bounded domains.

{\bf Ergodicity.}
Based on the choice of the weight function, in the classical Dean--Kawasaki case we obtain the existence of constants $ c,q_0\in(0,\infty)$ such that 
\begin{align}\label{contraction-1-intro}
c_{\mathrm{low}}\mathbb{E}\|\rho(t;\rho_{1,0})-\rho(t;\rho_{2,0})\|_{L^1_{w;x}}
\leq c_{\mathrm{up}}\mathbb{E}\|\rho_{1,0}-\rho_{2,0}\|_{L^1_{w;x}}
-& c\mathbb{E}\int_0^t\|\rho(r;\rho_{1,0})-\rho(r;\rho_{2,0})\|_{L^1_{w;x}}\,dr\notag\\
-& c\mathbb{E}\int_0^t\|\rho(r;\rho_{1,0})-\rho(r;\rho_{2,0})\|^{q_0+1}_{L^{q_0+1}_{w;x}}\,dr.
\end{align}
We note that the negative contribution 
$\mathbb{E}\int_0^t\|\rho(r;\rho_{1,0})-\rho(r;\rho_{2,0})\|_{L^1_{w;x}}\,dr$
arises from the specific choice $\Phi(\rho)=\rho$, and it dominates the right-hand side of \eqref{contraction-1-intro} whenever $|\rho(r;\rho_{1,0})-\rho(r;\rho_{2,0})|<1$. By neglecting the last term 
$\mathbb{E}\int_0^t\|\rho(r;\rho_{1,0})-\rho(r;\rho_{2,0})\|^{q_0+1}_{L^{q_0+1}_{w;x}}\,dr$ 
and applying Gronwall's lemma, we deduce
\begin{align}\label{decayintime-1-intro}
\mathbb{E}\|\rho(t;\rho_{1,0})-\rho(t;\rho_{2,0})\|_{L^1_{w;x}}
\leq c_1 e^{-c_2t}.
\end{align}

In the porous medium Dean--Kawasaki case $(\Phi(\xi),\sigma(\xi))=(\xi^m,\xi^{\frac{m}{2}})$ for $m>1$, the contraction property \eqref{contraction-1-intro} is replaced by
\begin{align}\label{contraction-2-intro}
c_{\mathrm{low}}\mathbb{E}\|\rho(t;\rho_{1,0})-\rho(t;\rho_{2,0})\|_{L^1_{w;x}}
\leq c_{\mathrm{up}}\mathbb{E}\|\rho_{1,0}-\rho_{2,0}\|_{L^1_{w;x}}
-& c\mathbb{E}\int_0^t\|\rho(r;\rho_{1,0})-\rho(r;\rho_{2,0})\|^m_{L^m_{w;x}}\,dr\notag\\
-& c\mathbb{E}\int_0^t\|\rho(r;\rho_{1,0})-\rho(r;\rho_{2,0})\|^{q_0+1}_{L^{q_0+1}_{w;x}}\,dr,
\end{align}
from which one obtains a polynomial decay rate in time instead of the exponential decay in \eqref{decayintime-1-intro},
\begin{align}\label{decayintime-2-intro}
\mathbb{E}\|\rho(t;\rho_{1,0})-\rho(t;\rho_{2,0})\|_{L^1_{w;x}}
\leq ct^{-1/q_0}.
\end{align}
This follows by applying a technical result stated in Lemma~\ref{lemma: lemma B.2 from DGT}.

Considering the cases \eqref{decayintime-1-intro} and \eqref{decayintime-2-intro}, we conclude the decay rate $\mathrm{Decay}(t)$ as stated in the super-contraction property \eqref{super-contraction-intro}. 
Using \eqref{super-contraction-intro}, we establish ergodicity by proving the existence and uniqueness of invariant measures. 
Furthermore, we conclude that the convergence of the laws of the classical Dean--Kawasaki equation 
$(\Phi(\xi),\sigma(\xi))=(\xi,\sqrt{\xi})$ is exponentially fast, 
whereas in more general porous medium cases, the convergence occurs at a polynomial rate. 

To improve the convergence rate toward equilibrium, we aim to exploit the contribution of the noise structure. Specifically, when the coefficient $\sigma$ is sufficiently regular, the term corresponding to $\log \rho_1 - \log \rho_2$ in \eqref{negative-contrib-2} is actually the well-defined expression
\begin{align*}
\bigl|\Psi_{\sigma}(\rho_1)-\Psi_{\sigma}(\rho_2)\bigr|
	= \left|\int_{\rho_2}^{\rho_1} [\sigma'(\xi')]^2 \, d\xi'\right|.
\end{align*}
Under the growth assumption \eqref{eq: assumption on linear growth of Psi sigma} on $\Psi_\sigma$,
the corresponding term in \eqref{negative-contrib-2} is well-defined and provides an additional dissipative effect. By exploiting this extra negative contribution and applying Gronwall's inequality, we obtain exponential convergence toward equilibrium. This leads to an enhanced convergence rate for the porous medium type Dean--Kawasaki equation, yielding the regularization-by-noise result stated in Corollary~\ref{coro: L1 omega super contraction-intro}.

\subsection{Comments on the literature}\label{subsec: comments on literature}

\ 

{\bf Fluctuating hydrodynamics. }
The study of Dean--Kawasaki-type equations is closely connected to the theories of fluctuating hydrodynamics (FHD) and macroscopic fluctuation theory (MFT). MFT provides a unified framework for describing systems driven far from equilibrium, extending and refining the classical linear response theories valid only near equilibrium, see Bertini et al. \cite{BDGJL}, Derrida \cite{Derrida}.
At its core lies an ansatz formulated in terms of large deviation principles for interacting particle systems. Parallel to this perspective, the FHD approach aims to model microscopic fluctuations in accordance with the principles of statistical mechanics and non-equilibrium thermodynamics. Within this framework, one postulates conservative SPDEs that capture the essential fluctuation features of non-equilibrium systems, see Landau and Lifshitz \cite{LL87} and Spohn \cite{HS}. The central ansatz of MFT can then be viewed as the zero-noise large deviation limit of such conservative SPDEs. A typical example within FHD is the Dean--Kawasaki equation, which effectively characterizes the fluctuation behaviour of mean-field interacting particle systems.

{\bf Kinetic formulation.}
The kinetic formulation was first developed in the partial differential equation (PDE) literature and builds off the idea of entropy solutions. 
To the best of the authors' knowledge the formulation was introduced in the mid-nineties by Lions, Perthame and Tadmor \cite{lions1994kinetic} in the context of scalar conservation laws.
Relevant extensions to  degenerate parabolic-hyperbolic PDEs were given by Chen and Pertame \cite{chen2003well}, Bendahmane and
Karlsen \cite{bendahmane2004renormalized}, and Karlsen and Riseboro \cite{karlsen2000uniqueness}.

The first work extending the kinetic solution framework to the stochastic setting was in 2010 by Debussche and Vovelle \cite{DV10}, who extended the deterministic results of Lions, Perthame and Tadmor \cite{lions1994kinetic} to the case of scalar conservation laws with stochastic forcing.
The results were subsequently extended by Hofmanov\'a \cite{hofmanova2013degenerate} to the degenerate parabolic case.

{\bf The Dean--Kawasaki equations and conservative SPDEs. }
The Dean--Kawasaki equation was first introduced by Dean \cite{D96} and Kawasaki \cite{K98}. However, the model remained at a formal level until Konarovskyi, Lehmann and von Renesse \cite{KLR19}  established a rigorous martingale framework, thereby providing a mathematically justified formulation. Subsequently, Konarovskyi, Lehmann and von Renesse \cite{KLR20} extended the well-posedness result to the case of nonlocal interactions by employing a Girsanov transform.
Further studies such as Konarovskyi and M\"uller \cite{KF24}, Dello Schiavo and Konarovskyi \cite{DK25} and Dello Schiavo \cite{D25} investigated additional properties of the Dean--Kawasaki equation driven by space-time white noise. More recently, M\"uller, von Renesse and Zimmer \cite{FMJ25} extended the martingale approach to the Vlasov-Fokker-Planck-type Dean--Kawasaki equation.  

Motivated by the theory of fluctuating hydrodynamics, the Dean--Kawasaki equation driven by spatially correlated noise has been extensively studied. The inclusion of correlated noise reflects more realistic mesoscopic features and reveals additional analytical structures of the model. To the authors' understanding, these works rely heavily on techniques developed for stochastic conservation laws. Debussche and Vovelle \cite{DV10} studied the Cauchy problem for stochastic conservation laws through the concept of kinetic solutions. This framework was later extended to parabolic-hyperbolic stochastic PDEs with conservative noise by Gess and Souganidis \cite{GS17}, Fehrman and Gess \cite{FG19}, and Dareiotis and Gess \cite{DG20}. These methodologies have since been adapted to the analysis of Dean--Kawasaki-type equations. In the case of local interactions, Fehrman and Gess \cite{FG24} established the well-posedness of function-valued solutions to the Dean--Kawasaki equation with correlated noise. Building on this framework, they further investigated small-noise large deviations in Fehrman and Gess \cite{FG23}, aiming to rigorously connect FHD with MFT.  

More recently, Clini and Fehrman \cite{CF23} proved a central limit theorem for the nonlinear Dean--Kawasaki equation with correlated noise, while Gess, the second author, and Zhang \cite{GWZ24} derived higher-order fluctuation expansions for Dean--Kawasaki-type systems.
For models with Dirichlet boundary conditions, the first author established well-posedness, large deviation, and fluctuation results in \cite{Shyam25,Shyam25-fluc}. For models with nonlocal interactions, Wang, the second author, and Zhang \cite{WWZ22}, as well as the second author and Zhang \cite{WZ24}, analysed the well-posedness and large deviations for Dean--Kawasaki equations with singular interactions, with applications to the fluctuating Ising-Kac-Kawasaki model. Later on, the second author \cite{W25} investigated the nonlinear fluctuations of the fluctuating Ising--Kac--Kawasaki model.  For Keller--Segel-type interactions, Martini and Mayorcas \cite{AA25,AA24} established well-posedness and large deviation principles for an additive-noise approximation of the Keller-Segel-type Dean--Kawasaki equation. Additionally, Hao, the second author, and Zimmer \cite{HWJ25} established a strong well-posedness result for the Vlasov-Fokker-Planck-type Dean--Kawasaki equation with correlated noise.

Other well-posedness results for the Dean--Kawasaki equation with correlated noise include the recent works of Fehrman and Gess \cite{fehrman2024conservative} who consider the Dean--Kawasaki equation on the whole space $\mathbb{R}^d\times(0,\infty)$, and the work of Fehrman \cite{fehrman2025stochastic} who considers the Dean--Kawasaki equation on bounded domains with homogeneous Neumann boundary data.
Finally, we mention that under additional mollification of the square root, Djurdjevac, Kremp and Perkowski \cite{djurdjevac2022weak} and Djurdjevac, Ji and Perkowski \cite{djurdjevac2025weak} show the well-posedness of strong solutions.
They also make a connection with MFT by establishing a quantitative bound on the weak error between the approximating SPDE and the empirical measure of the corresponding particle system.

{\bf Ergodicity for stochastic conservation laws. }
We also mention related literature concerning ergodicity and long-time behavior of stochastic conservation laws and conservative SPDEs. Debussche and Vovelle \cite{DV15} established ergodicity for stochastic conservation laws on the torus driven by additive noise. Subsequently Gess and Souganidis \cite{GS17} applied an analogue of the averaging lemma, combined with the specific Gaussian structure of the noise, to investigate the regularity and long-time behaviour of conservation law equations with stochastic fluxes. This framework was later extended by Gess and Souganidis \cite{GS17-SPA} to conservative SPDEs of porous medium type. Concerning ergodicity for equations on bounded domains with Dirichlet boundary conditions, Dareiotis, Gess and Tsatsoulis \cite{DGT20} established ergodicity for the stochastic porous medium equation.
Building upon this, Dong, Zhang and Zhang \cite{DZZ23} employed the methodology developed by Dareiotis, Gess and Tsatsoulis \cite{DGT20} to prove ergodicity for stochastic conservation laws driven by multiplicative noise.

\subsection{Structure of the paper}\label{structure of the paper}
This paper is organized as follows. In Section \ref{sec: preliminaries section}, we introduce the basic notations and provide the precise definitions of the solution concepts.
Section \ref{sec: rigorous proof of super contraction subsection section} is devoted to establishing the super-contraction property. 
Finally, the ergodicity results are analysed in Section \ref{sec: ergodicity}.

\section{Preliminaries}\label{sec: preliminaries section}
We begin by introducing some of the notation that will be employed in the paper.
\subsection{Notation}\label{subsec: notation}
Let $T>0$. Let $(\Omega,\mathcal{F},\mathbb{P},\{\mathcal{F}_t\}_{t\in [0,T]},(\{\beta_k(t)\}_{t\geq0})_{k\in\mathbb{N}})$ be a stochastic basis. Without loss of generality, we assume that the filtration $\{\mathcal{F}_t\}_{t\in [0,T]}$ is complete and that $\{\beta_k(t)\}_{t\geq0}$, $k\in\mathbb{N}$, are independent $\{\mathcal{F}_t\}_{t\in [0,T]}$-Wiener processes taking values in $\mathbb{R}^d$. Expectation with respect to the probability measure $\mathbb{P}$ is denoted by $\mathbb{E}$.

Recall that $U$ is a $C^2$-regular, bounded domain in $\mathbb{R}^d$. For every $p \in [1, \infty]$, we denote by $\|\cdot\|_{L^p_x}$ the norm in the Lebesgue space $L^p(U)$. 
Let $C^\infty(U \times (0, \infty))$ be the space of infinitely differentiable functions on $U \times (0, \infty)$, and $C_c^\infty(U \times (0, \infty))$ its subspace consisting of compactly supported functions.
For a non-negative integer $a$, we write $H^a(U):= W^{a,2}(U)$ where $W^{a,2}(U)$ denotes the usual Sobolev space on the domain $U$ with integrability $2$.
The space $H^\alpha_0(U)$ is defined to be the closure of compactly supported smooth functions $C_c^\infty(U)$ in the space $H^a(U)$.
For every $p,q,r\in[1,\infty)$, we denote the space $L^p(\Omega;L^q([0,T];L^r(U)))$ using the shorthand $L^p_\omega L^q_t L^r_x$, where the subscripts denote the variables being integrated.
To emphasise the initial condition, we will sometimes denote solutions of equation \eqref{SPDE-0} at time $t\geq0$ by 
\begin{equation*}
    \rho(t;\rho_0):U\to\mathbb{R}.
\end{equation*}

\begin{definition}
(Weighted Lebesgue space) For a non-negative integrable weight $w:U\to\mathbb{R}$ and $p\in\mathbb{N}$, the weighted $L^p(U)$-space, denoted by $L^p_{w;x}$, consists of all measurable functions $f:U\to\mathbb{R}$ satisfying
\begin{equation*}
    \|f\|_{L^p_{w;x}}:=\int_U |f(x)|^pw(x)\,dx<\infty.
\end{equation*}
\end{definition}

We also define the function spaces 
\begin{equation*}
    B_b(L^1_{w;x}):=\{f:L^1_{w;x}\to\mathbb{R}: f\, \text{is bounded and measurable}\},
\end{equation*}
\begin{equation*}
    C_b(L^1_{w;x}):=\{f:L^1_{w;x}\to\mathbb{R}: f\, \text{is continuous, bounded and measurable}\}.
\end{equation*}

We use the notation $\lesssim$ to mean \say{less than or equal to, up to a constant}.
\subsection{Setup and kinetic formulation}\label{subsec: setup and kinetic formulation}
We begin by defining the noise in \eqref{SPDE-0}.
\begin{definition}[The noise {$\xi^F$}]\label{def: definition of the noise}
    For a sequence of continuously differentiable functions $F:=\{f_k : U\to\mathbb{R}\}_{k\in\mathbb{N}}$ and the independent Brownian motions $\{\beta_k:[0,T]\to\mathbb{R}^d\}_{k\in\mathbb{N}}$ introduced in Section \ref{subsec: notation}, the noise $\xi^F:U\times[0,T]\to\mathbb{R}^d$ is defined by 
\begin{equation*}
\xi^F(x,t):=\sum_{k=1}^{\infty} f_k(x)\beta_k(t).     
\end{equation*}
\end{definition}

By converting the equation \eqref{SPDE-0} into an equation with It\^o noise, we see that it is convenient to define the three quantities related to the spatial coefficients of the noise $F$,
\begin{align}
        &F_1:U\to\mathbb{R}
\quad \text{defined by} \quad F_1(x):=\sum_{k=1}^\infty f_k^2(x),\nonumber\\
&F_2:U\to\mathbb{R}^d
\quad \text{defined by} \quad F_2(x):=\sum_{k=1}^\infty f_k(x)\nabla f_k(x)=\frac{1}{2}\nabla F_1,\label{eq: def of F_2}\\
&F_3:U\to\mathbb{R}
\quad \text{defined by} \quad F_3(x):=\sum_{k=1}^\infty |\nabla f_k(x)|^2.\nonumber
\end{align}
We assume that the following assumption on the noise holds for every statement in the paper and will not repeatedly reference it.
\begin{assumption}\label{asm: assumption on spatial components of the noise}
    Assume that the spatial components of the noise $\{f_k\}_{k\in\mathbb{N}}$ are such that $\{F_i\}_{i=1,2,3}$ are continuous on $U$.
    We further assume that the noise is non-degenerate in $U$ in the sense that
    \begin{equation}\label{eq: Assumption that F1 is strictly positive}
        \inf_{x\in U}F_1(x)>0.
    \end{equation}
\end{assumption}

We consider the non-degeneracy condition \eqref{eq: Assumption that F1 is strictly positive} on the noise to be a very mild condition, and is satisfied if for example we assume that the first noise coefficient is non-degenerate $f_1=1$.

Using the Stratonovich-to-It\^o conversion, the generalised Dean--Kawasaki equation \eqref{SPDE-0} with It\^o noise is 
\begin{align*}
    \partial_t\rho=\Delta\Phi (\rho) -\nabla\cdot (\sigma(\rho) \dot{\xi}^F + \nu(\rho)) + \frac{1}{2}\nabla\cdot(F_1[\sigma'(\rho)]^2\nabla\rho + \sigma'(\rho)\sigma(\rho)F_2).
\end{align*}
As is well known by now, see for example Chapter 3 of \cite{FG24} or Section 5.4 of the work of the first author \cite{Shyam25-fluc}, the kinetic function 
\begin{equation}\label{eq: kinetic function}
    \chi(x,\xi,t)=\mathbbm{1}_{\{0<\xi<\rho(x,t)\}} 
\end{equation}
solves distributionally the equation
\begin{multline}\label{eq: distributional kinetic equation}
    \partial_t\chi=\nabla\cdot(\delta_0(\xi-\rho)\Phi'(\xi)\nabla\rho)+\frac{1}{2}\nabla\cdot(\delta_0(\xi-\rho)\left(F_1(\sigma'(\xi))^2\nabla\rho + \sigma(\xi)\sigma'(\xi)F_2\right))+\partial_\xi q\\
    -\frac{1}{2}\partial_\xi(\delta_0(\xi-\rho)\left(\sigma(\xi)\sigma'(\xi)\nabla\rho\cdot F_2+\sigma^2(\xi)F_3\right))-\delta_0(\xi-\rho)\nabla\cdot(\sigma(\rho) \dot{\xi}^F)
    -\delta_0(\xi-\rho)\nabla\cdot\nu(\rho),
\end{multline}
where $q$ denotes the kinetic measure, a non-negative, locally finite measure on $U\times(0,\infty)\times[0,T]$ satisfying that 
$$
q\geq\delta_0(\xi-\rho)\Phi'(\xi)|\nabla\rho|^2,
$$
and for every $\psi\in C_c^\infty(U\times(0,\infty))$,
    \begin{align*}
        (\omega,t)\in(\Omega,[0,T])\to\int_{\mathbb{R}}\int_0^t\int_U\psi(x,\xi)\, q(dx,d\xi,dt)(\omega)=:\int_{\mathbb{R}}\int_0^t\int_U\psi(x,\xi)\, dq(x,\xi,t)(\omega)
    \end{align*}
 \normalsize
    is $\{\mathcal{F}_t\}_{t\geq0}$-predictable.

Equation \eqref{eq: distributional kinetic equation} forms the basis of a kinetic solution, where the distributional equality is made rigorous by integrating against a test function.
To capture the boundary data of the solution we introduce the harmonic extension of $\rho_b$, where the boundary data $\rho_b$ is assumed to satisfy Assumption \ref{asm: assumption on initial data and boundary condition} below. 
That is, define $h_{\rho_b}$ to be the weak solution of
     \begin{equation}\label{eq: equation for PDE tilde h}
    \begin{cases}
            -\Delta h_{\rho_b} = 0 & \text{on } U,\\
    h_{\rho_b}=\rho_b& \text{on }\partial U.
    \end{cases}
\end{equation}
Well-posedness of $H^1(U)$-valued weak solutions of $h_{\rho_\beta}$ follows from the fact that the boundary $\partial U$ of our domain is assumed to be $C^2$-regular and the fact that $\rho_b$ is smooth.

\begin{definition}[Kinetic solution of the Dean--Kawasaki equation with correlated noise \eqref{SPDE-0}]\label{def: stochastic kinetic solution of DK}
    Let $\rho_0\in L^2_{\omega}L^2_x$ be a non-negative $\mathcal{F}_0$-measurable initial condition.
A stochastic kinetic solution of the generalised Dean--Kawasaki equation \eqref{SPDE-0} is a non-negative, almost surely continuous $L^1_x$-valued $\{\mathcal{F}_t\}_{t\geq 0}$-predictable process $\rho\in L^1(\Omega\times[0,T];L^1_x)$ that satisfies
\begin{enumerate}
    \item Integrability of flux: We have 
    \begin{equation*}
        \sigma(\rho)\in L^2(\Omega;L^2(U\times[0,T])) \hspace{15pt} \text{and}\hspace{15pt} \nu(\rho)\in L^1(\Omega;L^1(U\times[0,T];\mathbb{R}^d)).
    \end{equation*}
    \item Boundary condition, local regularity of solution: For $h_{\rho_b}$ the weak solution of equation \eqref{eq: equation for PDE tilde h}, for  each $k\in\mathbb{N}$ we have
     \begin{equation}\label{eq: local regularity property of stochastic kinetic solution}
         \left((\Phi(\rho)\wedge k)\vee 1/k\right)-\left((h_{\rho_b}\wedge k)\vee 1/k\right)\in L^2(\Omega;L^2([0,T];H^{1}_0(U))).
     \end{equation}
    Furthermore, there exists a kinetic measure $q$ that satisfies:
    \item Regularity: Almost surely, in the sense of non-negative measures,
    \begin{equation}\label{eq: bound for kinetic measure}
        \delta_0(\xi - \rho)\Phi'(\xi)|\nabla \rho|^2\leq q \hspace{5pt} \text{on}\hspace{5pt}U\times(0,\infty)\times[0,T].
    \end{equation}
    \item Vanishing at infinity: We have that
    \begin{equation}\label{eq: decay of kinetic measure at infinity}
        \lim_{M\to\infty}\mathbb{E}\left[q(U\times[M,M+1]\times[0,T])\right]=0.
    \end{equation}
    \item The kinetic equation:
    For every $\psi\in C_c^\infty(U\times(0,\infty))$ and every $t\in(0,T]$, almost surely,
\begin{align}\label{eq: kinetic equation}
&\int_{\mathbb{R}}\int_{U}\chi(x,\xi,t)\psi(x,\xi)\,dx\,d\xi=\int_{\mathbb{R}}\int_{U}\chi(x,\xi,t=0)\psi(x,\xi)\,dx\,d\xi \nonumber\\
&-\int_0^t \int_{U}\left(\Phi'(\rho)\nabla\rho+\frac{1}{2}F_1[\sigma'(\rho)]^2\nabla\rho +\frac{1}{2} \sigma'(\rho)\sigma(\rho)F_2\right)\cdot\nabla\psi(x,\xi)|_{\xi=\rho}\,dx\,ds\nonumber\\
    &-\int_{\mathbb{R}}\int_0^t\int_{U}\partial_\xi\psi(x,\xi)\,dq +\frac{1}{2}\int_0^t \int_{U}\left(
    \sigma'(\rho)\sigma(\rho)\nabla\rho\cdot F_2 + \sigma(\rho)^2F_3 \right)\partial_\xi\psi(x,\rho)\,dx\,ds\nonumber\\
    & -\int_0^t \int_{U}\psi(x,\rho)\nabla\cdot (\sigma(\rho) \,d{\xi}^F)\,dx -        \int_0^t \int_{U}\psi(x,\rho)\nabla\cdot\nu(\rho)\,dx\,ds.
\end{align}
\end{enumerate}
\end{definition}

The following well-posedness result was obtained by the first author in \cite{Shyam25}.

\begin{theorem}[Well-posedness of the generalised Dean--Kawasaki equation with correlated noise \eqref{SPDE-0}]\label{thm: standard well posedness thm from Popat25}
    Suppose that the noise coefficients $F$ satisfy Assumption \ref{asm: assumption on spatial components of the noise}, the non-linear coefficients $(\Phi,\sigma)$ satisfy either Assumption \ref{asm: classic-DK} or \ref{asm: assumptions on non-linear coeffs for contraction estime} as given in Section \ref{subsec: assumptions} below, $\nu$ satisfies Assumption \ref{asm: Assumptions on nu}, and the initial and boundary data satisfy Assumption \ref{asm: assumption on initial data and boundary condition}.
    
    Then there there exists a unique stochastic solution $\rho$ of equation \eqref{SPDE-0} in the sense of Definition \ref{def: stochastic kinetic solution of DK}.
    In particular, the solution is almost surely continuous in time $\rho\in C([0,T];L^1_x)$.

Finally, any two kinetic solutions $\rho_1,\rho_2$ of \eqref{SPDE-0} with initial data $\rho_{1,0}$ and $\rho_{2,0}$ respectively satisfy $\mathbb{P}$-almost surely the contraction
\begin{equation}\label{eq: standard l1 contraction}
    \sup_{t\in[0,T]}\|\rho_1(\cdot,t)-\rho_2(\cdot,t)\|_{L^1_x}\leq \|\rho_{1,0}-\rho_{2,0}\|_{L^1_x}.
\end{equation}
This illustrates that the solution map $\rho_0\mapsto \rho(\cdot;\rho_0)$ 
is continuous from $L^1_x$ to $L^1(U\times[0,T])$.
\end{theorem}

\subsection{Assumptions}\label{subsec: assumptions}
We begin by presenting the assumptions that are needed on the coefficients $(\Phi,\sigma,\nu)$ of equation \eqref{SPDE-0}.
Several of the assumptions are needed for the well-posedness of the equation (Theorem \ref{thm: standard well posedness thm from Popat25} from above) and are not used in the ergodicity arguments of the present work, but are stated for completeness.

As we motivated in the introduction, we will consider separately the classical Dean--Kawasaki case and the porous medium case.
In the classical case, the following choice of coefficients allows \eqref{SPDE-0} to be interpreted as the fluctuating mean-field limit of Brownian particles, coinciding with the classical Dean--Kawasaki model introduced by Dean \cite{D96}. 
\begin{assumption}[Classical Dean--Kawaski assumption]\label{asm: classic-DK}
Suppose that $\Phi(\xi)=\xi$, and $\sigma(\xi)=\sqrt{\xi}$, for every $\xi\geq0$. 
\end{assumption}

On the other hand, the porous medium case $(\Phi(\xi),\sigma(\xi))=(\xi^m,\xi^{\frac{m}{2}})$ with $m>1$ allows \eqref{SPDE-0} to be interpreted as the fluctuating hydrodynamics of the zero-range process with polynomial jump rates.

\begin{assumption}[Porous medium Dean--Kawasaki assumption]\label{asm: assumptions on non-linear coeffs for contraction estime}
    Suppose $\Phi,\sigma\in C([0,\infty))$ satisfy the following assumptions:
    \begin{enumerate}
        \item We have $\Phi,\sigma\in C^{1,1}((0,\infty))$. 
        \item The function $\Phi$ is strictly increasing and is zero at zero: $\Phi(0)=0$ with $\Phi'>0$ on $(0,\infty)$. 
        \item Growth conditions on functions related to $\Phi$: There exists constants $c,m\in(0,\infty)$ such that for every $\xi\in[0,\infty)$
        \begin{equation}\label{eq: growth bound on phi}
            \Phi(\xi)\leq c(1+\xi^m)\quad \text{and} \quad \Phi'(\xi)\leq c(1+\Phi(\xi)).
        \end{equation}
         Let $\Phi$ satisfy points 1 and 2 above.
         Define $\Theta_{\Phi,2}\in C([0,\infty))\cap C^1((0,\infty))$ to be the unique function satisfying
    \begin{equation*}\label{eq: definition of Theta Phi 2 in assumption}
        \Theta_{\Phi,2}(0)=0,\quad \Theta_{\Phi,2}'(\xi)=(\Phi'(\xi))^{1/2}. 
    \end{equation*}
    We assume that there exist constants
    $c\in(0,\infty)$, $q\in[1,\infty)$ such that for every $\xi,\eta\in[0,\infty)$
    \begin{equation*}\label{eq: second assumption on Phi}
        |\xi-\eta|^q\leq c|\Theta_{\Phi,2}(\xi)-\Theta_{\Phi,2}(\eta)|^2.
    \end{equation*}
        Furthermore, assume that there exists a constant $c\in(0,\infty)$ such that for every $\xi\in[0,\infty)$ and $m\in(0,\infty)$ as in \eqref{eq: growth bound on phi}, we have
        \begin{equation*}\label{eq: upper bound on Theta Phi 2 assumption}
            \Theta_{\Phi,2}(\xi)\geq c(\xi^{\frac{m+1}{2}}-1).
        \end{equation*}

 \item Growth conditions on functions related to $\sigma$:
 There exists a constant $c\in(0,\infty)$  such that for every $\xi\in[0,\infty)$,
    \begin{equation}\label{eq: bound on nu squared and sigma sigma' squared}
       |\sigma(\xi)|+ |\sigma(\xi)\sigma'(\xi)|^2\leq  c(1+\xi+\Phi(\xi)).
    \end{equation}
       Furthermore, for each $\delta\in(0,1)$ there exists a constant $c_\delta\in(0,\infty)$ such that for every $\xi\in(\delta,\infty)$,
    \begin{equation}
        \frac{(\sigma'(\xi))^4}{\Phi'(\xi)}\leq c_\delta (1+\xi+\Phi(\xi)).
    \end{equation}   
        
        \item At least linear decay of $\sigma^2$ at zero: There exists a constant $c\in(0,\infty)$ such that
        \begin{equation*}
            \limsup_{\xi\to 0^+}\frac{\sigma^2(\xi)}{\xi}\leq c.
        \end{equation*}
        In particular this implies that $\sigma(0)=0$.
        \item Regularity of oscillations of of $\sigma^2$ at infinity: There is a constant $c\in (0,\infty)$ such that 
        \begin{equation}\label{eq: oscillations of sigma2 at infinity}
            \sup_{\xi'\in[0,\xi]}\sigma^2(\xi')\leq  c(1+\xi+\sigma^2(\xi)),\quad \text{for every }\xi\in[0,\infty).
        \end{equation}
        \item Regularity of $\sigma\sigma'$: The function $\sigma\sigma':\mathbb{R}\to\mathbb{R}$ is non-decreasing.
        Furthermore we have that either $\sigma\sigma'\in C([0,\infty)])$ with $(\sigma\sigma')(0)=0$, or $\nabla\cdot F_2=0$, where $F_2$ is defined in  \eqref{eq: def of F_2}.
        \item Lipschitz type bound for $\sigma\sigma'$: Define the unique function $\Psi_{\sigma}:\mathbb{R}\to\mathbb{R}$ by 
        \begin{equation}\label{eq: def of Psi sigma}
            \Psi_{\sigma}(0)=0, \quad \Psi_{\sigma}'(\xi)=(\sigma'(\xi))^2.
        \end{equation}
        Assume that $\Psi_{\sigma}$ is well defined and that there exists a constant $c\in(0,\infty)$ such that
         \begin{align}\label{eq: local bound for sigma sigma'}
|\sigma(\xi_1)\sigma'(\xi_1)-\sigma(\xi_2)\sigma'(\xi_2)|\leq c|\Psi_{\sigma}(\xi_1)-\Psi_{\sigma}(\xi_2)|,\quad \text{for every }\xi_1,\xi_2\geq0.
\end{align}
 
    \end{enumerate}
\end{assumption}

\begin{remark}
We note that the first seven conditions in Assumption \ref{asm: assumptions on non-linear coeffs for contraction estime} are parallel to those in Fehrman and Gess \cite{FG24} and in the work of the first author \cite{Shyam25}, contributing to the well-posedness of \eqref{SPDE-0}. However, condition (8) in Assumption \ref{asm: assumptions on non-linear coeffs for contraction estime} is new and is introduced to establish the super-contraction property. 
As a representative example, we take $\sigma(\xi)=\xi^{\frac{m}{2}}$ with $m>1$ to obtain 
\begin{align*}
\Psi_{\sigma}(\xi):=\int^{\xi}_0[\sigma'(\zeta)]^2\,d\zeta=\frac{m^2}{4(m-1)}\xi^{m-1},
\end{align*}
and thus
\begin{align*}
\sigma(\xi)\sigma'(\xi)=\frac{m}{2}\xi^{m-1}=\frac{2(m-1)}{m}\Psi_{\sigma}(\xi). 
\end{align*}
Therefore, the bound \eqref{eq: local bound for sigma sigma'} holds for the full range of the porous medium equations.

\end{remark}
We now provide the assumption on the non-linear flux term $\nu$.
\begin{assumption}\label{asm: Assumptions on nu}
Suppose that the non-linear function $\nu\in C([0,\infty);\mathbb{R}^d)\cap C^1([0,\infty);\mathbb{R}^d)$ satisfies
\begin{enumerate}
    \item Non-decreasing: We assume that $\nu=(\nu_1,\hdots,\nu_d)$ is non-decreasing in each component.
     \item  Growth condition on $\nu$:  There exists a constant $c\in(0,\infty)$  such that for every $\xi\in[0,\infty)$,
    \begin{equation}\label{eq: growth bound on nu}
        |\nu(\xi)|+\leq  c(1+\xi+\Phi(\xi)).
    \end{equation}
    \item Regularity of oscillations at infinity: There is a constant $c\in(0,\infty)$ such that 
     \begin{equation}\label{eq: oscillations of nu at infinity}
            \sup_{\xi'\in[0,\xi]}|\nu(\xi')|\leq  c(1+\xi+|\nu(\xi)|),\quad \text{for every }\xi\in[0,\infty).
        \end{equation}
\end{enumerate}
\end{assumption}
We refer to Remark 4.2 of Fehrman and Gess \cite{FG24} for a comprehensive discussion on conditions \eqref{eq: oscillations of nu at infinity} and \eqref{eq: oscillations of sigma2 at infinity}.
The assumptions are satisfied in the case that the functions $\sigma^2, \nu$ are  increasing, are uniformly continuous, or grow linearly at infinity.
The assumption is more general than any of the above three examples and essentially amounts to a restriction on the growth of the magnitude of oscillations, rather than frequency of oscillations at infinity.

Let 
\begin{equation}\label{eq: definition of curly A}
    \mathcal{A}(\xi_1,\xi_2)=|\Phi(\xi_1) -\Phi(\xi_2)| +\sum_{i=1}^d|\nu_i(\xi_1)-\nu_i(\xi_2)|,\quad \text{for every }\xi_1,\xi_2\geq0.  
\end{equation}
Owing to the second point of Assumption \ref{asm: assumptions on non-linear coeffs for contraction estime}, $\mathcal{A}$ is strictly positive unless $\xi_1=\xi_2$.
We further remark that if $\rho_1,\rho_2$ are two stochastic kinetic solutions of \eqref{SPDE-0} in the sense of Definition \ref{def: stochastic kinetic solution of DK}, then the space-time integral\begin{equation*}
    \int_0^t\int_U \mathcal{A}(\rho_1(y,s),\rho_2(y,s))\,dy\,ds
\end{equation*}
is well-defined owing to the growth conditions \eqref{eq: growth bound on phi} and \eqref{eq: growth bound on nu} on $\Phi$ and $\nu$ respectively, combined with the estimates of Propositions 4.12 and 4.14 of the work of the first author \cite{Shyam25}.
For the ergodicity results we require the following upper bound on the growth of $\mathcal{A}$.
\begin{assumption}\label{asm: assumption on curly A}
Suppose that there exists a constant $q_0\geq0$ such that for every $\xi_1,\xi_2\geq0$, we have
\begin{equation}\label{eq: assumption on curly A}
    \mathcal{A}(\xi_1,\xi_2)\geq c_{q_0}|\xi_1-\xi_2|^{q_0+1}. 
\end{equation}
\end{assumption}
Clearly the assumption is satisfied if either $|\Phi(\xi_1)-\Phi(\xi_2)|\geq c_{q_0}|\xi_1-\xi_2|^{q_0+1}$, for every $\xi_1,\xi_2\geq0$, or there exists an $i\in\{1,\hdots,d\}$ such that $|\nu_i(\xi_1)-\nu_i(\xi_2)|\geq c_{q_0}|\xi_1-\xi_2|^{q_0+1}$, for every $\xi_1,\xi_2\geq0$.

In the classical Dean--Kawasaki case it is clear that the assumption holds with $q_0=0$, and in the porous medium case $\Phi(\xi)=\xi^m$, for every $m>1$, since it holds that 
\begin{equation*}
    |\xi_1^m-\xi_2^m|\geq |\xi_1-\xi_2|^{m},\quad \text{for every }\xi_1,\xi_2\geq0, 
\end{equation*}
equation \eqref{eq: assumption on curly A} holds for $q_0=m-1$.  
We comment that the growth assumption \eqref{eq: assumption on curly A} does not hold for fast-diffusion type non-linearities, that is, for $\Phi(\xi)=\xi^m$ for $m\in(0,1)$.
Consequently in the present work we restrict our attention to the classical and porous medium cases.

The assumption on the initial data and boundary data below guarantees the well-posedness of \eqref{SPDE-0}.
We assume that it holds for every statement in the paper and will not repeatedly write so.
\begin{assumption}[Initial data and boundary conditions.]\label{asm: assumption on initial data and boundary condition}
\leavevmode
\begin{enumerate}
    \item Initial data: Assume that the $\mathcal{F}_0$-measurable initial data $\rho_0$ satisfies $\rho_0\in L^2_\omega L^2_x$.
    \item Boundary data: 
    In the classical Dean--Kawasaki case we assume that the boundary data $\rho_b$  is non-negative constant function including zero.
    In the porous medium Dean--Kawaski case the boundary data is non-negative, independent of time and is either a constant function including zero, or a smooth function $\rho_b\in C^\infty(\partial U)$ bounded away from zero. 
\end{enumerate}
\end{assumption}
\begin{remark}
    In the work of the first author \cite{Shyam25} the assumptions on the boundary data are more general than that stated in Assumption \ref{asm: assumption on initial data and boundary condition}.
    Based on a priori estimates, we can consider $H^1(\partial U)$-regular\footnote{See page 176 of Fabes \cite{fabes1978potential} for a rigorous definition of the $H^1(\partial U)$-norm on a $C^1$-regular domain in terms of a local coordinate system of the boundary $\partial U$.} boundary data satisfying several additional integrability constraints.
    The specific integrability constraints are given in points 9-11 of Assumption 4.2, Assumption 4.9, points 5 and 6 of Assumption 4.20, Assumption 4.23 and Assumption 4.28 of \cite{Shyam25}.
    Since it will not play a key role here, we do not go into the details of this point.
    
\end{remark}

\section{Super contraction property and choice of weight function}\label{sec: rigorous proof of super contraction subsection section}

\subsection{Informal computation in the classical case}\label{sec: informal computation and choice of weight function}

In this section we provide an informal computation regarding super contraction estimate, see Section \ref{subsec: key ideas and technical comments}.
By exploiting the weight function $w$, we show that for the functional $\mathcal{A}$ defined in \eqref{eq: definition of curly A}, the standard contraction property \eqref{eq: standard l1 contraction} can be upgraded to a \emph{super--contraction}, albeit in an averaged (expectation) sense. Precisely, we aim to show 
\begin{equation}\label{eq: aim for super contraction in informal proof}
    \mathbb{E}\|\rho_1(\cdot,t)-\rho_2(\cdot,t)\|_{L_{w;x}^1}\leq \mathbb{E}\|\rho_{1,0}-\rho_{2,0}\|_{L_{w;x}^1}-c\mathbb{E}\int_0^t\int_U \mathcal{A}(\rho_1(y,s),\rho_2(y,s))\,dy\,ds,
\end{equation}
for a constant $c\in(0,\infty)$ independent of $t$.

To illustrate the main point, in this section we work with the classical Dean--Kawasaki SPDE \eqref{SPDE-1}, which we recall corresponds to the choice of non-linearities $(\Phi(\xi),\sigma(\xi))=(\xi,\sqrt{\xi})$ .
We will also assume that we have homogeneous Dirichlet boundary conditions $\rho_b=0$ which will allow us to integrate by parts and not pick up additional boundary terms.
These simplifications are removed in the rigorous proof presented in the next section. 

The It\^o form of equation \eqref{SPDE-1} is 
\begin{equation}\label{SPDE-2}
	\partial_t\rho=\Delta\rho-\nabla\cdot\nu(\rho)-\nabla\cdot(\sqrt{\rho}\xi^F)+\frac{1}{4}\nabla\cdot\left(\frac{F_1}{2}\rho^{-1}\nabla\rho+F_2\right). 
\end{equation}

We emphasise that the singular term $\rho^{-1}\nabla\rho=\nabla\log\rho$ arising from the Stratonovich-to-It\^o correction prevents us from considering weak solutions and makes necessary the kinetic solution framework.
The kinetic equation corresponding to \eqref{SPDE-2} satisfies, for every $\psi\in C_c^\infty(U\times(0,\infty))$ and every $t\in(0,T]$, almost surely,
\begin{align}\label{eq: kinetic equation in DK case}
&\int_{\mathbb{R}}\int_{U}\chi(x,\xi,t)\psi(x,\xi)\,dx\,d\xi=\int_{\mathbb{R}}\int_{U}\chi(x,\xi,t=0)\psi(x,\xi)\,dx\,d\xi \nonumber\\
&-\int_0^t \int_{U}\left(\nabla\rho+\frac{1}{8}F_1\rho^{-1}\nabla\rho +\frac{1}{4}F_2\right)\cdot\nabla\psi(x,\xi)|_{\xi=\rho}\,dx\,ds-\int_{\mathbb{R}}\int_0^t\int_{U}\partial_\xi\psi(x,\xi)\,dq\nonumber\\
    & +\frac{1}{2}\int_0^t \int_{U}\left(
    \frac{1}{2}\nabla\rho\cdot F_2 + \rho F_3 \right)\partial_\xi\psi(x,\rho)\,dx\,ds-\int_0^t \int_{U}\psi(x,\rho)\nabla\cdot (\sqrt{\rho} \,d{\xi}^F)\,dx\nonumber\\
    &  -        \int_0^t \int_{U}\psi(x,\rho)\nabla\cdot\nu(\rho)\,dx\,ds.
\end{align}
Let $\rho_1,\rho_2$ be two renormalized kinetic solutions with initial data $\rho_{1,0},\rho_{2,0}$ and kinetic measures $q_1,q_2$ respectively.
Using the properties of indicator functions we see that 
\begin{align}\label{eq: difference of two solutions in terms of kinetic equation}
d\mathbb{E}\|\rho_1-\rho_2\|	_{L^1_{w;x}}:=&d\mathbb{E}\int_U|\rho_1-\rho_2|\omega(x)\nonumber\\
=&d\mathbb{E}\int_U\Big(\int_{\mathbb{R}}|\chi_1-\chi_2|^2\Big)\omega(x)\nonumber\\
=&d\mathbb{E}\int_U\Big(\int_{\mathbb{R}}\chi_1+\chi_2-2\chi_1\chi_2\Big)\omega(x). 
\end{align}
In the informal computation we aim to emphasise the impact of the weight function.
Consequently, to obtain an expression for the first two terms on the right hand side of \eqref{eq: difference of two solutions in terms of kinetic equation}, we have for $i=1,2$ by picking test function\footnote{In the rigorous proof we will multiply this choice of test function $\psi$ with cutoff functions which ensure that $\psi$ is compactly supported in space and velocity.} $\psi(x,\xi)=w(x)$ in the kinetic equation \eqref{eq: kinetic equation in DK case}, the decomposition

\begin{equation}\label{eq: approx}
    \,d\int_{\mathbb{R}}\int_{U}\chi_i(x,\xi,t)w(x)\,dx\,d\xi\approx I_t^{i,wgt}+I_t^{i,mart}+I_t^{i,flux},
\end{equation}
with the weight term involving derivatives of the weight function,
\begin{equation*}
    I_t^{i,wgt}=-\int_0^t \int_{U}\left(\nabla\rho_i+\frac{1}{8}F_1\rho_i^{-1}\nabla\rho_i+\frac{1}{4}F_2\right)\cdot\nabla w(x)\,dx\,ds,
\end{equation*}
the martingale term involving the noise term
\begin{equation*}
    I_t^{i,mart}=-\int_0^t \int_{U}w(x)\nabla\cdot (\sqrt{\rho_i}\,d{\xi}^F)\,dx,
\end{equation*}
and the flux term involving the non-linearity $\nu$,
\begin{equation*}
    I_t^{i,flux}= -        \int_0^t \int_{U}w(x)\nabla\cdot\nu(\rho_i)\,dx\,ds.
\end{equation*}

Although we will not elaborate on this point in the informal computation, we emphasise that equation~\eqref{eq: approx} is written with the symbol ``$\approx$'' rather than as an equality because certain boundary terms arising from integration by parts have been omitted. In the rigorous argument, these terms vanish due to the introduction of a spatial cutoff function, defined in Definition~\ref{def: convolution kernels and spatial cutoff} below. Later, additional cutoff terms will appear, involving derivatives acting on the cutoff functions. By collecting all such terms, we will perform a careful analysis to justify the limiting procedure. Using further properties of the solutions $\rho_1$ and $\rho_2$, we will see that these cutoff contributions are non-positive and may therefore be discarded. To simplify the informal exposition in this subsection, we do not include the details of these cutoff techniques here.

Next we aim to address the cross term.
Using It\^o's product rule, we see that 
\begin{align}\label{eq: product rule for d chi 1 chi 2}
-2d\mathbb{E}\int_U\int_{\mathbb{R}}	\chi_1\chi_2\omega(x)=-2\mathbb{E}\int_U\int_{\mathbb{R}}	(d\chi_1)\chi_2\omega(x)+\chi_1(d\chi_2)\omega(x)+d\langle\chi_1,\chi_2\rangle \omega(x). 
\end{align}
To get an expression for the first term on the right hand side of equation \eqref{eq: product rule for d chi 1 chi 2}, we use the kinetic equation \eqref{eq: kinetic equation in DK case} with test function\footnote{Again we emphasise that in the rigorous proof we will multiply this choice of test function $\psi$ with cutoff functions which ensure that $\psi$ is compactly supported in space and velocity.} $\psi(x,\xi)=\chi_2w(x)$ and the product rule to get 
\begin{align*}
&\,\int_{\mathbb{R}}\int_{U}\,d\chi_1(x,\xi,t)(\chi_2 w(x))\,dx\,d\xi= -\int_0^t \int_{U}\left(\nabla\rho_1+\frac{1}{8}F_1\rho_1^{-1}\nabla\rho_1+\frac{1}{4}F_2\right)\cdot\nabla w(x) \chi_2|_{\xi=\rho_1}\,dx\,ds\nonumber\\
&-\int_0^t \int_{U}w(x)\left(\nabla\rho_1+\frac{1}{8}F_1\rho_1^{-1}\nabla\rho_1+\frac{1}{4}F_2\right)\cdot\nabla\chi_2 |_{\xi=\rho_1}\,dx\,ds\nonumber\\
&-\int_{\mathbb{R}}\int_0^t\int_{U}w(x)\partial_\xi\chi_2\,dq_1+\frac{1}{2}\int_0^t \int_{U}w(x)\left(
     \frac{1}{2}\nabla\rho_1\cdot F_2+\rho_1 F_3 \right)\partial_\xi\chi_2 |_{\xi=\rho_1}\,dx\,ds\nonumber\\
    & -\int_0^t \int_{U}w(x) \chi_2 |_{\xi=\rho_1}\nabla\cdot (\sqrt{\rho_1}\,d{\xi}^F)\,dx -        \int_0^t \int_{U}w(x)\chi_2 |_{\xi=\rho_1}\nabla\cdot\nu(\rho_1)\,dx\,ds.
\end{align*}
The distributional equalities
\begin{equation}\label{eq: distributional equality for derivatives of kinetic function}
    \partial_\xi \chi_2=\delta_0(\xi)-\delta_0(\xi-\rho_2),\quad \nabla_x\chi_2=\delta_0(\xi-\rho_2)\nabla\rho_2,
\end{equation}
then imply that the above can be re-written as
\begin{multline}\label{eq: final expression for d chi1 chi2 w}
\int_{\mathbb{R}}\int_{U}\,d\chi_1(x,\xi,t)(\chi_2 w(x))\,dx\,d\xi =I_t^{1,2,wgt}+I_t^{1,2,err}+I_t^{1,2,meas}+I_t^{1,2,mart}+I_t^{1,2,flux},
\end{multline}

with the weight term
\begin{equation*}
    I_t^{1,2,wgt}= -\int_0^t \int_{U}\left(\nabla\rho_1+\frac{1}{8}F_1\rho_1^{-1}\nabla\rho_1+\frac{1}{4}F_2\right)\cdot\nabla w(x) \chi_2|_{\xi=\rho_1}\,dx\,ds,
\end{equation*}
measure term
\begin{equation*}
    I_t^{1,2,meas}=-\int_0^t \int_{U}w(x)\nabla\rho_1\cdot\nabla\rho_2\delta_0(\rho_1-\rho_2)\,dx\,ds-\int_{\mathbb{R}}\int_0^t\int_{U}w(x)(\delta_0(\xi)-\delta_0(\xi-\rho_2))\,dq_1,
\end{equation*}
error term
\begin{multline*}
    I_t^{1,2,err}=-\int_0^t \int_{U}w(x)\left(\frac{1}{8}F_1\rho_1^{-1}\nabla\rho_1+\frac{1}{4}F_2\right)\cdot\nabla\rho_2\delta_0(\rho_1-\rho_2)\,dx\,ds\\
    +\frac{1}{2}\int_0^t \int_{U}w(x)\left(
     \frac{1}{2}\nabla\rho_1\cdot F_2+ \rho_1 F_3 \right)(\delta_0(\rho_1)-\delta_0(\rho_1-\rho_2))\,dx\,ds,
\end{multline*}
martingale term
\begin{equation*}
    I_t^{1,2,mart}= -\int_0^t \int_{U}w(x) \chi_2 |_{\xi=\rho_1}\nabla\cdot (\sqrt{\rho_1}\,d{\xi}^F)\,dx,
\end{equation*}
and finally the flux term
\begin{equation*}
    I_t^{1,2,flux}=-        \int_0^t \int_{U}w(x)\chi_2 |_{\xi=\rho_1}\nabla\cdot\nu(\rho_1)\,dx\,ds.
\end{equation*}

Similarly, for the second term on the right hand side \eqref{eq: product rule for d chi 1 chi 2}, we use the kinetic equation with respect to $\chi_2$ and test function $\psi=\chi_1 w$. 
This gives the analogous decomposition
\begin{align*}
\int_{\mathbb{R}}\int_{U}\,d\chi_1(x,\xi,t)(\chi_2 w(x))\,dx\,d\xi=I_t^{2,1,wgt}+I_t^{2,1,err}+I_t^{2,1,meas}+I_t^{2,1,mart}+I_t^{2,1,flux},
\end{align*}
where $I_t^{2,1,\cdot}$ terms are analogous to \eqref{eq: final expression for d chi1 chi2 w} with the the roles of $\rho_1$ and $\rho_2$ interchanged.

Finally, the covariation term can be calculated using the definition of the noise $\xi^F$ to give 

\begin{multline}\label{eq: expression for cross variation term}
	\int_U\int_{\mathbb{R}}\,d\langle\chi_1,\chi_2\rangle_t \omega(x)\,dx =\int^t_0\int_U\delta_0(\rho_1-\rho_2)\left(\frac{1}{4}F_1\rho_1^{-1/2}\rho_2^{-1/2}\nabla\rho_1\cdot\nabla\rho_2+\frac{1}{2}\rho_1^{1/2}\rho_2^{-1/2}\nabla\rho_2\cdot F_2\right.\\
    \left.+\frac{1}{2}\rho_2^{1/2}\rho_1^{-1/2}\nabla\rho_1\cdot F_2+F_3\sqrt{\rho_1}\sqrt{\rho_2}\right)\omega(x)\,dx\,ds.
\end{multline}
The covariation terms are put in the error term.

By combining \eqref{eq: difference of two solutions in terms of kinetic equation} and \eqref{eq: product rule for d chi 1 chi 2} with the corresponding computations for each term, we have the decomposition 
\begin{equation}\label{eq: decomposition of difference of two solutions}
    d\int_U\Big(\int_{\mathbb{R}}\chi_1+\chi_2-2\chi_1\chi_2\Big)\omega(x)=I_t^{wgt}+I_t^{err}+I_t^{meas}+I_t^{mart}+I_t^{flux},
\end{equation}
where 
\begin{equation*}
I_t^{wgt,mart,flux}:=I_t^{1,wgt,mart,flux}+I_t^{2,wgt,mart,flux}-2\left(I_t^{1,2,wgt,mart,flux}+I_t^{2,1,wgt,mart,flux}\right),
\end{equation*}
the measure term is given by
\begin{equation*}
    I_t^{meas}=-2\left(I_t^{1,2,meas} + I_t^{2,1,meas}\right),
\end{equation*}
and finally the error term is given by 
\begin{equation*}
    I_t^{err}=-2\left(I_t^{1,2,err} + I_t^{2,1,err}+\,\text{quadratic covariation contribution}\right).
\end{equation*}
Let us evaluate the contribution of each term in turn.

\textbf{Measure term.}

The measure term is
\begin{multline*}
    I_t^{meas}=4\int_0^t \int_{U}w(x)\nabla\rho_1\cdot\nabla\rho_2\delta_0(\rho_1-\rho_2)\,dx\,ds+2\int_{\mathbb{R}}\int_0^t\int_{U}w(x)(\delta_0(\xi)-\delta_0(\xi-\rho_2))\,dq_1\\
+2\int_{\mathbb{R}}\int_0^t\int_{U}w(x)(\delta_0(\xi)-\delta_0(\xi-\rho_1))\,dq_2.
\end{multline*}
Since the kinetic measure vanishes at zero, see Proposition 4.29 of the work of the first author \cite{Shyam25}, it follows that
\begin{equation*}
    \int_{\mathbb{R}}\int_0^t\int_{U}w(x)\delta_0(\xi)\,(dq_1+dq_2)=0.
\end{equation*}
For the remaining terms, it follows by the distributional equality \eqref{eq: bound for kinetic measure} for the kinetic measure that using H\"older's inequality,
\begin{multline*}
    4\int_0^t \int_{U}w(x)\nabla\rho_1\cdot\nabla\rho_2\delta_0(\rho_1-\rho_2)\,dx\,ds-2\int_{\mathbb{R}}\int_0^t\int_{U}w(x)\delta_0(\xi-\rho_2)\,dq_1-2\int_{\mathbb{R}}\int_0^t\int_{U}w(x)\delta_0(\xi-\rho_1)\,dq_2\\
    \leq0.
\end{multline*}
Hence we conclude that
\begin{equation*}
    I_t^{meas}\leq 0.
\end{equation*}

\textbf{Error term.}

The error term is 
\begin{multline*}
    I_t^{err}=2\int_0^t \int_{U}w(x)\left(\frac{1}{8}F_1\rho_1^{-1}\nabla\rho_1+\frac{1}{4}F_2\right)\cdot\nabla\rho_2\delta_0(\rho_1-\rho_2)\,dx\,ds\\
-\int_0^t \int_{U}w(x)\left(\frac{1}{2}\nabla\rho_1\cdot F_2+
     \rho_1 F_3 \right)(\delta_0(\rho_1)-\delta_0(\rho_1-\rho_2))\,dx\,ds\\
+2\int_0^t \int_{U}w(x)\left(\frac{1}{8}F_1\rho_2^{-1}\nabla\rho_2+\frac{1}{4}F_2\right)\cdot\nabla\rho_1\delta_0(\rho_2-\rho_1)\,dx\,ds\\
-\int_0^t \int_{U}w(x)\left(\frac{1}{2}\nabla\rho_2\cdot F_2+
     \rho_2 F_3 \right)(\delta_0(\rho_2)-\delta_0(\rho_2-\rho_1))\,dx\,ds\\
     -2\int^t_0\int_U\delta_0(\rho_1-\rho_2)\left(\frac{1}{4}F_1\rho_1^{-1/2}\rho_2^{-1/2}\nabla\rho_1\cdot\nabla\rho_2+\frac{1}{2}\rho_1^{1/2}\rho_2^{-1/2}\nabla\rho_2\cdot F_2+\right.\\
    \left.\frac{1}{2}\rho_2^{1/2}\rho_1^{-1/2}\nabla\rho_1\cdot F_2+F_3\sqrt{\rho_1}\sqrt{\rho_2}\right)\omega(x)\,dx\,ds.
\end{multline*}

We realise that the covariation term \eqref{eq: expression for cross variation term} helps us cancel out the other terms.

The terms involving $F_1$ can be factorised to obtain
\begin{equation*}
    \frac{1}{4}\int_0^t \int_{U}w(x)\delta_0(\rho_1-\rho_2)F_1(\rho_1^{-1/2}-\rho_2^{-1/2})^2\nabla\rho_1\cdot\nabla \rho_2=0
\end{equation*}

For the terms involving $F_2$, we first realise that, by Stampacchia's lemma, $\nabla\rho_i=0$ on the set $\{\rho_i=0\}$, therefore we have
\begin{equation*}
\int_0^t \int_{U}w(x)\nabla\rho_1\cdot F_2\delta_0(\rho_1)\,dx\,ds=0.
\end{equation*}
We can factorise the remaining terms based on whether we have a $\nabla\rho_1\cdot F_2$ contribution or whether we have a $\nabla\rho_2\cdot F_2$ contribution.
Factorising in this way gives that the remaining terms are given by
\begin{multline*}
    \int_0^t \int_{U}w(x) \delta_0(\rho_1-\rho_2)F_2\cdot\nabla\rho_2(1-\rho_2^{1/2}\rho_1^{-1/2})\,dx\,ds\\
    +\int_0^t \int_{U}w(x)\delta_0(\rho_1-\rho_2)F_2\cdot\nabla\rho_1(1-\rho_1^{1/2}\rho_2^{-1/2})\,dx\,ds=0,
\end{multline*}
where we relied on the identity $\rho_i^{1/2}\rho_j^{-1/2}\delta_0(\rho_i-\rho_j)=1$.
Finally, for the terms involving $F_3$, since we trivially have that $\rho_i=0$ on the set $\{\rho_i=0\}$ it follows that 
\begin{equation*}
    \int_0^t \int_{U}w(x)
     \rho_i F_3 \delta_0(\rho_i)=0.
\end{equation*}
For the remaining terms involving $F_3$, they can be factorised to obtain
\begin{align*}
     \int_0^t\int_U w(x) \delta_0(\rho_1-\rho_2)F_3 \left(\sqrt{\rho_1}-\sqrt{\rho_2}\right)^2=0.
\end{align*}
We showed that 
\begin{equation*}
    I_t^{err}=0.
\end{equation*}
\textbf{Weight term.}
By using the distributional equality
\begin{equation}\label{eq: distributional equality for 2 chi -1}
    2\chi_i|_{\xi=\rho_j}-1=2\mathbbm{1}_{0<\rho_j<\rho_i}-1\approx\sgn(\rho_i-\rho_j),
\end{equation}
for $i\neq j\in\{1,2\}$, where $\sgn$ denotes the sign function\footnote{We remark that the expression $2\mathbbm{1}_{0<\rho_j<\rho_i}-1$ is not 
exactly equivalent to $\sgn(\rho_i-\rho_j)$. In the rigorous argument, 
however, the left--hand side of \eqref{eq: distributional equality for 2 chi -1} 
is replaced by the limit of a velocity--convolution regularization. 
Since the indicator function $\mathbbm{1}_{0<\rho_j<\rho_i}$ is 
discontinuous at $\rho_j=\rho_i$, taking the convolution limit produces 
an additional term $\tfrac{1}{2}\mathbbm{1}_{\rho_j=\rho_i}$. This extra 
contribution ensures that \eqref{eq: distributional equality for 2 chi -1} 
agrees exactly with the sign function $\sgn(\rho_i-\rho_j)$ satisfying $\sgn(0)=0$.
}, the weight term can be re-written as

\begin{multline*}
    I_t^{wgt}=\int_0^t \int_{U}\left(\nabla\rho_1+\frac{1}{8}F_1\rho_1^{-1}\nabla\rho_1+\frac{1}{4}F_2\right)\cdot\nabla w(x) \sgn(\rho_2-\rho_1)\,dx\,ds\\
    +\int_0^t \int_{U}\left(\nabla\rho_2+\frac{1}{8}F_1\rho_2^{-1}\nabla\rho_2+\frac{1}{4}F_2\right)\cdot\nabla w(x)\sgn(\rho_1-\rho_2)\,dx\,ds\\
    =\int_0^t\int_U \left(\nabla\rho_1-\nabla\rho_2+\frac{1}{8}F_1\rho_1^{-1}\nabla\rho_1-\frac{1}{8}F_1\rho_2^{-1}\nabla\rho_2\right)\cdot\nabla w(x) \sgn(\rho_2-\rho_1)\,dx\,ds.
\end{multline*}
In the final line we combined the terms by using that $\sgn$ is an odd function.
Using the monotonicity of the identity function and integrating by parts gives that the difference of the first two terms can be informally simplified to
\begin{multline*}
    \int_0^t\int_U \left(\nabla\rho_1-\nabla\rho_2\right)\cdot\nabla w(x) \sgn(\rho_2-\rho_1)\,dx\,ds=-\int_0^t\int_U \nabla|\rho_1-\rho_2|\cdot\nabla w(x) \,dx\,ds\\
    =\int_0^t\int_U |\rho_1-\rho_2|\Delta w(x) \,dx\,ds.
\end{multline*}
To ensure that this term has a strictly negative contribution, we pick $w$ satisfying 
\begin{equation}\label{eq: first condition on w}
    (\Delta w)(x)< 0, \quad\forall x\in U.
\end{equation}
The same method applies to the second term, this time using that $\xi\mapsto \frac{1}{8}F_1 \log \xi$
is an increasing function.
After integration by parts and an application of the product rule, we obtain
\begin{multline}\label{eq: computation for log term in informal proof}
    \frac{1}{8}\int_0^t\int_U F_1\left(\nabla\log\rho_1-\nabla\log\rho_2\right)\cdot\nabla w(x) \sgn(\rho_2-\rho_1)\,dx\,ds\\
    =-\frac{1}{8}\int_0^t\int_U F_1\nabla|\log\rho_1-\log\rho_2|\cdot\nabla w(x)\,dx\,ds\\
=\frac{1}{8}\int^t_0\int_U\left|\log\rho_1-\log\rho_2\right|[F_2\cdot \nabla w(x)+F_1\Delta w(x)]\,dx\,ds.
\end{multline}

To ensure that this term has a negative contribution, we also impose on the weight function that
\begin{equation}\label{eq: second condition on w}
    F_2(x)\cdot \nabla w(x)+F_1(x)\Delta w(x)=\nabla\cdot(F_1(x)\nabla w(x))<0,\quad \forall x\in U.
\end{equation}
As we mentioned in the introduction, the computation \eqref{eq: computation for log term in informal proof} is only informal, and rigorously holds only in the presence of cutoffs of the solution around zero and infinity.
In the rigorous computation, since the term has a negative contribution, we can remove it before before taking the limits of the cutoff functions.

We showed that under conditions \eqref{eq: first condition on w} and \eqref{eq: second condition on w} on the weight, we have
\begin{equation*}
    I_t^{wgt}\leq \int_0^t\int_U |\rho_1-\rho_2|\Delta w(x) \,dx\,ds.
\end{equation*}

\textbf{Flux term.}

By the same reasoning as the weight term\footnote{The monotonicity of $\nu$ is guaranteed by first point of Assumption \ref{asm: Assumptions on nu}.}, the flux term can be re-written as
\begin{multline*}
     I_t^{flux}=\int_0^t \int_{U}w(x)\nabla\cdot\nu(\rho_1) \sgn(\rho_2-\rho_1)\,dx\,ds+\int_0^t \int_{U}w(x)\nabla\cdot\nu(\rho_2)\sgn(\rho_1-\rho_2)\,dx\,ds\\
     =\int_0^t \int_{U}w(x)\left(\nabla\cdot\nu(\rho_1)-\nabla\cdot\nu(\rho_2)\right) \sgn(\rho_2-\rho_1)\,dx\,ds\\
    =-\sum_{i=1}^d\int_0^t \int_{U}w(x)\partial_{x_i}|\nu_i(\rho_1)-\nu_i(\rho_2)|\,dx\,ds =\sum_{i=1}^d\int_0^t \int_{U}\partial_{x_i}w(x)|\nu_i(\rho_1)-\nu_i(\rho_2)|\,dx\,ds.
\end{multline*}
To ensure that this term has a strictly negative contribution, we get the final condition on the weight function
\begin{equation}\label{eq: third condition on weight function}
    \partial_{x_i} w(x) < 0,\quad \forall i\in \{1,\hdots,d\},\, x\in U.
\end{equation}
We conclude for the flux term that
\begin{equation*}
    I_t^{flux}=\sum_{i=1}^d\int_0^t \int_{U}\partial_{x_i}w(x)|\nu_i(\rho_1)-\nu_i(\rho_2)|\,dx\,ds.
\end{equation*}

\textbf{Martingale term.}

The distributional equality \eqref{eq: distributional equality for 2 chi -1} allows us to re-write the martingale term as 

\begin{align*}
    I_t^{mart}=\int_0^t \int_{U}w(x)\nabla\cdot \left(\left(\sqrt{\rho_1}-\sqrt{\rho_2}\right)\,d{\xi}^F\right)\sgn(\rho_2-\rho_1)\,dx.
\end{align*}

Since we only care about the super contraction in expectation, it follows by the fact the martingale term is a true martingale that
\begin{equation*}
    \mathbb{E} I_t^{mart} =0.
\end{equation*}

\textbf{Conclusion.}

Putting everything together, we have that
\begin{multline}\label{eq: equation at the end for difference of two solutions}
    d\mathbb{E}\|\rho_1-\rho_2\|	_{L^1_{w;x}}
    \leq \mathbb{E}\int_0^t\int_U|\rho_1-\rho_2|\Delta w(x) \,dx\,ds
    +\mathbb{E}\sum_{i=1}^d\int_0^t \int_{U}\partial_{x_i}w(x)|\nu_i(\rho_1)-\nu_i(\rho_2)|\,dx\,ds.
\end{multline}

Based on the discussion in Section \ref{subsec: key ideas and technical comments} and Proposition \ref{prop: choice of weight function in informal proof} below,  an admissible choice of weight function is given by $w(x)=-\exp(\alpha x\cdot e)+C 
$, where $\alpha>0$, $e=\frac{1}{\sqrt{d}}(1,...,1)$ and $C$ is a constant that ensures non-negativity of the weight function.
With this choice, it follows that the right hand side of \eqref{eq: equation at the end for difference of two solutions} can be explicitly written as

\begin{multline}\label{eq: simplifying the super contraction estimate by using def of weight}
    -\alpha^2 \int_0^t\int_U  \left|\rho_1 -\rho_2\right|\exp(\alpha y\cdot e)  \,dy\,ds-\alpha\sum_{i=1}^d\int_0^t \int_{U}\exp(\alpha y\cdot e)e_i|\nu_i(\rho_1)-\nu_i(\rho_2)|\,dy\,ds\\
    \leq - \left(\inf_{y\in U}\exp(\alpha y\cdot e)\right)\left(\alpha^2 \int_0^t\int_U  \left|\rho_1 -\rho_2\right|  \,dy\,ds -\frac{\alpha}{\sqrt{d}}\sum_{i=1}^d\int_0^t \int_{U}|\nu_i(\rho_1)-\nu_i(\rho_2)|\,dy\,ds\right)\\
    \leq -\frac{\alpha}{\sqrt{d}}\left(\inf_{y\in U}\exp(\alpha y\cdot e)\right) \int_0^t\int_U  \left(\left|\rho_1 -\rho_2\right| +\sum_{i=1}^d|\nu_i(\rho_1)-\nu_i(\rho_2)|\right)  \,dy\,ds.
\end{multline}
In the final inequality we used that $d\geq1$ and the fact that $\alpha$ can be chosen to satisfy $\alpha>1$.

In this way, if we let $\mathcal{A}$ be defined as in \eqref{eq: definition of curly A}, we obtain a contraction of the form
\begin{multline*}
     \mathbb{E}\|\rho_1(\cdot,t)-\rho_2(\cdot,t)\|_{L^1_{w;x}}\leq \mathbb{E}\|\rho_{1,0}-\rho_{2,0}\|_{L^1_{w;x}}\\
     -\frac{\alpha}{\sqrt{d}}\left(\inf_{y\in U}\exp(\alpha y\cdot e)\right) \mathbb{E}\int_0^t\int_U \mathcal{A}(\rho_1(y,s),\rho_2(y,s))\,dy\,ds.
\end{multline*}
   
To simplify notation we denote the constant appearing on the right hand side as
\begin{equation}\label{eq: choice of constant}
    c_{\alpha,U,d}:=\frac{\alpha}{\sqrt{d}}\left(\inf_{y\in U}\exp(\alpha y\cdot e)\right),
\end{equation}
which we may subsequently denote by $c$ if we do not want to emphasise the dependence on $\alpha,U$ and $d$.
This completes the formal proof.
\qed

\subsection{Rigorous proof of super-contraction}\label{subsec: rigorous proof of super-contraction}
In this section we extend the super-contraction results to the case of general non-linear coefficients $\Phi,\sigma$ and $\nu$ and general boundary data $\rho_b$ satisfying Assumption \ref{asm: assumption on initial data and boundary condition}. Before proceeding it, we clarify the choice of the weight function in detail.
Recall from the informal proof above that we required the conditions \eqref{eq: first condition on w}, \eqref{eq: second condition on w}, \eqref{eq: third condition on weight function} on the weight function $w$.
We now construct an admissible weight function.

\begin{proposition}[Choice of weight function]\label{prop: choice of weight function in informal proof}
    There exists a non-negative weight function $w:U\to\mathbb{R}$ satisfying for every $x\in U$,
\begin{align}\label{eq: requirements of weight function in the formal proof}
\begin{cases}
    \Delta w(x) < 0\\
    F_2\cdot \nabla w(x)+F_1\Delta w(x)=\nabla\cdot(F_1\nabla w(x))<0,\\
\partial_{x_i} w(x)<0, \text{ for every }i=1,...,d.
\end{cases}
\end{align}
\end{proposition}
\begin{proof}
We construct such a weight.
Let $e=\frac{1}{\sqrt{d}}(1,...,1)$, and denote $e_i=\frac{1}{\sqrt{d}}$, for every $i=1,...,d$. Let the weight function defined by 
\begin{align*}
w(x)=-\exp(\alpha x\cdot e)+C, 
\end{align*}
for some $\alpha>0$ to be decided later on, and $C$ a positive constant that guarantees the positivity of $w$. 

A straightforward computation shows that for every $i=1,...,d$, 
\begin{align*}
\partial_{x_i}w(x)=-\alpha\exp(\alpha x\cdot e)e_i<0,\quad \nabla w(x)=-\alpha\exp(\alpha x\cdot e)e, \text{ and }\Delta w(x)=-\alpha^2\exp(\alpha x\cdot e)<0.
\end{align*}
That is, the first and final condition of \eqref{eq: requirements of weight function in the formal proof} hold true.

Next, recall that we assumed that $F_1>0$ in \eqref{eq: Assumption that F1 is strictly positive}.
This implies that 
\begin{align*}
\nabla\cdot(F_1\nabla w)=\nabla\cdot(-F_1\alpha\exp(\alpha x\cdot e)e)
=-\alpha \exp(\alpha x\cdot e)(F_1\alpha+2F_2\cdot e). 
\end{align*}
Let $F_{1,inf}=\inf_{x\in U}F_1(x)>0$ and $F_{2,sup}=\sup_{x\in U}|F_2(x)|$, which both exist by the assumed continuity of $F_i, i=1,2,3$.
We choose $\alpha>0$ sufficiently large such that 
\begin{align*}
F_{1,inf}\alpha-2F_{2,sup}>0,
\end{align*}
which gives the final condition on the weight function.
\end{proof}

In the following, we make a remark illustrating how to formalise the computation in Section \ref{sec: informal computation and choice of weight function}.
\begin{remark}[Formalising the computation of Section \ref{sec: informal computation and choice of weight function}]
    The proof in Section \ref{sec: informal computation and choice of weight function} was not rigorous in the following ways.
    \begin{enumerate}
        \item We used the test functions $w(x)$ and $\chi_i w(x)$ in the kinetic equation to compute the terms on the right hand side of \eqref{eq: difference of two solutions in terms of kinetic equation}. 
        We can not do this because these choice of test functions are not compactly supported in space or velocity variables.
        Furthermore the kinetic function $\chi_i$ is not a smooth function.
        
        We will get around this by using smooth approximations of cutoff functions in space and velocity as part of the test function and appropriately convolving (smoothing) the kinetic function.

        \item As we mentioned in the informal proof, the singular logarithm term appearing in \eqref{eq: computation for log term in informal proof} is singular and only well defined in the presence of cutoff functions away from zero and infinity.

        \item The distributional equalities for the derivatives of the kinetic function $\partial_\xi \chi_i$ and $\nabla \chi_i$ given in \eqref{eq: distributional equality for derivatives of kinetic function} are not rigorous due to the Dirac delta functions.
        
        To deal with this rigorously we smooth the kinetic function via convolution and then integrate by parts.
        The Dirac delta functions in \eqref{eq: distributional equality for derivatives of kinetic function} corresponds to what is left after taking the limits of the convolution kernels.

        \item At various points we integrated by parts and did not pick up additional terms since we assumed boundary data was constant $\rho_b=0$.
        
        To prove the statement for more general  boundary data we use the cutoff in space as part of our test functions.
        When integrating by parts, we avoid boundary terms appearing, but consequently need to control additional terms that appear when the derivative hits the spatial cutoff.
    \end{enumerate}
\end{remark}

We now introduce the cutoff functions and convolution kernels.
To define the spatial cutoff, we will first need to define interior regions of the domain $U$.
\begin{definition}[Interior region of the domain $U$]
Let $d(x,\partial U)$ denote the usual minimum Euclidean distance from a point to a set.
For $\gamma\in(0,\gamma_U)$ we define the interior regions
    \begin{equation*}
        U_\gamma:=\{x\in U : d(x,\partial U)\geq \gamma\}\subset U,\hspace{20pt} \partial U_\gamma:=\{x\in U: d(x,\partial U)=\gamma\}.
    \end{equation*} 
\end{definition}

 Define the real valued positive constant $\gamma_U\in\mathbb{R}_+$ to be the largest distance away from the boundary such that every point in the interior of the domain at most $\gamma_U$ away from the boundary has a unique closest point on the boundary.
    That is,
\begin{equation}\label{eq: definition of gamma_U}
    \gamma_U:=max\{\tilde\gamma\in\mathbb{R}: \forall x\in\partial U_{\tilde\gamma}, \hspace{3pt} argmin\, d(x,\partial U) \hspace{3pt} \text{is a singleton}\},
\end{equation}
where $argmin\, d(x,\partial U)$ denotes the points $x$ that are closest to the boundary $\partial U$.
For a non-trivial $C^2$-regular domain, $\gamma_U$ is strictly positive, see page 153 of Foote \cite{foote1984regularity}.

\begin{definition}[Convolution kernels and cutoff functions]\label{def: convolution kernels and spatial cutoff}
\leavevmode
\begin{enumerate}
    \item Convolution kernel in space and velocity: 
    Let $\kappa_d\in C_c^\infty(\mathbb{R}^d), \kappa_1\in C_c^\infty(\mathbb{R})$ be non-negative and integrate to one.
    For every $\epsilon,\delta\in(0,1)$ we define $\kappa^\epsilon_d:U\to[0,\infty)$ and $\kappa_1^\delta:\mathbb{R}\to[0,\infty)$ to be convolution kernels/ mollifiers of scale $\epsilon$ and $\delta$ on $U$ and $\mathbb{R}$ respectively, defined by
    \begin{equation}\label{eq: definining convolution kernels}
        \kappa^\epsilon_{d}(x)=\frac{1}{\epsilon^d}\kappa_d\left(\frac{x}{\epsilon}\right),\hspace{5pt}\kappa^\delta_{1}(\xi)=\frac{1}{\delta}\kappa_1\left(\frac{\xi}{\delta}\right).
    \end{equation}
    We define the convolution on $U$ as follows. For any integrable function $f$ and every $x\in U_{2\epsilon}$, let 
$$
(f\ast\kappa^\epsilon_d)(x)=\int_{U}f(y)\kappa^\epsilon_d(x-y)dy.
$$
Finally, for convenience we define $\kappa^{\epsilon,\delta}$ to be the product of the convolution kernels
\begin{equation*}
    \kappa^{\epsilon,\delta}(x,y,\xi,\eta):=\kappa^\epsilon_{d}(x-y)\kappa_1^\delta(\xi-\eta), \hspace{15pt} (x-y,\xi,\eta)\in U\times \mathbb{R}^2.
\end{equation*}

\item Cutoff of small velocity:  For every $\beta\in(0,1)$ we define the  piecewise linear function function $\phi_{\beta}:\mathbb{R}\to[0,1]$ by
    \begin{equation*}
        \phi_\beta(\xi)=1 \hspace{5pt} \text{if}\hspace{5pt}\xi\geq\beta, \hspace{10pt}
    \phi_\beta(\xi)=0 \hspace{5pt} \text{if}\hspace{5pt}\xi\leq\beta/2, \hspace{10pt}
    \phi_\beta'(\xi)=\frac{2}{\beta}\mathbbm{1}_{\beta/2\leq\xi\leq\beta}.
    \end{equation*}
    When multiplied with another function $f:\mathbb{R}\to\mathbb{R}$, the product $(f\phi_\beta)$ takes the value $0$ for small arguments $\xi<\beta/2$ and leaves the function $f$ unchanged for $\xi>\beta$.
    \item Cutoff of large velocity: For every $M\in\mathbb{N}$ we define the piecewise linear function $\zeta_M:\mathbb{R}\to[0,1]$ by
    \begin{equation*}
        \zeta_M(\xi)=1 \hspace{5pt} \text{if}\hspace{5pt}\xi\leq M, \hspace{10pt}
    \zeta_M(\xi)=0 \hspace{5pt} \text{if}\hspace{5pt} \xi\geq M+1, \hspace{10pt}
    \zeta_M'(\xi)=-\mathbbm{1}_{M\leq\xi\leq M+1}.
    \end{equation*}
When multiplied with another function $f:\mathbb{R}\to\mathbb{R}$, the product $(f\zeta_M)$ takes the value $0$ for large arguments $\xi>M+1$ and leaves the function $f$ unchanged for $\xi<M$.

 \item Spatial cutoff around boundary: 
  For $0<\gamma'<\gamma<\gamma_U$ we define the piecewise linear function $\iota_{\gamma,\gamma'}:U\to[0,1]$ by
    \begin{equation}\label{spatial cutoff}
        \iota_{\gamma,\gamma'}(x):=
\begin{cases}
1, & \text{if} \hspace{3pt} d(x,\partial U)>\gamma,\\
(\gamma-\gamma')^{-1}d(x,\partial U_{\gamma'}), & \text{if} \hspace{3pt} \gamma'\leq d(x,\partial U)\leq\gamma,\\
0, & \text{if} \hspace{3pt} 0\leq d(x,\partial U)\leq\gamma',
\end{cases}
\end{equation}
When multiplied with another function $f:U\to\mathbb{R}$, the product $(f \iota_{\gamma,\gamma'})$ takes the value $0$ for points within a $\gamma'$-distance of the boundary and leaves the function $f$ unchanged for points more than distance $\gamma$ from the boundary.

\end{enumerate}
\end{definition}
Compared to the work of the first author \cite{Shyam25}, above we explicitly define $\iota_{\gamma,\gamma'}$ to be compactly supported, and so we do not need to employ an additional approximation argument.

We will also need to establish how to define the gradient of the spatial cutoff. 
To do this we use Lemma 3.3 of the work of the first author \cite{Shyam25}.

\begin{lemma}[Derivative of spatial cutoff]\label{remark derivative of spatial cutoff}
The spatial derivative of the cutoff function $\iota_{\gamma,\gamma'}$ is only non-zero on the $(\gamma,\gamma')$-strip around the boundary $U_{\gamma'}\setminus U_{\gamma}$ (otherwise $\iota_{\gamma,\gamma'}$ is constant so its derivative is zero).
For every $x\in U_{\gamma'}\setminus U_\gamma$ there is a unique closest point $x^*_{\gamma'}=x^*_{\gamma'}(x)$ on the boundary layer $\partial U_{\gamma'}$ to $x$.
Letting $v_{x,\gamma'}$ denote the inward pointing unit normal at the boundary $\partial U_{\gamma'}$ to point $x\in U_{\gamma'}\setminus U_{\gamma}$, the first derivative of the spatial cutoff is given by
\begin{equation*}
    \nabla\iota_{\gamma,\gamma'}(x)=(\gamma-\gamma')^{-1}\frac{x-x^*_{\gamma'}}{|x-x^*_{\gamma'}|}\mathbbm{1}_{U_{\gamma'}\setminus U_\gamma}(x)=(\gamma-\gamma')^{-1}v_{x,\gamma}\mathbbm{1}_{U_{\gamma'}\setminus U_\gamma}(x),\quad x\in U_{\gamma'}\setminus U_\gamma.
\end{equation*}
This implies that the size of the first derivative is of the order $(\gamma-\gamma')^{-1}$,
\begin{equation*}
    |\nabla\iota_{\gamma,\gamma'}(x)|=(\gamma-\gamma')^{-1}\mathbbm{1}_{U_{\gamma'}\setminus U_\gamma}(x).
\end{equation*}
\end{lemma}

To formalise the distributional equality $\nabla\chi_i=\delta_0(\xi-\rho_i)\nabla\rho_i$, we have the below lemma.
The proof can be found in Lemma 3.4 of the work of the first author \cite{Shyam25}.

\begin{lemma}[Integration by parts against kinetic function]\label{lemma integration by parts against kinetic function}
   Let $\psi\in C_c^\infty(U\times(0,\infty))$ be a compactly supported test function (in both arguments) and let $\chi$ be the kinetic function corresponding to solutions of \eqref{SPDE-0} as in \eqref{eq: kinetic function}.
   Then it holds that
\begin{equation*}   
\int_\mathbb{R}\int_{U}\nabla_x\psi(x,\xi)\chi(x,\xi,s)\,dx\,d\xi=-\int_{U}\psi(x,\rho(x,s))\nabla\rho(x,s)\,dx.    
\end{equation*}
\end{lemma}
We are now in a position to prove the $L^1_{\omega}L^1_{w;x}$-super contraction.

\begin{theorem}[{$L^1_{\omega}L^1_{w;x}$}-super contraction for solutions to the Dean--Kawasaki equation \eqref{SPDE-0}]\label{thm: L1 omega super contraction}
Suppose that $(\Phi,\sigma)$ satisfy either Assumption \ref{asm: classic-DK} or \ref{asm: assumptions on non-linear coeffs for contraction estime}, and that $(\Phi,\nu)$ further satisfy Assumptions \ref{asm: Assumptions on nu} and \ref{asm: assumption on curly A}.

Suppose $\rho_1$ and $\rho_2$ are two stochastic kinetic solutions of equation \eqref{SPDE-0} in the sense of Definition \ref{def: stochastic kinetic solution of DK} with $\mathcal{F}_0$-measurable initial data $\rho_{1,0},\rho_{2,0}\in L^1(\Omega;L^1_x)$ respectively and with the same boundary data $\rho_b$ satisfying Assumption \ref{asm: assumption on initial data and boundary condition}.
Let further $\mathcal{A}$ be defined as in \eqref{eq: definition of curly A}.
Then we have for every $t\geq0$
\begin{equation*}
    \mathbb{E}\|\rho_1(\cdot,t)-\rho_2(\cdot,t)\|_{L^1_{w;x}}\leq \mathbb{E}\|\rho_{1,0}-\rho_{2,0}\|_{L^1_{w;x}}-c\mathbb{E}\int_0^t\int_U\mathcal{A}(\rho_1(y,s),\rho_2(y,s))\,dy\,ds.
\end{equation*}
\end{theorem}
\begin{proof}
     Denote by $\chi_1,\chi_2$ the kinetic functions corresponding to the stochastic kinetic solutions $\rho_1,\rho_2$ respectively.
    For $\epsilon,\delta\in(0,1)$, $i=1,2$ and $\kappa^{\epsilon,\delta}$ the product of convolution kernels as in Definition \ref{def: convolution kernels and spatial cutoff}, the regularised kinetic function is defined by
\begin{equation*}
    \chi^{\epsilon,\delta}_{t,i}(y,\eta):=(\chi_i(\cdot,\cdot,t)\ast \kappa^{\epsilon,\delta})(y,\eta), \quad t\geq0,y\in U, \eta\in\mathbb{R}.
\end{equation*}
The symmetry of convolution kernels implies that for $\xi,\eta\in\mathbb{R}$ and  $x,y\in U$, we have that
\begin{equation*}
    \nabla_x \kappa^\epsilon_d(y-x) = -\nabla_y \kappa^\epsilon_d(y-x), \quad \partial_\xi\kappa^\delta_1 (\eta-\xi)=-\partial_\eta\kappa^\delta_1 (\eta-\xi).
\end{equation*}
Using the the kinetic equation \eqref{eq: kinetic equation}, it follows that there is a subset of full probability such that the smoothed kinetic function satisfies for every $i=1,2$, $t\geq0$ and $(y,\eta)\in U_{2\epsilon}\times(2\delta,\delta^{-1})$ such that the convolution kernel is compactly supported,
\begin{multline}\label{long expression for chi in uniqueness proof}
\left.\chi^{\epsilon,\delta}_{s,i}(y,\eta)\right|_{s=0}^t:=\left.(\chi_i(\cdot,\cdot,s)\ast \kappa^{\epsilon,\delta})(y,\eta)\right|_{s=0}^t :=\left. \int_{\mathbb{R}}\int_U \chi_i(x,\xi,s)\kappa^{\epsilon,\delta}(y,x,\eta,\xi)\,dx\,d\xi\right|_{s=0}^t\\
    =\nabla_y\cdot\left(\int_0^t \int_{U}\left(\Phi'(\rho_i)\nabla\rho_i+\frac{1}{2}F_1[\sigma'(\rho_i)]^2\nabla\rho_i +\frac{1}{2} \sigma'(\rho_i)\sigma(\rho_i)F_2\right)\kappa^{\epsilon,\delta}(y,x,\eta,\rho_i)\,dx\,ds\right)\\
    +\partial_\eta\left(\int_0^t \int_{\mathbb{R}}\int_{U}\kappa^{\epsilon,\delta}(y,x,\eta,\xi)\,dq_i\right)\\
    -\frac{1}{2}\partial_\eta\left(\int_0^t \int_{U}\left(\sigma'(\rho_i)\sigma(\rho_i)\nabla\rho_i\cdot F_2 + \sigma(\rho_i)^2F_3 \right)\kappa^{\epsilon,\delta}(y,x,\eta,\rho_i)\,dx\,ds\right)\\
     -\int_0^t \int_{U}\kappa^{\epsilon,\delta}(x,y,\rho,\eta)\nabla\cdot (\sigma(\rho_i) \,d{\xi}^F)\,dx -        \int_0^t \int_{U}\kappa^{\epsilon,\delta}(x,y,\rho,\eta)\nabla\cdot\nu(\rho_i)\,dx\,ds.
\end{multline}
Above we used the standard notation $f(s)|_{s=0}^t:=f(t)-f(0)$.
Let $w:U\to\mathbb{R}$ denote the weight function defined in Proposition \ref{prop: choice of weight function in informal proof}.
From the informal proof, we recall that equation \eqref{eq: difference of two solutions in terms of kinetic equation} tells us that to find an expression for the weighted difference in solutions, we look at a regularised version of
 \begin{multline}\label{eq: rigorous equation for difference of solutions in super contraction proof}
 \left.\int_{U}|\rho_1(x,s)-\rho_2(x,s)| w(x)\,dx\right|_{s=0}^t\\
 =\left.\int_{\mathbb{R}}\int_{U} \Bigl(\chi_1(\xi,\rho(x,s))+\chi_2(\xi,\rho(x,s)) -2\chi_1(\xi,\rho(x,s))\chi_2(\xi,\rho(x,s))\Bigr)w(x)\,dx\,d\xi\right|_{s=0}^t.
\end{multline}
The regularised version will involve smoothing the kinetic functions on the right hand side by using mollifiers as in equation \eqref{long expression for chi in uniqueness proof}, and in order to use the kinetic equation \eqref{eq: kinetic equation} we multiply by smooth approximations of the cutoff functions.
  We begin by treating the first two terms on the right hand side of equation \eqref{eq: rigorous equation for difference of solutions in super contraction proof}.
  Testing equation \eqref{long expression for chi in uniqueness proof} against smooth approximations of the product of cutoff functions and the weight $\zeta_M(\eta)\phi_\beta(\eta) \iota_{\gamma,\gamma'}(y) w(y)$, which are smooth and compactly supported, and subsequently taking the limit of the approximations yields for $(\epsilon,\delta)\in (0,\gamma'/4)\times(0,\beta/4)$ such that the convolutions are well defined in $U\times\mathbb{R}$, the decomposition
\begin{align}\label{eq: decomposition of first order terms in uniqueness proof}    \left.\int_{\mathbb{R}}\int_{U}\chi^{\epsilon,\delta}_{s,i}(y,\eta)w(y)\zeta_M(\eta)\phi_\beta(\eta) \iota_{\gamma,\gamma'}(y)\,dy\,d\eta\right|_{s=0}^t= I_t^{i,wgt}+I^{i,cut}_t+I^{i,mart}_t+I^{i,flux}_t,
\end{align}
where the weight term involves the terms with the gradient of the weight function
\begin{multline*}
    I_t^{i,wgt}:= -\int_{\mathbb{R}}\int_0^t \int_{U^2}\left(\Phi'(\rho_i)\nabla\rho_i +\frac{1}{2}F_1[\sigma'(\rho_i)]^2\nabla\rho_i +\frac{1}{2} \sigma'(\rho_i)\sigma(\rho_i)F_2\right)\\
     \times\kappa^{\epsilon,\delta}(y,x,\eta,\rho_i)\zeta_M(\eta)\phi_\beta(\eta)\iota_{\gamma,\gamma'}(y) \cdot\nabla w(y)\,dy\,dx\,ds\,d\eta,
\end{multline*}
the cutoff term involves terms with derivatives of the cutoff terms
\begin{multline*}
    I_t^{i,cut}:= -\int_{\mathbb{R}}\int_0^t \int_{U^2}\left(\Phi'(\rho_i)\nabla\rho_i +\frac{1}{2}F_1[\sigma'(\rho_i)]^2\nabla\rho_i +\frac{1}{2} \sigma'(\rho_i)\sigma(\rho_i)F_2\right)\\
       \times\kappa^{\epsilon,\delta}(y,x,\eta,\rho_i)\zeta_M(\eta)\phi_\beta(\eta) w(y) \cdot\nabla\iota_{\gamma,\gamma'}(y)\,dy\,dx\,ds\,d\eta\\
    -\int_{\mathbb{R}^2}\int_0^t\int_{U^2}\kappa^{\epsilon,\delta}(y,x,\eta,\xi)\partial_\eta(\zeta_M(\eta)\phi_\beta(\eta)) \iota_{\gamma,\gamma'}(y) w(y)\,dq_i\,dy\,d\eta\\ 
    +\frac{1}{2}\int_{\mathbb{R}}\int_0^t \int_{U^2}\left(
    \sigma'(\rho_i)\sigma(\rho_i)\nabla\rho_i\cdot F_2 + \sigma(\rho_i)^2F_3 \right)\\
    \times \kappa^{\epsilon,\delta}(y,x,\eta,\rho_i)\partial_\eta(\zeta_M(\eta)\phi_\beta(\eta)) \iota_{\gamma,\gamma'}(y) w(y)\,dy\,dx\,ds\,d\eta,
\end{multline*}
the martingale term involves the stochastic integral
\begin{equation*}
    I_t^{i,mart}:=-\int_{\mathbb{R}}\int_0^t \int_{U^2}\zeta_M(\eta)\phi_\beta(\eta) \iota_{\gamma,\gamma'}(y) w(y)\kappa^{\epsilon,\delta}(y,x,\eta,\rho_i)\nabla\cdot (\sigma(\rho_i) \,d{\xi}^F)\,dy\,dx\,ds\,d\eta,
\end{equation*}
and finally the flux term being the term that involves the non-linearity $\nu$,
\begin{equation*}
    I_t^{flux}:= -        \int_{\mathbb{R}}\int_0^t \int_{U^2}\zeta_M(\eta)\phi_\beta(\eta) \iota_{\gamma,\gamma'}(y) w(y)\kappa^{\epsilon,\delta}(y,x,\eta,\rho_i)\nabla\cdot\nu(\rho_i)\,dy\,dx\,ds\,d\eta.
\end{equation*}

To obtain an expression for the final term in equation \eqref{eq: rigorous equation for difference of solutions in super contraction proof}, we introduce the notation $(x,\xi)\in U \times \mathbb{R}$ for arguments in $\rho_1$ and related quantities, and $(x',\xi')\in U\times\mathbb{R}$ for arguments of $\rho_2$ and related quantities.
For convenience we introduce the shorthand notation
\begin{equation}\label{eq: notation for bar k epsilon delta}
    \bar{k}_{s,1}^{\epsilon,\delta}(x,y,\eta):=\kappa^{\epsilon,\delta}(x,y,\eta,\rho_1(x,s)), \quad \bar{k}_{s,2}^{\epsilon,\delta}(x',y,\eta):=\kappa^{\epsilon,\delta}(x',y,\eta,\rho_2(x',s)).
\end{equation}
The stochastic product rule tells us that for the final term in equation \eqref{eq: rigorous equation for difference of solutions in super contraction proof}, we have almost surely for $\beta\in(0,1)$, $M\in\mathbb{N}$, $\gamma'<\gamma\in(0,\gamma_U)$, $\delta\in(0,\beta/4)$, $\epsilon\in(0,\gamma'/4)$, that
\begin{align}\label{equation for stochastic product rule for uniqueness proof}   &\left.\int_{\mathbb{R}}\int_{U}\chi^{\epsilon,\delta}_{s,1}(y,\eta) \chi^{\epsilon,\delta}_{s,2}(y,\eta)w(y) \zeta_M(\eta)\phi_\beta(\eta) \iota_{\gamma,\gamma'}(y)\,dy\,d\eta\right|_{s=0}^t\nonumber\\
    &=\int_{\mathbb{R}}\int_0^t\int_U \Bigl(\chi^{\epsilon,\delta}_{s,1}(y,\eta) d\chi^{\epsilon,\delta}_{s,2}(y,\eta) +  \chi^{\epsilon,\delta}_{s,2}(y,\eta) d\chi^{\epsilon,\delta}_{s,1}(y,\eta)\nonumber\\
    &\hspace{150pt}+ \langle\chi^{\epsilon,\delta}_{1},\chi^{\epsilon,\delta}_{2}\rangle_s(y,\eta)\Bigr)w(y)\zeta_M(\eta)\phi_\beta(\eta) \iota_{\gamma,\gamma'}(y)\,dy\,d\eta \nonumber\\
    &=\int_{\mathbb{R}}\int_U \left(\chi^{\epsilon,\delta}_{s,1}(y,\eta) \left[\left.\chi^{\epsilon,\delta}_{s,2}(y,\eta)\right|_{s=0}^t\right] +  \chi^{\epsilon,\delta}_{s,2}(y,\eta) \left[\left.\chi^{\epsilon,\delta}_{s,1}(y,\eta)\right|_{s=0}^t\right]\right.\nonumber\\
    &\hspace{150pt}+\left.
\left[\left.\langle\chi^{\epsilon,\delta}_{1},\chi^{\epsilon,\delta}_{2}\rangle_s(y,\eta)\right|_{s=0}^t\right] \right) w(y)\zeta_M(\eta)\phi_\beta(\eta) \iota_{\gamma,\gamma'}(y)\,dy\,d\eta,
    \end{align}
where we smoothed the kinetic function so are allowed to use it as part of an admissible test function.
Using equation \eqref{long expression for chi in uniqueness proof} the first term in the final line of \eqref{equation for stochastic product rule for uniqueness proof} can be evaluated as
\begin{align}\label{eq: decomposition of first two terms in the mixed term}
    \int_{\mathbb{R}}\int_U&   \chi^{\epsilon,\delta}_{s,1}(y,\eta) \left.\left[\chi^{\epsilon,\delta}_{s,2}(y,\eta)\right]\right|_{s=0}^tw(y)\zeta_M(\eta)\phi_\beta(\eta) \iota_{\gamma,\gamma'}(y)\,dy\,d\eta\nonumber\\
    &= \int_{\mathbb{R}}\int_U   \zeta_M(\eta)\phi_\beta(\eta) \iota_{\gamma,\gamma'}(y)w(y)\chi^{\epsilon,\delta}_{s,1}(y,\eta)\left[\nabla_y\cdot \left(\int_0^t\int_{U}\left(\Phi'(\rho_2)\nabla\rho_2\right)\bar{k}^{\epsilon,\delta}_{s,2}\,dx'\,ds\right)\nonumber\right.\nonumber\\
    &+ \nabla_y \cdot\left(\int_0^t\int_{U}\left(\frac{1}{2}F_1[\sigma'(\rho_2)]^2\nabla\rho_2 +\frac{1}{2} \sigma'(\rho_2)\sigma(\rho_2)F_2\right)\bar{k}^{\epsilon,\delta}_{s,2}\,dx'\,ds\right)\nonumber\\
    &+\partial_\eta \left(\int_{\mathbb{R}}\int_0^t\int_{U}\kappa^{\epsilon,\delta}(x',y,\xi,\eta)\,dq_2\right)\nonumber\\
    &-\frac{1}{2}\partial_\eta \left(\int_0^t\int_{U}\left(
    \sigma'(\rho_2)\sigma(\rho_2)\nabla\rho_2\cdot F_2 + \sigma(\rho_2)^2F_3 \right)\bar{k}^{\epsilon,\delta}_{s,2}\,dx'\right)\,ds\nonumber\\
    &\left. - \int_0^t\int_{U}\bar{k}^{\epsilon,\delta}_{s,2}\nabla\cdot (\sigma(\rho_2) \,d{\xi}^F)\,dx' -         \int_0^t\int_{U}\bar{k}^{\epsilon,\delta}_{s,2}\nabla\cdot\nu(\rho_2)\,dx'\,ds\right]\,dy\,d\eta.
\end{align} 
Just as with the first order terms, we integrate by parts and move derivatives onto (smooth approximations of) the product $w(y)\zeta_M(\eta)\phi_\beta(\eta)\iota_{\gamma,\gamma'}(y)\chi^{\epsilon,\delta}_{s,1}(y,\eta)$, which are smooth and compactly supported so can be done using classical integration by parts. 
Using the product rule when integrating in the spatial variable and then the integration by part result of Lemma \ref{lemma integration by parts against kinetic function} gives
\begin{align}\label{eq: integrating by parts when gradient falls on the kinetic function}
    \nabla_y \chi^{\epsilon,\delta}_{s,1}(y,\eta)&:=\int_{\mathbb{R}}\int_U \chi_1(x,\xi,s)\nabla_y\kappa^{\epsilon,\delta}(y,x,\eta,\xi)\,dx\,d\xi\nonumber\\
    &=-\int_{\mathbb{R}}\int_U \chi_1(x,\xi,s)\nabla_x\kappa^{\epsilon,\delta}(y,x,\eta,\xi)\,dx\,d\xi\nonumber\\
    &=-\int_{U}\bar{k}^{\epsilon,\delta}_{s,i}\nabla\rho_1(x,s)\,dx.
\end{align}
We consequently obtain the decomposition 
\begin{multline*}
   \int_{\mathbb{R}}\int_0^t\int_U   \chi^{\epsilon,\delta}_{s,1}(y,\eta) d\chi^{\epsilon,\delta}_{s,2}(y,\eta)\zeta_M(\eta)\phi_\beta(\eta) \iota_{\gamma,\gamma'}(y)\,dy\,d\eta\\
   =I_t^{1,2,wgt}+I_t^{1,2,err}+I_t^{1,2,meas}+I_t^{1,2,cut}+I_t^{1,2,mart}+I_t^{1,2,flux}.
\end{multline*}
Adding the term 
\begin{equation}\label{eq: artifically added error term}
    \int_{\mathbb{R}}\int_0^t\int_{U^3}[\Phi'(\rho_1)]^{1/2}[\Phi'(\rho_2)]^{1/2}\nabla\rho_1\cdot\nabla\rho_2\bar{k}^{\epsilon,\delta}_{s,1}\bar{k}^{\epsilon,\delta}_{s,2}\phi_\beta(\eta)\zeta_M(\eta)\iota_{\gamma,\gamma'}(y)w(y)\,dx\,dx'\,dy\,ds\,d\eta
\end{equation}
to the error term and taking it away from the measure term, which enables us to conveniently factorise the error terms, gives for each term separately:

Weight term
\begin{multline*}
    I_t^{1,2,wgt}=-\int_{\mathbb{R}}\int_0^t\int_{U^2}   \zeta_M(\eta)\phi_\beta(\eta) \iota_{\gamma,\gamma'}(y)\chi^{\epsilon,\delta}_{s,1}(y,\eta)\nabla w(y)\cdot\\
    \times\left(\Phi'(\rho_2)\nabla\rho_2+\frac{1}{2}F_1[\sigma'(\rho_2)]^2\nabla\rho_2 +\frac{1}{2} \sigma'(\rho_2)\sigma(\rho_2)F_2\right)\bar{k}^{\epsilon,\delta}_{s,2} \,dx'\,dy\,ds\,d\eta,
\end{multline*}

the error term
\begin{align}\label{eq: for 1 2 error term}
    &I_t^{1,2,err}=-\int_{\mathbb{R}}\int_0^t\int_{U^3}  \zeta_M(\eta)\phi_\beta(\eta) \iota_{\gamma,\gamma'}(y)w(y)\Phi'(\rho_2)\nabla\rho_2\cdot\nabla\rho_1\bar{k}^{\epsilon,\delta}_{s,1}\bar{k}^{\epsilon,\delta}_{s,2}\,dx\,dx'\,dy\,ds\,d\eta\nonumber\\
     &-\int_{\mathbb{R}}\int_0^t\int_{U^3}  \zeta_M(\eta)\phi_\beta(\eta) \iota_{\gamma,\gamma'}(y)w(y)\left(\frac{1}{2}F_1[\sigma'(\rho_2)]^2\nabla\rho_2\cdot\nabla\rho_1 +\frac{1}{2} \sigma'(\rho_2)\sigma(\rho_2)F_2\cdot\nabla\rho_1\right)\nonumber\\
     &\hspace{150pt}\times\bar{k}^{\epsilon,\delta}_{s,1}\bar{k}^{\epsilon,\delta}_{s,2}\,dx\,dx'\,dy\,ds\,d\eta\nonumber\\
    &-\frac{1}{2}\int_{\mathbb{R}}\int_0^t\int_{U^3}   \zeta_M(\eta)\phi_\beta(\eta) \iota_{\gamma,\gamma'}(y)w(y)\left(
    \sigma'(\rho_2)\sigma(\rho_2)\nabla\rho_2\cdot F_2 + \sigma(\rho_2)^2F_3 \right)\bar{k}^{\epsilon,\delta}_{s,2}\bar{k}^{\epsilon,\delta}_{s,1}\,dx\,dx'\,dy\,ds\,d\eta\nonumber\\
    &+\int_{\mathbb{R}}\int_0^t\int_{U^3}[\Phi'(\rho_1)]^{1/2}[\Phi'(\rho_2)]^{1/2}\nabla\rho_1\cdot\nabla\rho_2\bar{k}^{\epsilon,\delta}_{s,1}\bar{k}^{\epsilon,\delta}_{s,2}\phi_\beta(\eta)\zeta_M(\eta)\iota_{\gamma,\gamma'}(y)w(y)\,dx\,dx'\,dy\,ds\,d\eta,
\end{align}

the measure term
\begin{align}\label{eq: for 1 2 measure term}
    I_t&^{1,2,meas}= \int_{\mathbb{R}^2}\int_0^t\int_{U^3}   \zeta_M(\eta)\phi_\beta(\eta)\iota_{\gamma,\gamma'}(y)w(y)\kappa^{\epsilon,\delta}(x',y,\xi,\eta) \bar{k}^{\epsilon,\delta}_{s,1}\,dq_2(x',\xi,s)\,dx\,dy\,d\eta\nonumber\\
    &-\int_{\mathbb{R}}\int_0^t\int_{U^3}[\Phi'(\rho_1)]^{1/2}[\Phi'(\rho_2)]^{1/2}\nabla\rho_1\cdot\nabla\rho_2\bar{k}^{\epsilon,\delta}_{s,1}\bar{k}^{\epsilon,\delta}_{s,2}\phi_\beta(\eta)\zeta_M(\eta)\iota_{\gamma,\gamma'}(y)w(y)\,dx\,dx'\,dy\,ds\,d\eta,
\end{align}

the cutoff term
\begin{align*}
    &I^{1,2,cut}_t = -\int_{\mathbb{R}^2}\int_0^t\int_{U^2}   \partial_\eta(\zeta_M(\eta)\phi_\beta(\eta))\chi^{\epsilon,\delta}_{s,1}(y,\eta) \iota_{\gamma,\gamma'}(y) w(y)\kappa^{\epsilon,\delta}(x',y,\xi,\eta)\,dq_2(x',\xi,s)\,dy\,d\eta\nonumber\\
    &+\frac{1}{2}\int_{\mathbb{R}}\int_0^t\int_{U^2}   \partial_\eta(\zeta_M(\eta)\phi_\beta(\eta)) \chi^{\epsilon,\delta}_{s,1}(y,\eta) \iota_{\gamma,\gamma'}(y) w(y)\left(
    \sigma'(\rho_2)\sigma(\rho_2)\nabla\rho_2\cdot F_2 + \sigma(\rho_2)^2F_3 \right)\nonumber\\
    &\hspace{150pt}\times\bar{k}^{\epsilon,\delta}_{s,2}\,dx'\,dy\,ds\,d\eta\nonumber\\
    &-\int_{\mathbb{R}}\int_0^t\int_{U^2}  \zeta_M(\eta)\phi_\beta(\eta) \chi^{\epsilon,\delta}_{s,1}(y,\eta)w(y)\nabla_y\iota_{\gamma,\gamma'}(y)\cdot\nonumber\\
    &\hspace{70pt}\left(\Phi'(\rho_2)\nabla\rho_2+\frac{1}{2}F_1[\sigma'(\rho_2)]^2\nabla\rho_2 +\frac{1}{2} \sigma'(\rho_2)\sigma(\rho_2)F_2\right) \bar{k}^{\epsilon,\delta}_{s,2}\,dx'\,dy\,ds\,d\eta,
\end{align*}

the martingale term
\begin{align*}
     I^{1,2,mart}_t&=-\int_{\mathbb{R}}\int_0^t\int_{U^2}   \bar{k}^{\epsilon,\delta}_{s,2}\zeta_M(\eta)\phi_\beta(\eta) \iota_{\gamma,\gamma'}(y)w(y)\chi^{\epsilon,\delta}_{s,1}(y,\eta)\nabla\cdot (\sigma(\rho_2) \,d{\xi}^F)\,dx'\,dy\,d\eta,
\end{align*}

And finally the flux term
\begin{align*}\label{eq: for mixed conservative term}
    I^{1,2,flux}_t&=-\int_{\mathbb{R}}\int_0^t\int_{U^2}   \bar{k}^{\epsilon,\delta}_{s,2}\zeta_M(\eta)\phi_\beta(\eta) \iota_{\gamma,\gamma'}(y) w(y)\chi^{\epsilon,\delta}_{s,1}(y,\eta)\nabla\cdot\nu(\rho_2)\,dx'\,dy\,ds\,d\eta.
\end{align*}
In the expressions above, we have an extra spatial integral for the terms when the derivative falls on the kinetic function due to \eqref{eq: integrating by parts when gradient falls on the kinetic function}.

An analogous decomposition holds for the second term on the right hand side of \eqref{equation for stochastic product rule for uniqueness proof}. 
The the corresponding weight, error, measure, cutoff, martingale and flux terms up to time $t\geq0$ are given by 
\begin{equation*}
    I^{2,1,\cdot}_t,
\end{equation*}
where we again artificially add the term \eqref{eq: artifically added error term} to the error term and take it away from the measure term.

Finally we deal with the quadratic covariation term in equation \eqref{equation for stochastic product rule for uniqueness proof}.
The computation in equation \eqref{eq: expression for cross variation term} is evaluated in two steps. 
Firstly, using convolution kernels to smooth the kinetic functions, it holds that
\begin{multline*}
d\langle\chi^{\epsilon,\delta}_{\cdot,1},\chi^{\epsilon,\delta}_{\cdot,2}\rangle_s(y,\eta):=d\langle(\chi_1\ast \kappa^{\epsilon,\delta}),(\chi_2\ast \kappa^{\epsilon,\delta})\rangle_s(y,\eta)\\
=\int_{U^2}\int_{\mathbb{R}^2}d\langle \chi_1,\chi_2\rangle_s \kappa^{\epsilon,\delta}_{s,1}(y,x,\eta,\xi)\kappa^{\epsilon,\delta}_{s,2}(y,x',\eta,\xi')\,d\xi\,d\xi'\,dx\,dx'\\
=\int_{U^2}\nabla\cdot\left(\sigma(\rho_1)\sum_{k=1}^\infty f_k(x)d\beta_k(s)\right)\nabla\cdot\left(\sigma(\rho_2)\sum_{j=1}^\infty f_j(x')d\beta_j(s)\right)
\kappa^{\epsilon,\delta}_{s,1}(x,y,\rho_1,\eta)\kappa^{\epsilon,\delta}_{s,2}(x',y,\rho_2,\eta)\,dx\,dx'\\
=\sum_{j,k=1}^\infty\int_{U^2}\left(f_k\sigma'(\rho_1)\nabla\rho_1+\sigma(\rho_1)\nabla f_k\right)\left(f_j\sigma'(\rho_2)\nabla\rho_2+\sigma(\rho_2)\nabla f_j\right)d\langle \beta_k(\cdot),\beta_j(\cdot)\rangle_s
\bar{k}^{\epsilon,\delta}_{s,1}\bar{k}^{\epsilon,\delta}_{s,2}dx\,dx'\\
=\sum_{k=1}^\infty\int_{U^2}\left(f_k\sigma'(\rho_1)\nabla\rho_1+\sigma(\rho_1)\nabla f_k\right)\left(f_k\sigma'(\rho_2)\nabla\rho_2+\sigma(\rho_2)\nabla f_k\right)\bar{k}^{\epsilon,\delta}_{s,1}\bar{k}^{\epsilon,\delta}_{s,2}\,dx\,dx' \,ds.
\end{multline*}
In the second step we integrate against smooth approximations of the product $\phi_\beta(\eta) \zeta_M(\eta)\iota_{\gamma,\gamma'}(y) w(y)$. Consequently one obtains that that quadratic covariation term is given by
    \begin{multline}\label{eq: equality for quadratic variation term}
        \int_\mathbb{R}\int_0^t\int_U d\langle\chi^{\epsilon,\delta}_{\cdot,1},\chi^{\epsilon,\delta}_{\cdot,2}\rangle_s(y,\eta)\, \phi_\beta(\eta) \zeta_M(\eta) \iota_{\gamma,\gamma'}(y)w(y)\,dy\,d\eta\\
        =\sum_{k=1}^{\infty}\int_\mathbb{R}\int_0^t\int_{U^3} \left(f_k\sigma'(\rho_1)\nabla\rho_1+\sigma(\rho_1)\nabla f_k\right)\cdot\left(f_k\sigma'(\rho_2)\nabla\rho_2+\sigma(\rho_2)\nabla f_k\right)\\
        \times\bar{k}^{\epsilon,\delta}_{s,1}\bar{k}^{\epsilon,\delta}_{s,2} \phi_\beta(\eta) \zeta_M(\eta) \iota_{\gamma,\gamma'}(y) w(y)\,dx\,dx'\,dy\,ds\,d\eta.
    \end{multline} 

Putting \eqref{equation for stochastic product rule for uniqueness proof} and subsequent computations from \eqref{eq: decomposition of first two terms in the mixed term}-\eqref{eq: equality for quadratic variation term} together, it follows that for the mixed term, we have the decomposition
\begin{multline}\label{mixed term in uniqueness}
    \left.\int_{\mathbb{R}}\int_{U}\chi^{\epsilon,\delta}_{s,1}(y,\eta)\chi^{\epsilon,\delta}_{s,2}(y,\eta) w(y)\zeta_M(\eta)\phi_\beta(\eta) \iota_{\gamma,\gamma'}(y)\,dy\,d\eta\right|_{s=0}^t\\
=I_t^{mix,wgt}+I_t^{mix,err}+I_t^{mix,meas}+I_t^{mix,cut}+I_t^{mix,mart}+I_t^{mix,flux}.
\end{multline}
All four terms of the quadratic covariation term \eqref{eq: equality for quadratic variation term} are put into the error term
\begin{equation*}
    I_t^{mix,err}=I_t^{1,2,err} + I_t^{2,1,err}+\,\text{quadratic covariation contribution}.
\end{equation*}

The contribution of the remaining terms just comes from the first two components of \eqref{equation for stochastic product rule for uniqueness proof},
\begin{equation*}
    I_t^{mix,wgt,meas,cut,mart,flux}=I_t^{1,2,wgt,meas,cut,mart,flux}+I_t^{2,1,wgt,meas,cut,mart,flux}.
\end{equation*}

Returning back to the equation of interest that governs the $L^1_x$-difference of two solutions, equation \eqref{eq: rigorous equation for difference of solutions in super contraction proof}, we have the decomposition
\begin{equation}\label{equation for difference of two solutions in uniqueness proof}
    \left.\int_{\mathbb{R}}\int_U \left(\chi^{\epsilon,\delta}_{s,1}+ \chi^{\epsilon,\delta}_{s,2} -2\chi^{\epsilon,\delta}_{s,1}\chi^{\epsilon,\delta}_{s,2}\right) w\phi_\beta \zeta_M\iota_{\gamma,\gamma'}\right|_{s=0}^t=  I_t^{err}+ I_t^{meas} + I_t^{wgt} +  I_t^{mart} + I_t^{cut}+ I_t^{flux}.
\end{equation}

The error and measure term and arise solely from the mixed term \eqref{mixed term in uniqueness},
\begin{equation}\label{eq: equation for error and measure term in terms of the mixed error and measure terms}
    I_t^{err,meas}=-2I_t^{mix,err,meas}
\end{equation}
the final term on the left hand side of \eqref{equation for difference of two solutions in uniqueness proof}.
The remaining terms have a contribution from all three terms on the left hand side of equation \eqref{equation for difference of two solutions in uniqueness proof}, and are given respectively by
\begin{equation}\label{eq: expression for martingale cut and cons terms}
    I_t^{wgt,mart,cut,flux}= I_t^{1,wgt,mart,cut,flux}+ I_t^{2,wgt,mart,cut,flux}-2 I_t^{mix,wgt,mart,cut,flux}.
\end{equation}
We deal with each term on the right hand side of \eqref{equation for difference of two solutions in uniqueness proof} separately.

\textbf{Measure term.}

From equations \eqref{eq: for 1 2 measure term} and \eqref{eq: equation for error and measure term in terms of the mixed error and measure terms}, the measure term is 
\begin{multline*}
    I_t^{meas}= -2\int_{\mathbb{R}^2}\int_0^t\int_{U^3}   \zeta_M(\eta)\phi_\beta(\eta)\iota_{\gamma,\gamma'}(y) w(y)\kappa^{\epsilon,\delta}(x,y,\xi,\eta) \bar{k}^{\epsilon,\delta}_{s,2}\,dq_1(x,\xi,s)\,dx'\,dy\,d\eta\\
    -2\int_{\mathbb{R}^2}\int_0^t\int_{U^3}   \zeta_M(\eta)\phi_\beta(\eta)\iota_{\gamma,\gamma'}(y) w(y)\kappa^{\epsilon,\delta}(x',y,\xi,\eta) \bar{k}^{\epsilon,\delta}_{s,1}\,dq_2(x',\xi,s)\,dx\,dy\,d\eta\\
    +4\int_{\mathbb{R}}\int_0^t\int_{U^3}[\Phi'(\rho_1)]^{1/2}[\Phi'(\rho_2)]^{1/2}\nabla\rho_1\cdot\nabla\rho_2\bar{k}^{\epsilon,\delta}_{s,1}\bar{k}^{\epsilon,\delta}_{s,2}\phi_\beta(\eta)\zeta_M(\eta)\iota_{\gamma,\gamma'}(y) w(y)\,dx\,dx'\,dy\,ds\,d\eta.
\end{multline*}
H\"older's inequality and the distributional equality \eqref{eq: bound for kinetic measure} from the definition of stochastic kinetic solutions tells us that
\begin{equation}
    I_{t}^{meas}\leq 0.
\end{equation} 

\textbf{Error term.}

Combining equations \eqref{eq: for 1 2 error term} and \eqref{eq: equality for quadratic variation term}, it follows that we have a convenient factorisation for the error term similar to the factorisation of the error term in the formal proof.

Specifically, using the dominated convergence theorem and the notation\footnote{This is similar to \eqref{eq: notation for bar k epsilon delta} without the $\epsilon$-dependence.} $\bar{k}^\delta_{s,i}(y,\eta):=\kappa_1^\delta(\rho_i(y,s)-\eta)$ for $i=1,2$, for every fixed $\gamma'<\gamma\in(0,\gamma_U)$, $\delta\in(0,\beta/4)$, after taking the limit as $\epsilon\to0$ the error term is
\small
\begin{multline}\label{eq: full error term}
    \lim_{\epsilon\to0}I_t^{err}=2\int_{\mathbb{R}}\int_0^t\int_U\left((\Phi'(\rho_1))^{1/2}-(\Phi'(\rho_2))^{1/2}\right)^2\nabla\rho_1\cdot\nabla\rho_2\bar{k}^\delta_{s,1}\bar{k}^\delta_{s,2}\phi_\beta(\eta)\zeta_M(\eta)\iota_{\gamma,\gamma'}(y) w(y)\,dy\,ds\,d\eta\\
+\int_\mathbb{R}\int_0^t\int_U\left(F_1\left(\sigma'(\rho_1)-\sigma'(\rho_2)\right)^2\nabla\rho_1\cdot\nabla\rho_2+F_3\left(\sigma(\rho_1)-\sigma(\rho_2)\right)^2\right)\bar{k}^\delta_{s,1}\bar{k}^\delta_{s,2}\phi_\beta(\eta)\zeta_M(\eta)\iota_{\gamma,\gamma'}(y) w(y)\,dy\,ds\,d\eta\\
    +\int_\mathbb{R}\int_0^t\int_U\left(\sigma'(\rho_1)\sigma(\rho_1)+\sigma'(\rho_2)\sigma(\rho_2)-2\sigma'(\rho_1)\sigma(\rho_2)\right)F_2\cdot\nabla\rho_1\bar{k}^\delta_{s,1}\bar{k}^\delta_{s,2}\phi_\beta(\eta)\zeta_M(\eta)\iota_{\gamma,\gamma'}(y) w(y)\,dy\,ds\,d\eta\\
    +\int_\mathbb{R}\int_0^t\int_U\left(\sigma'(\rho_1)\sigma(\rho_1)+\sigma'(\rho_2)\sigma(\rho_2)-2\sigma'(\rho_2)\sigma(\rho_1)\right)F_2\cdot\nabla\rho_2\bar{k}^\delta_{s,1}\bar{k}^\delta_{s,2}\phi_\beta(\eta)\zeta_M(\eta)\iota_{\gamma,\gamma'}(y) w(y)\,dy\,ds\,d\eta.
\end{multline}
\normalsize
The definition of the convolution kernels imply that whenever we have 
\begin{equation*}\label{eq: cutoffs and convolution kernels are not zero}
    \bar{k}^\delta_{s,1}(y,\eta)\bar{k}^\delta_{s,2}(y,\eta)\phi_\beta(\eta)\zeta_M(\eta)\iota_{\gamma,\gamma'}(y)\neq 0,
\end{equation*}
for $\delta\in(0,\beta/4)$, we have owing to Assumption \ref{asm: assumptions on non-linear coeffs for contraction estime}, the local Lipschitz regularity of $\Phi'$, the local $\frac{1}{2}$-H\"older continuity of the square root, the triangle inequality and the local boundedness and Lipschitz regularity of $\sigma$ and $\sigma'$ respectively, that for a constant $c\in(0,\infty)$ depending on $M$ and $\beta$,
\begin{multline*}
    \left((\Phi'(\rho_1))^{1/2}-(\Phi'(\rho_2))^{1/2}\right)^2+\left(\sigma'(\rho_1)-\sigma'(\rho_2)\right)^2+\left(\sigma(\rho_1)-\sigma(\rho_2)\right)^2\\
    +\left|\sigma'(\rho_1)\sigma(\rho_1)+\sigma'(\rho_2)\sigma(\rho_2)-2\sigma'(\rho_1)\sigma(\rho_2)\right|+\left|\sigma'(\rho_1)\sigma(\rho_1)+\sigma'(\rho_2)\sigma(\rho_2)-2\sigma'(\rho_2)\sigma(\rho_1)\right|
    \\\leq c\delta\mathbbm{1}_{0\leq|\rho_1(y,s)-\rho_2(y,s)|\leq c\delta}.
\end{multline*}
Coming back to the error term \eqref{eq: full error term}, the boundedness of $F_i$ for $i=1,2,3$ and H\"older's and Young's inequalities prove that there exists a constant $c\in (0,\infty)$ depending on $M$ and $\beta$ satisfying for every $t\geq0$,
\begin{multline}\label{eq: bounding the error term by a factor of delta}
    \limsup_{\epsilon\to0}|I_t^{err}|\\
    \leq c\int_{\mathbb{R}}\int_0^t\int_{U}\delta \mathbbm{1}_{0\leq|\rho_1(y,s)-\rho_2(y,s)|\leq c\delta}\left(1+|\nabla\rho_1|^2+|\nabla\rho_2|^2\right)\bar{k}^\delta_{s,1}\bar{k}^\delta_{s,2}\phi_\beta(\eta)\zeta_M(\eta)\iota_{\gamma,\gamma'}(y) w(y)\,dy\,ds\,d\eta.
\end{multline}
    It is then a consequence of the definition of the cutoff functions and convolution kernels, the local regularity property\footnote{The local regularity property in \eqref{eq: local regularity property of stochastic kinetic solution} is stated for $\Phi(\rho)$, but since we assume $\Phi$ is strictly increasing it also holds that locally $\rho\in H^1(U)$.} of stochastic kinetic solutions \eqref{eq: local regularity property of stochastic kinetic solution}, the dominated convergence theorem and equation \eqref{eq: bounding the error term by a factor of delta} that almost surely, for every $t\geq0$,
\begin{equation*}
    \limsup_{\delta\to0}\left(\limsup_{\epsilon\to0}|I_t^{err}|\right)=0.
\end{equation*}

\textbf{Weight term.}

Taking the $\epsilon,\delta\to 0$ limits and using that the sign function is odd to combine terms gives that

\small
\begin{multline}\label{eq: epsilon and delta limit of weight term}
    \lim_{\epsilon,\delta\to0}I_t^{wgt}=\int_0^t\int_U \left(\Phi'(\rho_1)\zeta_M(\rho_1)\phi_\beta(\rho_1)\nabla\rho_1 -\Phi'(\rho_2)\zeta_M(\rho_2)\phi_\beta(\rho_2)\nabla\rho_2\right)\cdot\nabla w(y) \iota_{\gamma,\gamma'}(y) \sgn(\rho_2-\rho_1)\,dy\,ds\\
    +\int_0^t\int_U \left(\frac{1}{2}F_1[\sigma'(\rho_1)]^2\zeta_M(\rho_1)\phi_\beta(\rho_1)\nabla\rho_1 -\frac{1}{2}F_1[\sigma'(\rho_2)]^2\zeta_M(\rho_2)\phi_\beta(\rho_2)\nabla\rho_2\right)\cdot\nabla w(y) \iota_{\gamma,\gamma'}(y) \sgn(\rho_2-\rho_1)\,dy\,ds\\
    \int_0^t\int_U \left(\frac{1}{2} \sigma'(\rho_1)\sigma(\rho_1)\zeta_M(\rho_1)\phi_\beta(\rho_1)F_2 -\frac{1}{2} \sigma'(\rho_2)\sigma(\rho_2)\zeta_M(\rho_2)\phi_\beta(\rho_2)F_2\right)\cdot\nabla w(y) \iota_{\gamma,\gamma'}(y) \sgn(\rho_2-\rho_1)\,dy\,ds.
\end{multline}
\normalsize
For the first integral, we define the unique function $\Phi_{M,\beta}:[0,\infty)\to[0,\infty)$ by 
\begin{equation}\label{eq: definition of Phi m beta}
    \Phi_{M,\beta}(0)=0,\quad \Phi_{M,\beta}'(\xi)= \Phi'(\xi)\zeta_M(\xi)\phi_\beta(\xi).
\end{equation}
By the fact that $\Phi'_{M,\beta}\geq0$, it follows that $\Phi_{M,\beta}$ is a non-decreasing function, and so we have for the first integral on the right hand side of \eqref{eq: epsilon and delta limit of weight term}, that by taking the limit as $\gamma'\to0$ followed by $\gamma\to0$ and integrating by parts,
\begin{align}\label{eq: computation for fx-fy times sgn x-y}
    \lim_{\gamma\to0}\lim_{\gamma'\to0}&\int_0^t\int_U \left(\nabla\Phi_{M,\beta}(\rho_1) -\nabla\Phi_{M,\beta}(\rho_2)\right)\cdot\nabla w(y) \iota_{\gamma,\gamma'}(y) \sgn(\rho_2-\rho_1)\,dy\,ds\nonumber\\
    =&\lim_{\delta\to0}\int_0^t\int_U \left(\nabla\Phi_{M,\beta}(\rho_1) -\nabla\Phi_{M,\beta}(\rho_2)\right)\cdot\nabla w(y) \sgn_{\delta}(\rho_2-\rho_1)\,dy\,ds\nonumber\\
    =&-\lim_{\delta\to0}\int_0^t\int_U \nabla a_{\delta}\left(\Phi_{M,\beta}(\rho_1) -\Phi_{M,\beta}(\rho_2)\right)\cdot\nabla w(y) \,dy\,ds\nonumber\\
    =&\lim_{\delta\to0}\int_0^t\int_U  a_{\delta}\left(\Phi_{M,\beta}(\rho_1) -\Phi_{M,\beta}(\rho_2)\right)\Delta w(y) \,dy\,ds\nonumber\\
    =&\int_0^t\int_U  \left|\Phi_{M,\beta}(\rho_1) -\Phi_{M,\beta}(\rho_2)\right|\Delta w(y) \,dy\,ds.
\end{align}
Above, $\sgn_\delta$ denotes the smoothed sign function $\sgn_\delta:=\sgn\ast\kappa^\delta_1$ with $\kappa^\delta_1$ defined in Definition \ref{def: convolution kernels and spatial cutoff} and $a_\delta$ is a smooth approximation of the absolute value satisfying $a_\delta(0)=0$. 
In the penultimate line we integrate by parts and do not pick up additional boundary terms due to the fact that the solutions $\rho_1,\rho_2$ coincide on the boundary.

We finally take the limit as $M\to\infty$ and $\beta\to0$ and use the pointwise convergence 
    $\Phi_{M,\beta}\to\Phi$
to deduce that this term can be written as
\begin{equation*}
    \lim_{M\to\infty,\beta\to0}\int_0^t\int_U  \left|\Phi_{M,\beta}(\rho_1) -\Phi_{M,\beta}(\rho_2)\right|\Delta w(y) \,dy\,ds=\int_0^t\int_U  \left|\Phi(\rho_1) -\Phi(\rho_2)\right|\Delta w(y) \,dy\,ds.
\end{equation*}
We apply an analogous reasoning to the second weight term. 
For $M\in\mathbb{N}, \beta\in(0,1)$ we define the unique function $\Psi_{\sigma,M,\beta}:\mathbb{R}\to\mathbb{R}$ by 
\begin{equation}
    \Psi_{\sigma,M,\beta}(0)=0,\quad \Psi_{\sigma,M,\beta}'(\xi)=(\sigma'(\xi))^2\phi_\beta(\xi)\zeta_M(\xi),
\end{equation}
which is non-negative and well defined due to the cutoff at zero.
This gives that the second term on the right hand side of \eqref{eq: epsilon and delta limit of weight term} can be written analogously to \eqref{eq: computation for fx-fy times sgn x-y} in the $\gamma'\to0$ followed by $\gamma\to0$ limit as

\begin{multline}\label{Psisigma}
    \lim_{\delta\rightarrow0}\frac{1}{2}\int_0^t\int_U F_1\left(\nabla \Psi_{\sigma,M,\beta}(\rho_1)-\nabla \Psi_{\sigma,M,\beta}(\rho_2)\right)\cdot\nabla w(y)
    \sgn_{\delta}(\rho_2-\rho_1)\,dy\,ds\\
    =\lim_{\delta\rightarrow0}\frac{1}{2}\int_0^t\int_U F_1\left(\nabla a_\delta(\Psi_{\sigma,M,\beta}(\rho_1)- \Psi_{\sigma,M,\beta}(\rho_2))\right)\cdot\nabla w(y)
    \,dy\,ds\\
    =-\frac{1}{2}\int_0^t\int_U |\Psi_{\sigma,M,\beta}(\rho_1)- \Psi_{\sigma,M,\beta}(\rho_2)|\nabla\cdot(F_1\nabla w(y))
    \,dy\,ds.
\end{multline}

Regarding the final term in \eqref{eq: epsilon and delta limit of weight term},
it follows from the seventh point of Assumption \ref{asm: assumptions on non-linear coeffs for contraction estime} that $\sigma\sigma'\zeta_M\phi_{\beta}$ is non-decreasing, and so
\begin{multline}\label{sigmasigmaprime}
     \frac{1}{2}\lim_{\gamma\to0}\lim_{\gamma'\to0} \int_0^t\int_U \left(\sigma'(\rho_1)\sigma(\rho_1)\zeta_M(\rho_1)\phi_\beta(\rho_1) -\sigma'(\rho_2)\sigma(\rho_2)\zeta_M(\rho_2)\phi_\beta(\rho_2)\right)\\
     \times F_2\cdot\nabla w(y) \iota_{\gamma,\gamma'}(y) \sgn(\rho_2-\rho_1)dy \,ds\\
      =-\frac{1}{2} \int_0^t\int_U \left|\sigma'(\rho_1)\sigma(\rho_1)\zeta_M(\rho_1)\phi_\beta(\rho_1) -\sigma'(\rho_2)\sigma(\rho_2)\zeta_M(\rho_2)\phi_\beta(\rho_2)\right|F_2\cdot\nabla w(y) \, dy\,ds. 
\end{multline}

Next we analyse \eqref{sigmasigmaprime} within different assumption classes. Under Assumption \ref{asm: classic-DK}, $\sigma\sigma'(\xi)=\frac{1}{2}$, thus by the dominated convergence theorem, we observe that 
\begin{align*}
	-\frac{1}{2}\lim_{M\rightarrow\infty,\beta\rightarrow0} \int_0^t\int_U \left|\sigma'(\rho_1)\sigma(\rho_1)\zeta_M(\rho_1)\phi_\beta(\rho_1) -\sigma'(\rho_2)\sigma(\rho_2)\zeta_M(\rho_2)\phi_\beta(\rho_2)\right|F_2\cdot\nabla w(y)  \,dy\,ds=0. 
\end{align*}
This was the case in the informal computation.
On the other hand, under Assumption \ref{asm: assumptions on non-linear coeffs for contraction estime}, \eqref{sigmasigmaprime} will not vanish.
We control \eqref{sigmasigmaprime} by using \eqref{Psisigma}. Precisely, for every $M\in\mathbb{N}$ and $\beta\in(0,1)$, let 
\begin{equation}\label{eq: def of Sigma sigma m beta}
    \Sigma_{\sigma,M,\beta}(\xi)=\sigma(\xi)\sigma'(\xi)\zeta_M(\xi)\phi_{\beta}(\xi),\quad  \xi\geq0.
\end{equation}
We find that the contribution from \eqref{Psisigma} and \eqref{sigmasigmaprime} can be combined and controlled after taking the regularization limits using the bound \eqref{eq: local bound for sigma sigma'} from the final point of Assumption \ref{asm: assumptions on non-linear coeffs for contraction estime},
\begin{multline*}
-\lim_{M\to\infty,\beta\to0}\frac{1}{2} \int_0^t\int_U \left|\Sigma_{\sigma,M,\beta}(\rho_1) -\Sigma_{\sigma,M,\beta}(\rho_2)\right|F_2\cdot\nabla w(y)  \,dy\,ds\\
+\lim_{M\to\infty,\beta\to0}\frac{1}{2}\int_0^t\int_U \left| \Psi_{\sigma,M,\beta}(\rho_1)- \Psi_{\sigma,M,\beta}(\rho_2)\right|\nabla \cdot\left(F_1\nabla w(y)\right)
    \,dy\,ds\\
    =\frac{1}{2} \int_0^t\int_U \left(\left| \Psi_{\sigma}(\rho_1)- \Psi_{\sigma}(\rho_2)\right|-\left|\Sigma_{\sigma}(\rho_1) -\Sigma_{\sigma}(\rho_2)\right|\right)F_2\cdot\nabla w(y)  \,dy\,ds\\
    +\frac{1}{2}\int_0^t\int_U \left| \Psi_{\sigma}(\rho_1)- \Psi_{\sigma}(\rho_2)\right|F_1\Delta w(y)
    \,dy\,ds\\
    \leq \frac{1}{2}(c+1)\int_0^t\int_U \left| \Psi_{\sigma}(\rho_1)- \Psi_{\sigma}(\rho_2)\right||F_2\cdot\nabla w(y)|
    \,dy\,ds+\frac{1}{2}\int_0^t\int_U \left| \Psi_{\sigma}(\rho_1)- \Psi_{\sigma}(\rho_2)\right|F_1\Delta w(y)
    \,dy\,ds.
\end{multline*}
We used the fact that $\Psi_\sigma$ is well defined without the cutoff functions in the porous medium case.
With the help of the construction of the weight function, we have that for every $M>0$ and $\beta\in(0,1)$, 
\begin{align}\label{weight}
F_1\Delta w+\frac{1}{2}(c+1)|F_2\cdot\nabla w|\leq0. 
\end{align}
Therefore we have that
\begin{multline}\label{correction-contribution}
-\lim_{M\to\infty,\beta\to0}\frac{1}{2} \int_0^t\int_U \left|\Sigma_{\sigma,M,\beta}(\rho_1) -\Sigma_{\sigma,M,\beta}(\rho_2)\right|F_2\cdot\nabla w(y)  \,dy\,ds\\
+\lim_{M\to\infty,\beta\to0}\frac{1}{2}\int_0^t\int_U \left| \Psi_{\sigma,M,\beta}(\rho_1)- \Psi_{\sigma,M,\beta}(\rho_2)\right|\nabla \cdot\left(F_1\nabla w(y)\right)
    \,dy\,ds\leq0.
\end{multline}
That is, we showed that the contribution coming from the second and third weight terms is non-negative.
Since we want to prove the theorem for both the porous medium and classical cases, we drop the contribution of this term from the estimate.

Putting everything together, we showed that 
\begin{equation*}
    \lim_{M\to\infty,\beta\to0}\lim_{\gamma\to0}\lim_{\gamma'\to0}\lim_{\epsilon,\delta\to0}I_t^{wgt}\leq\int_0^t\int_U  \left|\Phi(\rho_1) -\Phi(\rho_2)\right|\Delta w(y) \,dy\,ds.
\end{equation*}

\textbf{Flux term.}

The flux term can be handled a similar way to the weight term.
After again taking the $\epsilon,\delta\to 0$ limit followed by taking the $\gamma'\to0$ and $\gamma\to0$ limits, we get

\begin{multline}
   \lim_{\gamma\to0} \lim_{\gamma'\to0}\lim_{\epsilon,\delta\to0}I_t^{flux}
    \\ =\int_0^t\int_U     \left(\nabla\cdot\nu(\rho_1)\zeta_M(\rho_1)\phi_\beta(\rho_1)-\nabla\cdot\nu(\rho_2)\zeta_M(\rho_2)\phi_\beta(\rho_2)\right)w(x) \sgn(\rho_2-\rho_1)\,dy\,ds.
\end{multline}
Again repeating the computation of the weight term, but first taking the limits of the cutoff functions,
and consequently using the first point of Assumption \ref{asm: Assumptions on nu} that $\nu=(\nu_1,\hdots,\nu_d)$ is a non-decreasing function in each component, and integration by parts after smoothing the sign function analogous to the computation \eqref{Psisigma} gives
\begin{multline*}
\lim_{M\to\infty,\beta\to0}\int_0^t\int_U     \left(\nabla\cdot\nu(\rho_1)\zeta_M(\rho_1)\phi_\beta(\rho_1)-\nabla\cdot\nu(\rho_2)\zeta_M(\rho_2)\phi_\beta(\rho_2)\right)w(x) \sgn(\rho_2-\rho_1)\,dy\,ds\\
= \lim_{\delta\to0}\int_0^t \int_{U}\left(\nabla\cdot\nu(\rho_1)-\nabla\cdot\nu(\rho_2)\right)w(y) \sgn_\delta(\rho_2-\rho_1)\,dy\,ds\\
=-\sum_{i=1}^d\int_0^t \int_{U}\partial_{x_i}|\nu_i(\rho_1)-\nu_i(\rho_2)|w(y)\,dy\,ds  =\sum_{i=1}^d\int_0^t \int_{U}\partial_{x_i}w(y)|\nu_i(\rho_1)-\nu_i(\rho_2)|\,dy\,ds.
\end{multline*}
We showed that
\begin{equation*}
    \lim_{M\to\infty,\beta\to0}\lim_{\gamma\to0}\lim_{\gamma'\to0}\lim_{\epsilon,\delta\to0}I_t^{flux}=\sum_{i=1}^d\int_0^t \int_{U}\partial_{x_i}w(y)|\nu_i(\rho_1)-\nu_i(\rho_2)|\,dy\,ds.
\end{equation*}

\textbf{Cutoff term.}

The cutoff term consists of the terms involving derivatives of the cutoff functions.
It is
\begin{align}\label{eq: cutoff term in super contraction proof}
    &I_t^{cut}= \int_{\mathbb{R}^2}\int_0^t\int_{U^2}\kappa^{\epsilon,\delta}(y,x,\eta,\xi)\partial_\eta(\zeta_M(\eta)\phi_\beta(\eta)) \iota_{\gamma,\gamma'}(y)w(y) (-1+2\chi^{\epsilon,\delta}_{s,2})\,dq_1(x, \xi, s)\,dy\,d\eta\nonumber\\
    &+\frac{1}{2}\int_{\mathbb{R}}\int_0^t \int_{U^2}\left(
    \sigma'(\rho_1)\sigma(\rho_1)\nabla\rho_1\cdot F_2 + \sigma(\rho_1)^2F_3 \right)
    \bar{k}^{\epsilon,\delta}_{s,1}(1-2\chi^{\epsilon,\delta}_{s,2})\nonumber\\
    &\hspace{150pt}\times\partial_\eta(\zeta_M(\eta)\phi_\beta(\eta)) \iota_{\gamma,\gamma'}(y)w(y)\,dy\,dx\,ds\,d\eta\nonumber\\
    &+\int_{\mathbb{R}^2}\int_0^t\int_{U^2}\kappa^{\epsilon,\delta}(y,x',\eta,\xi)\partial_\eta(\zeta_M(\eta)\phi_\beta(\eta)) \iota_{\gamma,\gamma'}(y)w(y)(-1+2\chi^{\epsilon,\delta}_{s,1})\,dq_2(x', \xi, s)\,dy\,d\eta\nonumber\\
    &+\frac{1}{2}\int_{\mathbb{R}}\int_0^t \int_{U^2}\left( \sigma'(\rho_2)\sigma(\rho_2)\nabla\rho_2\cdot F_2 + \sigma(\rho_2)^2F_3 \right)\bar{k}^{\epsilon,\delta}_{s,2}(1-2\chi^{\epsilon,\delta}_{s,1})\nonumber\\
    &\hspace{150pt}\times\partial_\eta(\zeta_M(\eta)\phi_\beta(\eta)) \iota_{\gamma,\gamma'}(y)w(y)\,dy\,dx\,ds\,d\eta\nonumber\\
    &+\int_{\mathbb{R}}\int_0^t \int_{U^2}\left(\Phi'(\rho_1)\nabla\rho_1+\frac{1}{2}F_1[\sigma'(\rho_1)]^2\nabla\rho_1 +\frac{1}{2} \sigma'(\rho_1)\sigma(\rho_1)F_2\right)\cdot\nabla\iota_{\gamma,\gamma'}(y)\nonumber\\
    &\hspace{150pt}\times \bar{k}^{\epsilon,\delta}_{s,1}\zeta_M(\eta)\phi_\beta(\eta)w(y) (-1+2\chi^{\epsilon,\delta}_{s,2})\,dy\,dx\,ds\,d\eta\nonumber\\
    &+\int_{\mathbb{R}}\int_0^t \int_{U^2}\left(\Phi'(\rho_2)\nabla\rho_2 +\frac{1}{2}F_1[\sigma'(\rho_2)]^2\nabla\rho_2+\frac{1}{2} \sigma'(\rho_2)\sigma(\rho_2)F_2\right)\cdot\nabla\iota_{\gamma,\gamma'}(y)\nonumber\\
    & \hspace{150pt}\times\bar{k}^{\epsilon,\delta}_{s,2}\zeta_M(\eta)\phi_\beta(\eta)  w(y)(-1+2\chi^{\epsilon,\delta}_{s,1})\,dy\,dx'\,ds\,d\eta.    
\end{align}
We emphasise that this term is handled in a similar way to the corresponding cutoff term in the uniqueness proof stated in Theorem \ref{thm: standard well posedness thm from Popat25}, where we just use that the additional weight function that is present in the integrands is bounded. 
Here we will make explicit the dependence of the estimate on the compact support parameter $\gamma'$.
We provide brief intuition illustrating why each of the terms is bounded, but refer to \cite{Shyam25} for the full details.

For the integrals in the first and third lines of \eqref{eq: cutoff term in super contraction proof} involving the kinetic measures, taking the derivative of the cutoff functions and then taking the limits in $\epsilon,\delta$, we have that there exists a constant $c\in(0,\infty)$ such that
\begin{multline*}
   \limsup_{\epsilon,\delta\to0}\left|\int_{\mathbb{R}^2}\int_0^t\int_{U^2}\kappa^{\epsilon,\delta}(y,x,\eta,\xi)\partial_\eta(\zeta_M(\eta)\phi_\beta(\eta)) \iota_{\gamma,\gamma'}(y)w(y) (-1+2\chi^{\epsilon,\delta}_{s,2})\,dq_1(x, \xi, s)\,dy\,d\eta\right|\\
   \leq c\left(\beta^{-1}q_1(U\times[\beta/2,\beta]\times[0,T])+q_1(U\times[M,M+1]\times[0,T])\right).
\end{multline*}
In the above we used the pointwise boundedness of the weight and cutoff functions.
The bound is then clearly independent of $\gamma'$ and $\gamma$ so that the $\gamma,\gamma'\to0$ limits can be taken.
An analogous inequality holds for the term involving the kinetic measure $q_2$.
The decay of the kinetic measures at infinity is a consequence of \eqref{eq: decay of kinetic measure at infinity} in the definition of kinetic solutions, and the decay at zero is precisely proved in Proposition 4.29 of the work of the first author \cite{Shyam25}.
By Fatou's lemma, it then follows that there almost surely exists subsequences $\beta\to 0$ and $M\to\infty$ such that
\begin{equation*}
    \lim_{M\to\infty}\lim_{\beta\to0}\left(\beta^{-1}q_1(U\times[\beta/2,\beta]\times[0,T])+q_1(U\times[M,M+1]\times[0,T])\right)=0.
\end{equation*}

We next deal with the terms in the second and fourth lines of \eqref{eq: cutoff term in super contraction proof}.
Firstly, to bound the terms involving $F_2$, analogously to weight and flux terms above, in the $\epsilon,\delta\to0$ limit we use the fact that $\sgn$ is an odd function and the  product rule to evaluate the derivative of the cutoffs to get

\begin{multline}\label{eq: long expression for old cutoff terms}
      \lim_{\epsilon,\delta\to0}\left(\frac{1}{2}\int_{\mathbb{R}}\int_0^t \int_{U^2}
    \sigma'(\rho_1)\sigma(\rho_1)\nabla\rho_1\cdot F_2
    \bar{k}^{\epsilon,\delta}_{s,1}(1-2\chi^{\epsilon,\delta}_{s,2})\partial_\eta(\zeta_M(\eta)\phi_\beta(\eta)) \,\iota_{\gamma,\gamma'}(y)w(y)\,dy\,dx\,ds\,d\eta\right.\\
    \left.+\frac{1}{2}\int_{\mathbb{R}}\int_0^t \int_{U^2} \sigma'(\rho_2)\sigma(\rho_2)\nabla\rho_2\cdot F_2 \bar{k}^{\epsilon,\delta}_{s,2}(1-2\chi^{\epsilon,\delta}_{s,1})\partial_\eta(\zeta_M(\eta)\phi_\beta(\eta)) \,\iota_{\gamma,\gamma'}(y)w(y)\,dy\,dx\,ds\,d\eta\right)\\
    =\frac{1}{4}\int_0^t\int_U\left(\nabla\sigma^2(\rho_2)-\nabla\sigma^2(\rho_1)\right)\cdot F_2 \left(\mathbbm{1}_{M<\rho<M+1}+\beta^{-1}\mathbbm{1}_{\beta/2<\rho<\beta}\right) \sgn(\rho_2-\rho_1)\,\iota_{\gamma,\gamma'}(y)w(y)\,dy\,ds.
    \end{multline}
The term involving the indicator in $M$ in the final integral of \eqref{eq: long expression for old cutoff terms} is handled in a similar way to the term involving $\beta$, so we just illustrate the computation for the $\beta$ term.
For convenience introduce the shorthand notation for $i=1,2$, $(x,t)\in U\times[0,T]$,
\begin{equation*}
    \rho_{i,\beta}(x,t):=\left(\rho_i(x,t)\vee\beta\right)\wedge\beta/2.
\end{equation*} 
The indicator functions allows us to re-write the $\beta$-term in the final line as
    \begin{align}\label{eq: explicit terms expanding the indicator functions}
    \frac{1}{4\beta}\int_0^t\int_U\nabla\left(\left(\sigma^2(\rho_{2,\beta})-\sigma^2(\beta/2)\right)-\left(\sigma^2(\rho_{1,\beta})-\sigma^2(\beta/2)\right)\right)
   \cdot F_2  \sgn(\rho_2-\rho_1)\iota_{\gamma,\gamma'}(y)w(y)\,dy\,ds.
\end{align}
Note that $\sigma^2$ is not necessarily increasing, so we cannot use that $\sgn(\rho_2-\rho_1)=\sgn(\sigma^2(\rho_2)-\sigma^2(\rho_1))$ as in \eqref{eq: computation for fx-fy times sgn x-y}.
Instead we smooth out the sign function as we did above and integrate by parts, which gives for the term involving $\rho_2$ (the term involving $\rho_1$ is handled analogously) in the final line of \eqref{eq: explicit terms expanding the indicator functions} that
\begin{align}\label{eq: decomposition of sigma 2 when integrating by parts}
    &\frac{1}{4\beta}\int_0^t\int_U\nabla\left(\sigma^2(\rho_{2,\beta})-\sigma^2(\beta/2)\right)\cdot F_2  \sgn(\rho_2-\rho_1)\iota_{\gamma,\gamma'}(y) w(y)\,dy\,ds\nonumber\\
    &=-\lim_{\delta\to0}\frac{1}{4\beta}\int_0^t\int_U\left(\sigma^2(\rho_{2,\beta})-\sigma^2(\beta/2)\right)\nabla\cdot F_2  \sgn_\delta(\rho_2-\rho_1)\iota_{\gamma,\gamma'}(y) w(y)\,dy\,ds\nonumber\\
    &\hspace{20pt}-\lim_{\delta\to0}\frac{1}{4\beta}\int_0^t\int_U\left(\sigma^2(\rho_{2,\beta})-\sigma^2(\beta/2)\right) \nabla(\rho_2-\rho_1)\cdot F_2  (\sgn_\delta)'(\rho_2-\rho_1)\iota_{\gamma,\gamma'}(y) w(y)\,dy\,ds\nonumber\\
    &\hspace{20pt}-\lim_{\delta\to0}\frac{1}{4\beta}\int_0^t\int_U\left(\sigma^2(\rho_{2,\beta})-\sigma^2(\beta/2)\right)\nabla(\iota_{\gamma,\gamma'}(y)w(y))\cdot F_2\, \sgn_\delta(\rho_2-\rho_1)\,dy\,ds.
\end{align}
For the first two terms we can directly take the $\gamma'\to0$ followed by  $\gamma\to0$ limits.
The first term is shown to converge to zero using seventh point in Assumption \ref{asm: assumptions on non-linear coeffs for contraction estime}, which tells us that either $\nabla\cdot F_2=0$ and so the integral vanishes immediately, or we can use the fundamental theorem of calculus and the assumption that $\sigma\sigma'(0)=0$ to get
\begin{multline}\label{eq: how to handle sigma2 term}
     -\lim_{\beta\to0}\lim_{\gamma\to0}\lim_{\gamma'\to0}\lim_{\delta\to0}\frac{1}{4\beta}\int_0^t\int_U\left(\sigma^2(\rho_{2,\beta})-\sigma^2(\beta/2)\right)\nabla\cdot F_2  \sgn_\delta(\rho_2-\rho_1)\iota_{\gamma,\gamma'}(y) w(y)\,dy\,ds\\
     =-\frac{1}{2}\int_0^t\int_U(\sigma\sigma')(0)\mathbbm{1}_{\rho>0}\nabla\cdot F_2  \sgn(\rho_2-\rho_1)=0.
\end{multline}
The corresponding terms involving $M$ in the final line of \eqref{eq: long expression for old cutoff terms} are treated similarly, but we use the assumption on the oscillations of $\sigma^2$ at infinity \eqref{eq: oscillations of sigma2 at infinity}.\\
The term in the middle line also vanishes by using the fact that
\begin{equation*}
    \left(\sigma^2(\rho_{2,\beta})-\sigma^2(\beta/2)\right) (\sgn_\delta)'(\rho_2-\rho_1)\leq c\mathbbm{1}_{\{0\leq |\rho_1-\rho_2|\leq c\delta \,\,\text{and}\,\, \beta/2\leq\rho_2\leq\beta\}}.
\end{equation*}
The only fundamentally new term to handle is the final term of equation \eqref{eq: decomposition of sigma 2 when integrating by parts} involving the gradient of spatial cutoff and weight function.
When the gradient hits the spatial cutoff, this term is treated by grouping the $\rho_2$ and $\rho_1$ terms and following the same computation from equation \eqref{eq: first equation for bounding cutoff term that does not have a gradient} to \eqref{eq: final equation for bounding cutoff term that does not have a gradient} below.
We therefore reserve the analysis for then and will not repeat it here.
The new term arising when the gradient hits the weight function is
\begin{equation*}
    -\frac{1}{4\beta}\int_0^t\int_U\left(\sigma^2(\rho_{2,\beta})-\sigma^2(\beta/2)\right)\iota_{\gamma,\gamma'}(y)\nabla w(y)\cdot F_2\, \sgn(\rho_2-\rho_1)\,dy\,ds.
\end{equation*}
This term is treated in the same way as equation as \eqref{eq: how to handle sigma2 term} above, and vanishes in the $\beta\to0$ limit.
We further mention that the terms involving $F_3$ in the second and fourth lines can be bounded in a similar way as described above, using the $L^2_x$-integrability of $\sigma(\rho_i)$ and the boundedness of $F_3$.

Let us conclude by bounding the final two integrals of the cutoff term \eqref{eq: cutoff term in super contraction proof} above, comprising of the terms involving gradients of the spatial cutoff. 
After taking $\epsilon,\delta\to 0$ limits, we are left to handle
\begin{align}\label{new terms in uniqueness}
    &\int_0^t \int_{U}\left(\Phi'(\rho_1)\nabla\rho_1+\frac{1}{2}F_1[\sigma'(\rho_1)]^2\nabla\rho_1 +\frac{1}{2} \sigma'(\rho_1)\sigma(\rho_1)F_2\right)\zeta_M(\rho_1)\phi_\beta(\rho_1)w \cdot\nabla\iota_{\gamma,\gamma'}\sgn(\rho_2-\rho_1)dyds\nonumber\\
    &+\int_0^t \int_{U}\left(\Phi'(\rho_2)\nabla\rho_2 +\frac{1}{2}F_1[\sigma'(\rho_2)]^2\nabla\rho_2+\frac{1}{2} \sigma'(\rho_2)\sigma(\rho_2)F_2\right)\zeta_M(\rho_2)\phi_\beta(\rho_2)w \cdot\nabla\iota_{\gamma,\gamma'}\sgn(\rho_1-\rho_2)dyds.
\end{align}
We combine the terms in the two lines using the fact that $\sgn$ is an odd function.

We will deal with the terms that have a factor of $\nabla\rho_i$ and the final term which does not separately.
For the first terms of \eqref{new terms in uniqueness}, using the definition $\Phi_{M,\beta}$ from \eqref{eq: definition of Phi m beta} with Lemma \ref{remark derivative of spatial cutoff} to define the spatial derivative of the cutoff, the notation $v_{y,\gamma'}:=\frac{y-y^*_{\gamma'}}{|y-y^*_{\gamma'}|}$ for the inward pointing unit normal at the boundary $\partial U_{\gamma'}$ and the fundamental theorem of calculus, we show that the contribution coming from the difference of the first terms of \eqref{new terms in uniqueness} is non-negative

\begin{align}\label{equation bringing sgn into derivative for increasing function}
     &-\int_0^t \int_{U}\left(\nabla\Phi_{M,\beta}(\rho_2)-\nabla\Phi_{M,\beta}(\rho_1)\right) \cdot\nabla\iota_{\gamma,\gamma'}(y)\,w(y) \sgn(\rho_2-\rho_1)\,dy\,ds\nonumber\\
     &=-(\gamma-\gamma')^{-1}\int_0^t \int_{U_{\gamma'}\setminus U_\gamma}\nabla|\Phi_{M,\beta}(\rho_2)-\Phi_{M,\beta}(\rho_1)| \cdot v_{y,\gamma'} w(y)\,dy\,ds\nonumber\\
     &=-(\gamma-\gamma')^{-1}\int_0^t \int_{\gamma'}^\gamma \int_{\partial U_{z}}\nabla|\Phi_{M,\beta}(\rho_2(y^*+zv_{y,\gamma'},s))-\Phi_{M,\beta}(\rho_1(y^*+zv_{y,\gamma'},s))| \cdot v_{y,\gamma'} w(y)\,dS(y)\,dz\,ds\nonumber\\
     &=-(\gamma-\gamma')^{-1}\int_0^t \int_{\gamma'}^\gamma \int_{\partial U_{z}}\frac{\partial}{\partial z}|\Phi_{M,\beta}(\rho_2(y^*+zv_{y,\gamma'},s))-\Phi_{M,\beta}(\rho_1(y^*+zv_{y,\gamma'},s))|w(y)\,dS(y)\,dz\,ds\nonumber\\
     &=(\gamma-\gamma')^{-1}\int_0^t \int_{\partial U_{\gamma'}}|\Phi_{M,\beta}(\rho_2)-\Phi_{M,\beta}(\rho_1)|w(y)\,dS(y)\,ds\nonumber\\
     &\hspace{100pt}- (\gamma-\gamma')^{-1}\int_0^t \int_{\partial U_\gamma}|\Phi_{M,\beta}(\rho_2)-\Phi_{M,\beta}(\rho_1)|w(y)\,dS(y)\,ds.
     \end{align}
Sending $\gamma'\to0$, it follows from the local regularity of solutions \eqref{eq: local regularity property of stochastic kinetic solution} and the trace theorem, see Chapter 5.5 of Evans \cite{evans2022partial}, the fact that trace zero $W^{1,1}(U)$-functions can be approximated by smooth compactly supported functions and the fact that solutions coincide on the boundary, that 
$$
\lim_{\gamma'\to0}(\gamma-\gamma')^{-1}\int_0^t \int_{\partial U_{\gamma'}}|\Phi_{M,\beta}(\rho_2)-\Phi_{M,\beta}(\rho_1)|w(y)\,dS(y)\,ds=0. 
$$
The final term in \eqref{equation bringing sgn into derivative for increasing function} is signed for every $\gamma\in(0,\gamma')$ so can be removed from the estimate.

By repeating the same arguments, noting that $\frac{1}{2}F_1(\sigma'(\rho_2))^2\geq 0$, one can conclude that the combination of second terms of \eqref{new terms in uniqueness} are non-positive for every $\gamma>0$. 
Hence the contribution of these terms can be neglected from the estimate.

For the final term of \eqref{new terms in uniqueness}, we have by Lemma \ref{remark derivative of spatial cutoff} as well as by the boundedness of $F_2$, $\sgn$ and the weight function $w$,
\begin{multline}\label{eq: first equation for bounding cutoff term that does not have a gradient}
    \frac{1}{2} \int_0^t \int_{U} \left(\sigma'(\rho_2)\sigma(\rho_2)\zeta_M(\rho_2)\phi_\beta(\rho_2) - \sigma'(\rho_1)\sigma(\rho_1)\zeta_M(\rho_1)\phi_\beta(\rho_1) \right) \sgn(\rho_2-\rho_1)w(y)F_2\cdot\nabla\iota_{\gamma,\gamma'}(y)\,dy\,ds\\
    \leq c (\gamma-\gamma')^{-1} \int_0^t \int_{U} \left|\sigma'(\rho_2)\sigma(\rho_2)\zeta_M(\rho_2)\phi_\beta(\rho_2) - \sigma'(\rho_1)\sigma(\rho_1)\zeta_M(\rho_1)\phi_\beta(\rho_1) \right|\mathbbm{1}_{U_{\gamma'} \setminus U_{\gamma}}(y)\,dy\,ds.
\end{multline}

Here we can directly take the limit $\gamma'\to0$ to obtain
\begin{equation}\label{eq: taking gamma' limit in final cutoff term}
    c \gamma^{-1} \int_0^t \int_{U} \left|\sigma'(\rho_2)\sigma(\rho_2)\zeta_M(\rho_2)\phi_\beta(\rho_2) - \sigma'(\rho_1)\sigma(\rho_1)\zeta_M(\rho_1)\phi_\beta(\rho_1) \right|\mathbbm{1}_{U\setminus U_{\gamma}}(y)\,dy\,ds
\end{equation}

For every $y\in U \setminus U_\gamma$, let $y^*=y^*(y)$ denote the unique closest point on the boundary $\partial U$ to $y$. 
For $i=1,2$ and $y^*\in \partial U$, introduce the notation 
\begin{equation*}
    \rho_{i}(y^*,s):=(\Phi^{-1}(\rho_b(y^*))
\end{equation*}
and
\begin{equation*}
    \rho_{i,M,\beta}(y^*,s):=(\rho_i(y^*,s)\vee\beta/2)\wedge (M+1)=(\Phi^{-1}(\rho_b(y^*))\vee\beta/2)\wedge (M+1).
\end{equation*}

Using the fact that solutions coincide on the boundary, subtracting the term\\ $\sigma'(\rho_2(y^*,s))\sigma(\rho_2(y^*,s))\zeta_M(\rho_2(y^*,s))\phi_\beta(\rho_2(y^*,s))$ and adding the same function but evaluated at $\rho_1$, using the triangle inequality and Lipschitz property of $\sigma'\sigma\zeta_M\phi_\beta$ which holds due to the cutoff functions $\zeta_M\phi_\beta$, we have that \eqref{eq: taking gamma' limit in final cutoff term} is bounded by
\begin{align}\label{final cutoff term in uniqueness}
    c \gamma^{-1}
    \int_0^t \int_{U\setminus U_{\gamma}} (|\rho_{2,M,\beta}(y,s) - \rho_{2,M,\beta}(y^*,s)|+|\rho_{1,M,\beta}(y^*,s)-\rho_{1,M,\beta}(y,s)|)\,dy\,ds.
\end{align}

Both of the above terms are controlled in the same way.
For $i=1,2$, fixed distance from the boundary $\tilde{\gamma}\in(0,\gamma)$ and fixed time $s\in[0,t]$, it follows by the fundamental theorem of calculus and the Cauchy--Schwartz inequality that
\begin{multline}
    \int_{\partial U_{\tilde\gamma}}|\rho_{i,M,\beta}(y,s)-\rho_{i,M,\beta}(y^*,s)|\,dS(y)=\int_{\partial U_{\tilde\gamma}}\left|\int_0^{1}\nabla\rho_{i,M,\beta}\left(y^*+zv_y,s\right)\cdot v_y\,dz \right| \,dS(y)\\
    \leq\int_{{U}\setminus U_{\tilde\gamma}}|\nabla\rho_{i,M,\beta}(y,s)|\,dy
    \leq |U\setminus U_{\tilde\gamma}|^{1/2}\,\|\nabla\rho_{i,M,\beta}\|_{L^2(U\setminus U_{\tilde\gamma})}\leq \tilde\gamma^{1/2} \,\|\nabla\rho_{i,M,\beta}\|_{L^2_x},
\end{multline}
where in the final inequality we made the norm of $\nabla\rho_{i,M,\beta}$ independent of $\tilde\gamma$.
We therefore bound the terms of \eqref{final cutoff term in uniqueness} by
\begin{multline}\label{eq: final equation for bounding cutoff term that does not have a gradient}
         \gamma^{-1}  \int_0^t \int_{U\setminus U_{\gamma}} |\rho_{i,M,\beta}(y,s) - \rho_{i,M,\beta}(y^*)|\,dy\,ds\\
        = \gamma^{-1}  \int_0^t \int_{0}^\gamma \int_{\partial U_{\tilde\gamma}} |\rho_{i,M,\beta}(y,s) - \rho_{i,M,\beta}(y^*)|\,dS(y)\,d\tilde{\gamma}\,ds\\
        \leq \gamma^{-1}  \int_0^t \|\nabla\rho_{i,M,\beta}\|_{L^2_x}\left(\int_0^\gamma  \tilde{\gamma}^{1/2}\,d\tilde\gamma\right)\,ds\\
        \leq  \gamma^{1/2}\|\nabla\rho_{i,M,\beta}\|_{L^1([0,t];L^2_x)},
\end{multline}
which converges to $0$ in the $\gamma\to0$ limit for fixed $M,\beta$.
Therefore conclude that final two lines of the cutoff \eqref{eq: cutoff term in super contraction proof} consisting of the new terms are non-positive in the $\gamma\to0$ limit.

Putting \eqref{eq: cutoff term in super contraction proof} and subsequent computations together, we conclude
\begin{equation*}
    \lim_{M\to\infty,\beta\to0}\lim_{\gamma\to0}\lim_{\gamma'\to0}\lim_{\epsilon,\delta\to0}I_t^{cut}\leq0.
\end{equation*}

\textbf{Martingale term.}

By analogous computations to the weight and flux terms, we have as $\epsilon,\delta\to0$ that
\begin{multline*}
    \lim_{\epsilon,\delta\to0}I_t^{mart}\\
    =\int_0^t\int_U \sgn(\rho_2-\rho_1) \iota_{\gamma,\gamma'}(y)w(y)\left(\zeta_M(\rho_1)\phi_\beta(\rho_1)\nabla\cdot\left(\sigma(\rho_1)\,d\xi^F\right)-\zeta_M(\rho_2)\phi_\beta(\rho_2)\nabla\cdot\left(\sigma(\rho_2)\,d\xi^F\right)\right)\,dy.
\end{multline*}
Just as in the formal proof, we realise that we are only aiming to prove the contraction in expectation.
Since the martingale term is a true martingale in the sense that the integrand at time $s$ is $\mathcal{F}_s$-measurable, it vanishes in expectation
\begin{equation*}
    \mathbb{E}I_t^{mart}=0,
\end{equation*}
and no further analysis is necessary for this term.

\textbf{Conclusion.}

Putting everything together, from equation \eqref{equation for difference of two solutions in uniqueness proof} we illustrated that the expected difference of two stochastic kinetic solutions $\rho_1$ and $\rho_2$ in the weighted $L^1_{w;x}$-norm can be upper bounded
\begin{multline*}
    \mathbb{E}\left.\int_\mathbb{R}\int_{U}|\chi^1_s-\chi^2_s|^2w\right|_{s=0}^t=\lim_{\beta\to0,M\to\infty}\lim_{\gamma\to0}\lim_{\gamma'\to0}\lim_{\epsilon,\delta\to0}\mathbb{E}\left.\int_\mathbb{R}\int_{U}|\chi^{\epsilon,\delta}_{s,1}-\chi^{\epsilon,\delta}_{s,2}|^2w\phi_\beta\zeta_M\iota_{\gamma,\gamma'}\right|_{s=0}^t\\
    =\lim_{\beta\to0,M\to\infty}\lim_{\gamma\to0}\lim_{\gamma'\to0}\lim_{\epsilon,\delta\to0}\mathbb{E}\left( I_t^{err} + I_t^{meas} + I_t^{mart} + I_t^{cut}+ I_t^{flux}+I_t^{bound}\right)\\
    \leq \mathbb{E}\int_0^t\int_U  \left|\Phi(\rho_1) -\Phi(\rho_2)\right|\Delta w(y) \,dy\,ds +\mathbb{E}\sum_{i=1}^d\int_0^t \int_{U}\partial_{x_i}w(y)|\nu_i(\rho_1)-\nu_i(\rho_2)|\,dy\,ds.
\end{multline*}
Compared to \eqref{eq: requirements of weight function in the formal proof}, here the weight function instead needs to satisfy
\begin{equation*}
    \begin{cases}
        \Delta w <0 \\
        
        F_1\Delta w+\frac{1}{2}(c+1)|F_2\cdot\nabla w|<0,\\
        \partial_{x_i}w(x)<0\quad \text{for every } i=1,\dots,d.
    \end{cases}
\end{equation*}
We realise that similarly to Proposition \ref{prop: choice of weight function in informal proof}, we can choose the weight function satisfying
\begin{align*}
w(x)=-\exp(\alpha x\cdot e)+C
\end{align*}
where the constant $C$ is again chosen large enough so that the weight is non-negative, and $\alpha$ is modified to be sufficiently large, depending on the constant $c$.

The result then follows by repeating the computation in the formal proof, from \eqref{eq: simplifying the super contraction estimate by using def of weight} to \eqref{eq: choice of constant}.
\end{proof}

The above theorem suggests that the super contraction stems from the contribution of the non-linear functions $\Phi$ and $\nu$.
In the case that $\sigma$ is sufficiently regular, for example the porous medium case $\sigma(\xi)=\xi^m, m>1$, we show that we get a contribution from the Stratonovich-to-It\^o terms too.
To see this, we recall that the limiting function $\Psi_{\sigma}:\mathbb{R}\to\mathbb{R}$ is defined in \eqref{eq: def of Psi sigma} by
\begin{equation}
    \Psi_{\sigma}(\xi)=\int^{\xi}_0[\sigma'(\xi')]^2d\xi',\quad \xi\geq 0.
\end{equation}
\begin{corollary}\label{corrolary: improved super contraction in porous medium case}
    Suppose that the assumptions from Theorem \ref{thm: L1 omega super contraction} are satisfied.
    Suppose additionally that $\Psi_{\sigma}$  as in \eqref{eq: def of Psi sigma} is well-defined and satisfies \begin{equation}\label{eq: Psi sigma in L1 condition}
        \Psi_\sigma(\rho_i)\in L^1(U\times[0,T])\quad \text{for i=1,2}.
    \end{equation}
Then the super contraction from Theorem \ref{thm: L1 omega super contraction} can be strengthened to 
\begin{equation}\label{enhanced-super-contraction}
    \mathbb{E}\|\rho_1(\cdot,t)-\rho_2(\cdot,t)\|_{L^1_{w;x}}\leq \mathbb{E}\|\rho_{1,0}-\rho_{2,0}\|_{L^1_{w;x}}-c\mathbb{E}\int_0^t\int_U\tilde{\mathcal{A}}(\rho_1(y,s),\rho_2(y,s))\,dy\,ds
\end{equation}
with
\begin{equation}\label{eq: equation for tilde A}
    \tilde{\mathcal{A}}(\rho_1,\rho_2):=\mathcal{A}(\rho_1,\rho_2)
    +\frac{1}{2} \left| \Psi_{\sigma}(\rho_1)- \Psi_{\sigma}(\rho_2)\right|\left(F_1\Delta w+(c+1)|F_2\cdot\nabla w|\right).
\end{equation}

\end{corollary}
\begin{proof}
We now return to the proof of Theorem~\ref{thm: L1 omega super contraction}. Due to a potential lack of regularity of $\Psi_\sigma$ in the classical Dean--Kawasaki case (where $\sigma(\xi)=\sqrt{\xi}$), the negative contribution of \eqref{correction-contribution} must be discarded before the truncations are removed. However, when the diffusion coefficient $\sigma$ is sufficiently regular so that condition \eqref{eq: Psi sigma in L1 condition} holds, we may apply the dominated convergence theorem to \eqref{correction-contribution} to obtain
\begin{align*}
	\limsup_{M\to\infty,\ \beta\to0}&\Bigg[-\frac{1}{2} \int_0^t\!\!\int_U 
	\left|\Sigma_{\sigma,M,\beta}(\rho_1) - \Sigma_{\sigma,M,\beta}(\rho_2)\right|
	F_2\cdot\nabla w(y)\,dy\,ds \notag\\
	&\quad + \frac{1}{2}\int_0^t\!\!\int_U 
	\left| \Psi_{\sigma,M,\beta}(\rho_1) - \Psi_{\sigma,M,\beta}(\rho_2)\right|
	\nabla\cdot\!\left(F_1\nabla w(y)\right)\,dy\,ds \Bigg] \\
	&\leq c\int_0^t\!\!\int_U 
	\frac{1}{2}\left| \Psi_{\sigma}(\rho_1) - \Psi_{\sigma}(\rho_2)\right|
	\left(F_1\Delta w + (c+1)\lvert F_2\cdot\nabla w\rvert\right)\,dy\,ds.
\end{align*}
Therefore, the non-negative contribution in \eqref{correction-contribution} can be retained in the estimate. 
Following the remaining steps in the proof of Theorem~\ref{thm: L1 omega super contraction}, we complete the argument.
\end{proof}

To conclude the section on super contraction estimates, we prove one final estimate. We show that by imposing an analogue of linear structure for $\Psi_\sigma$, we can obtain an improved rate of convergence.
The result follows directly from Corollary \ref{corrolary: improved super contraction in porous medium case}.

\begin{corollary}\label{coro: L1 omega super contraction}
Under the same assumptions as Theorem \ref{thm: L1 omega super contraction}, suppose additionally that $\Psi_{\sigma}$ defined in \eqref{eq: def of Psi sigma} is well-defined and satisfies the lower Lipschitz bound
\begin{align}\label{Lip-Psisigma}
c|\xi_1-\xi_2|\leq|\Psi_{\sigma}(\xi_1)-\Psi_{\sigma}(\xi_2)|, 
\end{align}
for some $c>0$, uniformly for every $\xi_1,\xi_2\geq0$.
There exists a constant $c\in(0,\infty)$ such that we have for every $t\geq0$
\begin{multline*}
    \mathbb{E}\|\rho_1(\cdot,t)-\rho_2(\cdot,t)\|_{L^1_{w;x}}\leq \mathbb{E}\|\rho_{1,0}-\rho_{2,0}\|_{L^1_{w;x}}\\
    -c\mathbb{E}\int_0^t\int_U\mathcal{A}(\rho_1(y,s),\rho_2(y,s))\,dy\,ds-c\mathbb{E}\int^t_0\int_U|\rho_1(y,s)-\rho_2(y,s)|\,dy\,ds.
\end{multline*}
\end{corollary}
\begin{proof}
	Under the additional upper bound assumption in \eqref{Lip-Psisigma}, and using the $L^1_x$-boundedness of both solutions $\rho_1$ and $\rho_2$, we obtain
$$
\Psi_\sigma(\rho_i) \in L^1(U \times [0,T]) \quad \text{for } i=1,2.
$$
Consequently, the estimate \eqref{enhanced-super-contraction} in Corollary~\ref{corrolary: improved super contraction in porous medium case} applies. Moreover, by invoking the lower bound in \eqref{Lip-Psisigma} together with the fact that $F_1 \Delta w + (c+1)\lvert F_2 \cdot \nabla w \rvert$ admits a positive lower bound, we conclude the proof.

\end{proof}

\section{Ergodicity}\label{sec: ergodicity}
For the ergodicity results we need the below technical result, Lemma B.2 from Dareiotis, Gess and Tsatsoulis \cite{DGT20}. \begin{lemma}\label{lemma: lemma B.2 from DGT}
    Let $f,h:\mathbb{R}\to\mathbb{R}$ be continuous functions.
    Suppose that there exists constants $c\in(0,\infty)$ such that for every $0<s\leq t$ and $m\geq1$,
    \begin{equation*}
        \begin{cases}
            f(t)-f(s)\leq -c\int_s^t|f(r)|^{m-1}f(r)\,dr\\
            h(t)-h(s)=-c\int_s^t|h(r)|^{m-1}h(r)\,dr
        \end{cases}
    \end{equation*}
    If $f$ is upper bounded by $h$ at the initial time $f(0)\leq h(0)$, it follows $f$ remains upper bounded by $h$, $f(t)\leq h(t)$, for every $t\geq 0$.
\end{lemma}

We can now prove a polynomial in time decay of the difference of two kinetic solutions of \eqref{SPDE-0} with different initial conditions.
The result follows the same ideas as Proposition 6.1 of Dong, Zhang and Zhang \cite{DZZ23}, see also Theorem 3.8 of Dareiotis, Gess, and Tsatsoulis \cite{DGT20}.

\begin{proposition}[Convergence rates for the difference of two solutions of \eqref{SPDE-0}]\label{prop: polynomial decay of solutions}
   Suppose that $(\Phi,\sigma)$ satisfy either Assumption \ref{asm: classic-DK} or \ref{asm: assumptions on non-linear coeffs for contraction estime}, and that $(\Phi,\nu)$ further satisfy Assumptions \ref{asm: Assumptions on nu} and \ref{asm: assumption on curly A}.

    Let $\rho(t;\rho_{1,0})$ and $\rho(t;\rho_{2,0})$ denote stochastic kinetic solutions of equation \eqref{SPDE-0} with initial data $\rho_{1,0}, \rho_{2,0}\in L^2_\omega L^2_{x}$ respectively and the same boundary data.
    There exists a constant $c\in(0,\infty)$ depending only on $\alpha,d,U$ and $q_0$ (but notably independent of $t$) such that for every $t\geq0$,
    \begin{equation}\label{eq: polynomial rate of decay for two solutions of DK}
        \sup_{\rho_{1,0},\rho_{2,0}\in L^2_\omega L^2_{x}}\mathbb{E}\|\rho(t;\rho_{1,0})-\rho(t;\rho_{2,0})\|_{L^1_{w;x}}\leq c\|w\|_{L^{q^*}_x}^{q^*}t^{-1/q_0},
    \end{equation}
    where $q^*:=\frac{q_0+1}{q_0}$.
    
     In the case that either $\Phi$ or $\nu$ additionally satisfy a lower Lipschitz bound\footnote{A lower Lipschitz bound for function $f$ is a bound of the form $c|\xi_1-\xi_2|\leq |f(\xi_1)-f(\xi_2)|$ for some constant $c\in(0,\infty)$ uniformly for every $\xi_1,\xi_2$ in the domain of $f$.}, which in particular captures the classical Dean--Kawasaki case \eqref{SPDE-1}, we can upgrade the polynomial rate of decay \eqref{eq: polynomial rate of decay for two solutions of DK} to the exponential rate
    \begin{align}\label{eq: exponential contraction for classical case of lipschitz nu}
\mathbb{E}\|\rho(t;\rho_{1,0})-\rho(t,\rho_{2,0})\|_{L^1_{w;x}}\leq&\mathbb{E}\|\rho_{1,0}-\rho_{2,0}\|_{L^1_{w;x}}\exp(-c\|w\|_{L^{\infty}_x}^{-1}t). 
\end{align}
In particular this applies to the classical Dean--Kawasaki equation, equation \eqref{SPDE-1}.
\end{proposition}
\begin{proof}
    Theorem \ref{thm: L1 omega super contraction} alongside Assumption \ref{asm: assumption on curly A} and the time continuity of solutions (Theorem \ref{thm: standard well posedness thm from Popat25} above) give for every $0\leq s\leq t$ and a running constant $c\in(0,\infty)$
    \begin{multline}\label{eq: combining the super contraction and assumption on mathcal A}
         \mathbb{E}\|\rho(t;\rho_{1,0})-\rho(t;\rho_{2,0})\|_{L^1_{w;x}}\\
         \leq \mathbb{E}\|\rho(s;\rho_{1,0})-\rho(s;\rho_{2,0})\|_{L^1_{w;x}}-c \mathbb{E}\int_s^t\int_U \mathcal{A}\left(\rho(r;\rho_{1,0}),\rho(r;\rho_{2,0})\right)\,dy\,dr\\
         \leq \mathbb{E}\|\rho(s;\rho_{1,0})-\rho(s;\rho_{2,0})\|_{L^1_{w;x}}-c \mathbb{E}\int_s^t\int_U |\rho(r;\rho_{1,0})-\rho(r;\rho_{2,0})|^{q_0+1}\,dy\,dr\\
         =\mathbb{E}\|\rho(s;\rho_{1,0})-\rho(s;\rho_{2,0})\|_{L^1_{w;x}}-c \mathbb{E}\int_s^t\|\rho(r;\rho_{1,0})-\rho(r;\rho_{2,0})\|^{q_0+1}_{L^{q_0+1}_x}\,dr.
    \end{multline}
    Furthermore, by using H\"older's inequality, we have that
    \begin{equation*}
        \|\rho(r;\rho_{1,0})-\rho(r;\rho_{2,0})\|_{L^1_{w;x}}\leq \|\rho(r;\rho_{1,0})-\rho(r;\rho_{2,0})\|_{L^{q_0+1}_x}\|w\|_{L^{q^*}_x},
    \end{equation*}
    with $q^*:=\frac{q_0+1}{q_0}$ the H\"older conjugate of $q_0+1$.
    Taking expectation and applying Jensen's inequality $\mathbb{E}[X]\leq \mathbb{E}[X^{q_0+1}]^{\frac{1}{q_0+1}}$, for any suitable random variables $X$, gives after re-arranging
    \begin{align*}
         \|w\|_{L^{q^*}_x}^{-(q_0+1)}\left(\mathbb{E}\|\rho(r;\rho_{1,0})-\rho(r;\rho_{2,0})\|_{L^1_{w;x}}\right)^{q_0+1}\leq \mathbb{E}\|\rho(r;\rho_{1,0})-\rho(r;\rho_{2,0})\|^{q_0+1}_{L^{q_0+1}_x}.
    \end{align*}
Substituting this into the final term of \eqref{eq: combining the super contraction and assumption on mathcal A} gives
\begin{multline}\label{eq: equation to set up use of technical lemma}
         \mathbb{E}\|\rho(t;\rho_{1,0})-\rho(t;\rho_{2,0})\|_{L^1_{w;x}}\leq\mathbb{E}\|\rho(s;\rho_{1,0})-\rho(s;\rho_{2,0})\|_{L^1_{w;x}}\\
         -c \|w\|_{L^{q^*}_x}^{-(q_0+1)}\int_s^t\left(\mathbb{E}\|\rho(r;\rho_{1,0})-\rho(r;\rho_{2,0})\|_{L^1_{w;x}}\right)^{q_0+1}\,dr.
\end{multline}
We now wish to apply Lemma \ref{lemma: lemma B.2 from DGT}.
To this end, define $f(t):=\mathbb{E}\|\rho(t;\rho_{1,0})-\rho(t;\rho_{2,0})\|_{L^1_{w;x}}$.
Equation \eqref{eq: equation to set up use of technical lemma} can be written for $s\leq t$ as the integral inequality
\begin{equation*}
    f(t)-f(s)\leq -c \|w\|_{L^{q^*}_x}^{-(q_0+1)}\int_s^t f(r)^{q_0+1}\,dr,
\end{equation*}
with initial condition $f(0)=\mathbb{E}\|\rho_{1,0}-\rho_{2,0}\|_{L^1_{w;x}}$.
By Lemma \ref{lemma: lemma B.2 from DGT}, $f(t)\leq h(t)$ for every $t>0$, where $h$ satisfies
\begin{equation}
    \begin{cases}\label{eq: PDE h}
        h'(t)=-c \|w\|_{L^{q^*}_x}^{-(q_0+1)}h(t)^{q_0+1},\\
        h(0)=f(0).
    \end{cases}
\end{equation}
By applying separation of variables, the solution of \eqref{eq: PDE h} is given by 
\begin{equation*}
    h(t)=\left(\frac{1}{h(0)^{-q_0}+c \|w\|_{L^{q^*}_x}^{-(q_0+1)}q_0t}\right)^{1/q_0}.
\end{equation*}
Consequently for every $t>0$ we have
\begin{multline*}
    \mathbb{E}\|\rho(t;\rho_{1,0})-\rho(t;\rho_{2,0})\|_{L^1_{w;x}}\leq \left(\frac{1}{h(0)^{-q_0}+c \|w\|_{L^{q^*}_x}^{-(q_0+1)}q_0t}\right)^{1/q_0}\\
    \leq\left(\frac{1}{c \|w\|_{L^{q^*}_x}^{-(q_0+1)}q_0t}\right)^{1/q_0}=c\|w\|_{L^{q^*}_x}^{q^*}t^{-1/q_0}.
\end{multline*}

In the following, suppose that $(\Phi,\sigma)$ satisfy Assumptions~\ref{asm: classic-DK} or that $\nu$ is Lipschitz. In either of these cases, following the computation of Theorem \ref{thm: L1 omega super contraction} gives for every $0\leq s\leq t$,
    \begin{multline*}
\mathbb{E}\|\rho(t;\rho_{1,0})-\rho(t,\rho_{2,0})\|_{L^1_{w;x}}\\
\leq\mathbb{E}\|\rho(s;\rho_{1,0})-\rho(s,\rho_{2,0})\|_{L^1_{w;x}}-c\mathbb{E}\int_s^t\int_U|\rho(r;\rho_{1,0})-\rho(r,\rho_{2,0})|\,dx\,dr .
    \end{multline*}
We removed the negative contribution of the higher-order power term on the right hand side.
Here we directly use the boundedness of $w$ to derive
\begin{align*}
\|w\|_{L^{\infty}_x}^{-1}\|\rho(r;\rho_{1,0})-\rho(r,\rho_{2,0})\|_{L^1_{w;x}}\leq \|\rho(r;\rho_{1,0})-\rho(r,\rho_{2,0})\|_{L^1_x},
\end{align*}
This gives
 \begin{align*}
\mathbb{E}\|\rho(t;\rho_{1,0})-\rho(t,\rho_{2,0})\|_{L^1_{w;x}}\leq&\mathbb{E}\|\rho(s;\rho_{1,0})-\rho(s,\rho_{2,0})\|_{L^1_{w;x}}\\
&-c\|w\|_{L^{\infty}_x}^{-1}\int^t_s\mathbb{E}\|\rho(r;\rho_{1,0})-\rho(r,\rho_{2,0})\|_{L^1_{w;x}}\,dr.
\end{align*}
Applying Gronwall's inequality gives the result of the exponential decay rate.

\end{proof}

Using an essentially identical argument, together with Corollary~\ref{coro: L1 omega super contraction}, if the contribution coming from the Stratonovich-to-It\^o terms satisfies the lower Lipschitz property, we can upgrade the polynomial convergence rate to an exponential one.
This in particular applies to the porous medium case for $m=2$.
\begin{corollary}[regularization by noise]\label{prp-improvedrate}
Suppose that $(\Phi,\sigma,\nu)$ satisfy Assumptions~\ref{asm: assumptions on non-linear coeffs for contraction estime}, \ref{asm: Assumptions on nu}, and~\ref{asm: assumption on curly A}.
Suppose further that $\Psi_\sigma$ as defined in \eqref{eq: def of Psi sigma} is well-defined and satisfies the lower Lipschitz property \eqref{Lip-Psisigma}.

With the same notation as Proposition \ref{prop: polynomial decay of solutions}, we have
\begin{align*}
\mathbb{E}\|\rho(t;\rho_{1,0})-\rho(t,\rho_{2,0})\|_{L^1_{w;x}}\leq&\mathbb{E}\|\rho_{1,0}-\rho_{2,0}\|_{L^1_{w;x}}\exp(-C\|w\|_{L^{\infty}_x}^{-1}t). 
\end{align*}
\end{corollary}

Now that we have established polynomial and exponential decays of norms of solutions of \eqref{SPDE-0} with different initial data, we would like to establish rates of convergence of the laws of our solutions towards the invariant measure.

For stochastic kinetic solutions $\rho(t;\rho_0)$ of equation \eqref{SPDE-0} we begin by defining the associated Markovian semigroup.
\begin{definition}[Semigroup]\label{def: definition of semigroup}
    For any $t\geq 0$ and $B_b(L^1_{w;x})$ defined in Section \ref{subsec: notation}, define $P_t:B_b(L^1_{w;x})\to B_b(L^1_{w;x})$ by
    \begin{equation}\label{eq: equation for semigroup}
        P_t F(\rho_0):=\mathbb{E} \left[F(\rho(t;\rho_0))\right],\quad F\in B_b(L^1_{w;x}), \,\rho_0\in L^1_{w;x}.
    \end{equation}
\end{definition}

In order to identify the invariant measure in Theorem \ref{thm: existence and uniqueness of invariant measure} we need to extend the time horizon of our SPDE \eqref{SPDE-0} to $-\infty$.
Informally, we will see that the invariant measure is the law of the process at time zero when the process is started at time $-\infty$.

For $-\infty<s\leq T$ we introduce the notation $\rho_s(t;\rho_0)$ to denote the stochastic kinetic solution of equation \eqref{SPDE-0} at time $t$ with initial condition $\rho_0$ at time $s$.
That is, the solution of
\begin{align}\label{eq: SPDE extended to negative times}
\begin{cases}
    \partial_t\rho_s(t;\rho_0)=\Delta\Phi(\rho_s(t;\rho_0))-\nabla\cdot(\sigma(\rho_s(t;\rho_0))\circ\dot{\xi}^F+\nu(\rho_s(t;\rho_0))),& (x,t)\in U\times(s,T],\\
	\rho_s(s;\rho_0)=\rho_0,& (x,t)\in U\times\{t=s\},\\
	\Phi(\rho_s(t;\rho_0))=\rho_b,& (x,t)\in\partial U\times(s,T]. 
\end{cases}
\end{align}
We extend the Brownian motions in the definition of $\xi^F$ to negative times by gluing at $t=0$ an independent Brownian motion evolving backwards in time.
For consistency we also have $\rho_0(t;\rho_0)=\rho(t;\rho_0)$.
The global well-posedness of \eqref{eq: SPDE extended to negative times} follows precisely the same arguments as the case $s=0$ from Theorem \ref{thm: standard well posedness thm from Popat25}. We state this as a corollary without proof.
\begin{corollary}[Well-posedness and contraction estimate for equation \eqref{eq: SPDE extended to negative times}]\label{corrolary: well posedness and contraction for DK on arbitrary time interval}
    For every $-\infty<s<T$ we can define stochastic kinetic solutions $\rho_s(\cdot;\rho_0)$ of the Dean--Kawasaki equation on the interval $[s,T]$.
    The kinetic measure $q_s$ can be defined on $U\times[0,\infty]\times[s,T]$, and the kinetic equation \eqref{eq: decay of kinetic measure at infinity} is amended so that the initial condition is at time $s$, and time integrals on the right hand side range from $[s,T]$ rather than $[0,T]$.

    Well-posedness of stochastic kinetic solutions of $\rho_s$ follows precisely the same arguments as Theorem \ref{thm: standard well posedness thm from Popat25}, see the well-posedness results in the work of the first author \cite{Shyam25}.
    Furthermore, for two stochastic kinetic solutions $\rho_{s,1},\rho_{s,2}$ with initial conditions $\rho_{s,1,0}, \rho_{s,2,0}$ at time $s$ respectively, we have the almost-sure contraction estimate
    \begin{equation*}
    \sup_{t\in[s,T]}\|\rho_{s,1}(t;\rho_{s,1,0})-\rho_{s,2}(t;\rho_{s,2,0})\|_{L^1_x}\leq \|\rho_{s,1,0}-\rho_{s,2,0}\|_{L^1_x}.
\end{equation*}
\end{corollary}

We want to prove that the semigroup corresponding to $\rho(t;\rho_0)$ as defined in Definition \ref{def: definition of semigroup} is Feller.
In order to do this, we first show the \say{flow property} for kinetic solutions of \eqref{eq: SPDE extended to negative times}.
In words, the flow property states that the random solution at time $t$ obtained by starting from $\rho_0$ at time $s_1$ is the same (in the weighted $L^1_x$-space in expectation) as if you first flow forward to time $s_2$ and restart the dynamics at $s_2$ with that intermediate state, and flow to time $t$.
The statement \eqref{eq: flow property} compares the laws of the solutions defined on different time domains ($[s_1,T]$ and $[s_2,T]$), but both are evaluated at a common future time $t$.
We note the similarity of this property with the Chapman-Kolmogorov equation from Markov processes.

\begin{proposition}[Flow property]\label{prop: flow property}
        For every $\rho_0\in L^2_x$ and $-\infty\leq s_1\leq s_2\leq t \leq T$, it holds that 
    \begin{equation}\label{eq: flow property}
        \rho_{s_1}(t;\rho_0)=\rho_{s_2}(t;\rho_{s_1}(s_2;\rho_0)),\quad \text{in } L^1_\omega L^1_{w;x}.
    \end{equation}
\end{proposition}
\begin{proof}

Due to the time-homogeneous coefficients in equation \eqref{SPDE-0} it suffices to prove the flow property \eqref{eq: flow property} for $(s_1,s_2)=(0,s)$. 
    That is, we wish to prove that
    \begin{equation}\label{eq: equation for the flow property in the case 0,s}
        \rho(t;\rho_0)=\rho_{s}(t;\rho(s;\rho_0)).
    \end{equation}

For $0<s<t<T$, taking the kinetic equation \eqref{eq: kinetic equation} for $\rho(t;\rho_0)$ and subtracting it from the kinetic equation for $\rho(s;\rho_0)$ gives that, for every $\psi\in C^{\infty}_c(U\times\mathbb{R}_+)$, 
\begin{align*}\
&\int_{\mathbb{R}}\int_{U}\chi(x,\xi,t)\psi(x,\xi)\,dx\,d\xi=\int_{\mathbb{R}}\int_{U}\chi(x,\xi,s)\psi(x,\xi)\,dx\,d\xi \nonumber\\
&-\int_s^t \int_{U}\left(\Phi'(\rho)\nabla\rho+\frac{1}{2}F_1[\sigma'(\rho)]^2\nabla\rho +\frac{1}{2} \sigma'(\rho)\sigma(\rho)F_2\right)\cdot\nabla\psi(x,\xi)|_{\xi=\rho}\,dx\,ds\nonumber\\
    &-\int_{\mathbb{R}}\int_s^t\int_{U}\partial_\xi\psi(x,\xi)\,dq +\frac{1}{2}\int_s^t \int_{U}\left(
    \sigma'(\rho)\sigma(\rho)\nabla\rho\cdot F_2 + \sigma(\rho)^2F_3 \right)\partial_\xi\psi(x,\rho)\,dx\,ds\nonumber\\
    & -\int_s^t \int_{U}\psi(x,\rho)\nabla\cdot (\sigma(\rho) \,d{\xi}^F)\,dx -        \int_s^t \int_{U}\psi(x,\rho)\nabla\cdot\nu(\rho)\,dx\,ds.
\end{align*}

By inspection, due to the uniqueness of stochastic kinetic solutions, it follows that $\rho(t;\rho_0)$ satisfies the kinetic equation on the subinterval $[s,t]$ started at the point $\rho(s;\rho_0)$.
The other regularity conditions of Definition \ref{def: stochastic kinetic solution of DK} hold by the fact that $\rho(t;\rho_0)$ was assumed to be a stochastic kinetic solution on the whole interval $[0,t]$.
By the uniqueness of solutions, Corollary \ref{corrolary: well posedness and contraction for DK on arbitrary time interval}, the flow property \eqref{eq: equation for the flow property in the case 0,s} holds.
\end{proof}

\begin{proposition}[Semigroup is Feller]\label{prop: semigroup is Feller}
Suppose that $(\Phi,\sigma)$ satisfy either Assumption \ref{asm: classic-DK} or \ref{asm: assumptions on non-linear coeffs for contraction estime}, and that $(\Phi,\nu)$ further satisfy Assumptions \ref{asm: Assumptions on nu} and \ref{asm: assumption on curly A}.

The family $(P_t)_{t\geq 0}$ is a Feller semigroup.
 That is, $P_t$ maps $C_b(L^1_{w;x})$ to $C_b(L^1_{w;x})$.
\end{proposition}
\begin{proof}
We prove the properties of a Feller semigroup one by one.
\begin{itemize}
    \item $P_0=Id$: This is trivial since $P_0 F(\rho_0):=\mathbb{E} \left[F(\rho(0;\rho_0))\right]=F(\rho_0)$.
\item Positive preserving: If $F\geq0$ then using the definition of $P_t$, it holds that $P_t F\geq0$.
Furthermore, it holds that $P_t1=1$.
\item $P_t$ is a contraction: This follows by standard properties of the expectation. 
For a given $F\in C_b(L^1_{w;x})$ and $\rho_0\in L^2_{x}$, we have
\begin{equation*}
|P_t F(\rho_0)|=|\mathbb{E} \left[F(\rho(t;\rho_0))\right]|\leq \mathbb{E} \left[|F(\rho(t;\rho_0))|\right] \leq \sup_{\rho\in L^1_{w;x}}|F(\rho)|,
\end{equation*}
where the right hand side is well-defined by the fact that $F$ is bounded.
Taking the supremum over $\rho_0$ in the above equation gives the contraction property.
\item $P_t$ is continuous on $C_b(L^1_{w;x})$: Since $F\in C_b(L^1_{w;x})$ and solutions to \eqref{eq: SPDE extended to negative times} are continuous in time, it holds that for every $\rho_0\in L^2_{x}$ we have $\lim_{t\to0}P_tF(\rho_0)=F(\rho_0)$.
\item $P_t$ satisfies the semigroup property $P_{t+s}=P_sP_t$:
    Without loss of generality consider $-\infty\leq s\leq t$. 
    Taking $F\in C_b(L^1_{w;x})$ and $\rho_0\in L^2_{x}$, applying the proof of Theorem 9.14 of Da Prato and Zabczyk \cite{da2014stochastic}, due to the time-homogeneity of the coefficients and the flow property, we have that 
    \begin{align*}
         P_{t+s} F(\rho_0)=\mathbb{E}\left[\mathbb{E} \left[F(\rho(t+s;\rho_0))|\mathcal{F}_s\right]\right]=\mathbb{E}\left[\mathbb{E}\left[F(\rho_s(t;\zeta))\right]|_{\zeta=\rho(s;\rho_0)}\right]=P_t(P_s F(\rho_0)). 
    \end{align*}
\end{itemize}
\end{proof}

\begin{remark}
    Although it is not necessary for the Feller property above, we mention that if we work over a bounded subset of the space of $L^1_{w;x}$, for instance $L^1_{w;x;M}:=\{\rho_0\in L^2_{x}: \|\rho_0\|_{L^1_{w;x}}\leq M \}$, we can strengthen the continuity property of the fourth point above to uniform continuity. 
    This is a consequence of the fact that
\begin{multline}\label{eq: strong continuity over bounded initial data}
    \sup_{\rho_0\in L^1_{w;x;M}}|P_tF(\rho_0)-F(\rho_0)|\\
    \leq \sup_{\rho_0\in L^1_{w;x;M}}|P_tF(\rho_0)-P_tF_n(\rho_0)|+\sup_{\rho_0\in L^1_{w;x;M}}|P_tF_n(\rho_0)-F_n(\rho_0)|+\sup_{\rho_0\in L^1_{w;x;M}}|F_n(\rho_0)-F(\rho_0)|\\
    \leq 2\sup_{\rho_0\in L^1_{w;x;M}}|F_n(\rho_0)-F(\rho_0)|+\sup_{\xi\in L^1_{w;x;M}}|D F_n(\xi)|\sup_{\rho_0\in L^1_{w;x;M}}\mathbb{E}\|\rho(t;\rho_0)-\rho_0\|_{L^1_{w;x}},
\end{multline}
where $(F_n)_{n\geq1}\subseteq C^1_b(L^1_{w;x})$ satisfies 
\begin{equation*}
    \sup_{\rho_0\in A}|F_n(\rho_0)-F(\rho_0)|\to0,\quad\text{as $n\to\infty$, for every bounded set $A\subseteq L^1_{w;x}$}.
\end{equation*}
The first term in equation \eqref{eq: strong continuity over bounded initial data} converges to zero as $n\to\infty$ due to the fact that $F_n\to F$, and the second term converges to zero as $t\to0$ due to the time continuity of solutions.
The strong continuity follows by taking the limits $n\to\infty$ followed by $t\to \infty$ in equation \eqref{eq: strong continuity over bounded initial data}.
\end{remark}
The below theorem not only gives us existence and uniqueness of invariant measures, but provides us with an exponential mixing rate uniformly with respect to the initial condition.
The proof is similar to Theorem 4.3.9 of Pr\'ev\^ot and R\"ockner \cite{prevot2007concise}, see also Theorem 3.6 of Dong, Zhang and Zhang \cite{DZZ23} and Theorem 3.13 of Dareiotis, Gess and Tsatsoulis \cite{DGT20}.
\begin{theorem}\label{thm: existence and uniqueness of invariant measure}
    We have the following dichotomy.
    
    (i) Suppose that $(\Phi,\sigma)$ satisfy either Assumption \ref{asm: assumptions on non-linear coeffs for contraction estime}, and that $(\Phi,\nu)$ further satisfy Assumptions \ref{asm: Assumptions on nu} and \ref{asm: assumption on curly A}. Then there exists a unique invariant measure $\mu\in \mathcal{M}_1(L^1_{w;x})$ for the semigroup $P_t$. 
        Furthermore, there exists a constant $c\in(0,\infty)$ depending only on $\alpha,d,U, w$ and $q_0$ (but notably independent of $t$) such that for every $t\geq 0$, we have
        \begin{equation}\label{eq: polynomial bound for mixing of invariant measure}
            \sup_{\rho_0\in L^2_{x}}\sup_{\|F\|_{Lip(L^1_{w;x})}\leq 1}\left|P_t F(\rho_0)-\int_{L^1_{w;x}}F(\xi)\mu(d\xi)\right| \leq ct^{-1/q_0}.
    \end{equation}
        The space $Lip(L^1_{w;x})$ denotes the space of functions from $L^1_{w;x}$ to $\mathbb{R}$ that are Lipschitz continuous.
        
        (ii) Suppose that $(\Phi,\sigma)$ satisfy either Assumption \ref{asm: classic-DK}, and that $(\Phi,\nu)$ further satisfy Assumptions \ref{asm: Assumptions on nu} and \ref{asm: assumption on curly A}. Then there exists a unique invariant measure $\mu\in \mathcal{M}_1(L^1_{w;x})$ for the semigroup $P_t$. Furthermore, there exist constants $c_1,c_2\in(0,\infty)$ depending only on $w$, $d$, and $U$ (but notably independent of $t$) such that, for every $\rho_0\in L^2_{x}$ and $t\ge0$,  
\begin{equation}\label{exponential-decay}
    \sup_{\|F\|_{\mathrm{Lip}(L^1_{w;x})}\le1}
    \left|P_t F(\rho_0)-\int_{L^1_{w;x}}F(\xi)\,\mu(d\xi)\right|
    \le c_1 e^{-c_2 t}. 
\end{equation}
\end{theorem}

\begin{proof}
    \textbf{Step 1: Defining and proving well-posedness of the invariant measure}

For $s<0$ we define $\eta_s(\rho_0):=\rho_s(0,\rho_0)$ to be the solution of equation \eqref{eq: SPDE extended to negative times} time $0$ initiated at $\rho_0$ at time $s$.
Using the flow property, Proposition \ref{prop: flow property}, it holds that for every $s_1\leq s_2\leq-1$ that
\begin{equation*}
    \eta_{s_1}(\rho_0)=\rho_{s_2}(0;\rho_{s_1}(s_2;\rho_0)),\quad \text{in }L^1_\omega L^1_{w;x}.
\end{equation*}
By the polynomial decay of solutions, Proposition \ref{prop: polynomial decay of solutions}, we have that 
\begin{equation*}
    \mathbb{E}\|\eta_{s_2}(\rho_0)-\eta_{s_1}(\rho_0)\|_{L^1_{w;x}}=\mathbb{E}\|\rho_{s_2}(0,\rho_0)-\rho_{s_1}(0,\rho_0)\|_{L^1_{w;x}}\lesssim \|w\|_{L^{q^*}_x}^{q^*}|s_2|^{-1/q_0}.
\end{equation*}
The last inequality tells us that $(\eta_s(\rho_0))_{s\leq-1}$ is a Cauchy sequence in the space $L^1_\omega L^1_{w;x}$. 
Consequently, it holds that $\eta_s(\rho_0)\to\eta(\rho_0)$ as $s\to-\infty$ for some random variable $\eta(\rho_0)\in L^1_\omega L^1_{w;x}$.

It holds that $\eta$ is independent of the initial condition $\rho_0$, which can be seen by again using the polynomial decay of solutions, Proposition \ref{prop: polynomial decay of solutions},
\begin{equation*}
    \mathbb{E}\|\eta_s(\rho_{1,0})-\eta_s(\rho_{2,0})\|_{L^1_{w;x}}=\mathbb{E}\|\rho_{s}(0,\rho_{1,0})-\rho_{s}(0,\rho_{2,0})\|_{L^1_{w;x}}\lesssim \|w\|_{L^{q^*}_x}^{q^*}|s|^{-1/q_0}
\end{equation*}
and sending $s\to-\infty$.

We now define
\begin{equation*}
    \mu:=\mathcal{L}(\eta)\in \mathcal{M}_1(L^1_{w;x})
\end{equation*}
where we can choose arbitrary initial condition, for instance $\eta=\eta(0)$.

For $s\leq t$, denote by $P_{s,t}$ the semigroup associated to equation \eqref{eq: SPDE extended to negative times} at time $t$, defined in a similar way as Definition \ref{def: definition of semigroup} by $ P_{s,t} F(\rho_0):=\mathbb{E} \left[F(\rho_s(t;\rho_0))\right]$.
In the notation of Definition \ref{def: definition of semigroup}, $P_t=P_{0,t}$.
By analogous reasoning as Proposition \ref{prop: semigroup is Feller}, $P_{s,t}$ is Feller for every $s<t$ and further satisfies the identities \begin{equation*}
    \begin{cases}
        P_{s,t}=P_{s+\tau,t+\tau},\quad \tau\in\mathbb{R},\\
        P_{s,\tau}P_{\tau,t}=P_{s,t},\quad s<\tau<t.
    \end{cases}
\end{equation*}
We further note that the first property follows directly from the 
homogeneity of the coefficients in \eqref{SPDE-0} together with the 
stationarity of Brownian increments. In particular, by the identity in law
$$
  \xi^{F}(t+\tau)-\xi^{F}(\tau)\overset{d}{=}\xi^{F}(t),
$$
the desired conclusion follows.

Using these properties and the definition of $\mu$ gives for every $F\in C_b(L^1_{w;x})$
\begin{equation*}
    \int_{L^1_{w;x}}P_{0,t}F(\tilde\rho)\mu(d\tilde\rho)=\lim_{s\to\infty} P_{-s,0}(P_{0,t}F)(0)=\lim_{s\to\infty}P_{-s,t}F(0)=P_{-(t+s),0}F(0)=\int_{L^1_{w;x}}F(\tilde\rho)\mu(d\tilde\rho).
\end{equation*}
This implies that $P_t^*\mu=\mu$ as measures on $L^1_{w;x}$, i.e. $\mu$ is an invariant measure for the semigroup $P_t$.

\textbf{Step 2: The bound \eqref{eq: polynomial bound for mixing of invariant measure}}

By the polynomial decay of Proposition \ref{prop: polynomial decay of solutions}, it holds that for $F\in Lip(L^1_{w;x})$ and $\rho_{1,0},\rho_{2,0}\in L^2_{x}$
\begin{equation*}
    \left|P_tF(\rho_{1,0})-P_tF(\rho_{2,0})\right|=\left|\mathbb{E} \left[F(\rho(t;\rho_{1,0}))-F(\rho(t;\rho_{2,0}))\right]\right|\lesssim \|F\|_{Lip(L^1_{w;x})}\|w\|_{L^{q^*}_x}^{q^*}|t|^{-1/q_0}.
\end{equation*}
This implies that any two invariant measures $\mu$ and $\tilde\mu$ on $L^1_{w;x}$ coincide. 
This also gives us that 
\begin{equation*}
    \left|P_tF(\rho_{1,0})-\int_{L^1_{w,x}}F(\tilde\rho)\mu(d\tilde\rho)\right|\lesssim \|F\|_{Lip(L^1_{w;x})}\|w\|_{L^{q^*}_x}^{q^*}|t|^{-1/q_0},
\end{equation*}
which proves the first claim \eqref{eq: polynomial bound for mixing of invariant measure} after taking the supremum over $\|F\|_{Lip(L^1_{w;x})}\leq1$ and $\xi\in L^1_{w;x}$.

For the second claim, we consider the case that either Assumption~\ref{asm: classic-DK} is satisfied or that $\nu$ is Lipschitz.
By following precisely the same steps as above but using the exponential decay rate \eqref{eq: exponential contraction for classical case of lipschitz nu} rather than the polynomial one, we get the exponential decay \eqref{exponential-decay}. 
\end{proof}

Using an essentially identical approach, combining with Corollary \ref{coro: L1 omega super contraction}, we have the following regularization by noise phenomenon for the case of porous medium Dean--Kawasaki equation. 
\begin{corollary}[Regularization by noise]\label{coro: L1 omega super contraction-ergodic}
Suppose that $(\Phi,\sigma,\nu)$ satisfy Assumptions~\ref{asm: assumptions on non-linear coeffs for contraction estime}, \ref{asm: Assumptions on nu}, and~\ref{asm: assumption on curly A}. 
Suppose further that $\Psi_\sigma$ as defined in \eqref{eq: def of Psi sigma} is well-defined and satisfies the lower Lipschitz property \eqref{Lip-Psisigma}.
Let $(P_t)_{t\ge0}$ denote the semigroup generated by \eqref{SPDE-0}, and let $\mu$ be its invariant measure. Then there exist constants $c_1,c_2\in(0,\infty)$ depending only on $w$, $d$, and $U$ (but notably independent of $t$) such that, for every $\rho_0\in L^2_{x}$ and $t\ge0$,  
\begin{equation*}
 \sup_{\|F\|_{\mathrm{Lip}(L^1_{w;x})}\le1}
    \left|P_t F(\rho_0)-\int_{L^1_{w;x}}F(\xi)\,\mu(d\xi)\right|
    \le c_1 e^{-c_2 t}. 
\end{equation*}
\end{corollary}

We conclude by transferring the convergence rate results of the semigroup established in Theorem~\ref{thm: existence and uniqueness of invariant measure} and Corollary~\ref{coro: L1 omega super contraction-ergodic} to the corresponding convergence of the law of the stochastic process \eqref{SPDE-0} in the 1-Wasserstein distance.
We further make the observation that the invariant measure can be considered on $L^1_x$ rather than the weighted space $L^1_{w,x}$ by the equivalence of the two norms.
\begin{corollary}
    We have the following dichotomy. 

   (i) Suppose that $(\Phi,\sigma)$ satisfy either Assumption  \ref{asm: assumptions on non-linear coeffs for contraction estime}, and that $(\Phi,\nu)$ further satisfy Assumptions \ref{asm: Assumptions on nu} and \ref{asm: assumption on curly A}. Then there exists a unique invariant measure $\mu\in \mathcal{M}_1(L^1_x)$ for the semigroup $P_t$. Furthermore, there exists a constant $c\in(0,\infty)$ depending only on $w,d,U$ and $q_0$ (but notably independent of $t$) such that for every $t>0$,
    \begin{equation}\label{eq: wasserstein distance of invariant measure and law of process}
        \sup_{\rho_0\in L^2_x}\mathcal{W}_1\left(\mathcal{D}_{\mathbb{P}}\left(\rho(t;\rho_0)\right), \mu\right)\leq  ct^{-1/q_0},
    \end{equation}
    where $\mathcal{W}_1$ denotes the 1-Wasserstein distance and $\mathcal{D}_{\mathbb{P}}\left(\rho(t;\rho_0)\right):=\mathbb{P}\circ \rho(t;\rho_0)^{-1}$ denotes the law of $\rho(t;\rho_0)$ under $\mathbb{P}$.
    
     (ii) Suppose that $(\Phi,\sigma)$ satisfy either Assumption \ref{asm: assumptions on non-linear coeffs for contraction estime}, and that $(\Phi,\nu)$ further satisfy Assumptions \ref{asm: Assumptions on nu} and \ref{asm: assumption on curly A}. Then there exists a unique invariant measure $\mu\in \mathcal{M}_1(L^1_x)$ for the semigroup $P_t$. Furthermore, there exist constants $c_1,c_2\in(0,\infty)$ depending only on $w$, $d$, and $U$ (but notably independent of $t$) such that, for every $\rho_0\in L^2_x$ and $t\ge0$,  
\begin{equation}\label{exponential-decay-2}
\mathcal{W}_1\left(\mathcal{D}_{\mathbb{P}}\left(\rho(t;\rho_0)\right), \mu\right)\leq c_1 e^{-c_2 t}.
\end{equation}
\end{corollary}
\begin{proof}
Since the weighted norm $L^1_{w;x}$ is equivalent to the $L^1_x$-norm, it follows that we can find an invariant measure on the space $L^1_x$. 

    Estimates \eqref{eq: wasserstein distance of invariant measure and law of process} and \eqref{exponential-decay-2} follow directly from Theorem~\ref{thm: existence and uniqueness of invariant measure} together with the Kantorovich--Rubinstein formula, see Villani~\cite[Theorem~5.10]{villani2008optimal}. 
\end{proof}

We have the same improved result for the regularization by noise case.
\begin{corollary}
    Finally, in the case that $\Psi_\sigma$ as defined in \eqref{eq: def of Psi sigma} is well-defined and satisfies the lower Lipschitz property \eqref{Lip-Psisigma}, then there exist constants $c_1,c_2\in(0,\infty)$ depending only on $w$, $d$, and $U$ (but notably independent of $t$) such that, for every $\rho_0\in L^2_x$ and $t\ge0$,  
\begin{equation}\label{exponential-decay-3}
\mathcal{W}_1\left(\mathcal{D}_{\mathbb{P}}\left(\rho(t;\rho_0)\right), \mu\right)\leq c_1 e^{-c_2 t}. 
\end{equation}
\end{corollary}
\begin{proof}
    This estimate follows directly by combining Corollary~\ref{coro: L1 omega super contraction-ergodic} with the Kantorovich--Rubinstein formula, see Villani~\cite[Theorem~5.10]{villani2008optimal}. .
\end{proof}

\textbf{Acknowledgements}
The first author was supported by the EPSRC Centre for Doctoral Training in Mathematics of Random Systems:
Analysis, Modelling and Simulation (EP/S023925/1), and now acknowledges funding from the SDAIM project ANR-22-CE40-0015 funded by the French National Research Agency (ANR).
The second author is supported by the US Army Research Office, grant W911NF2310230. 

\bibliography{references}       
\bibliographystyle{alpha}  
\end{document}